\documentclass[11pt,oneside]{amsart}

      \usepackage{amssymb}
      \usepackage{graphicx}
      \usepackage{dcolumn,indentfirst,cite,color,hyperref}
      \usepackage{anysize}  
\marginsize{3.5cm}{3.5cm}{3.5cm}{3.5cm}

      \theoremstyle{plain}
     \newtheorem{thm}{Theorem}[section]
\newtheorem{lem}[thm]{Lemma}

\newtheorem{cor}[thm]{Corollary}
\newtheorem{pro}[thm]{Proposition}
\newtheorem{rem}[thm]{Remark}

\newtheorem{defi}[thm]{Definition}
\newcommand{\bess}{\begin{eqnarray*}}
\newcommand{\eess}{\end{eqnarray*}}
\input xy
\xyoption{all}

      \makeatletter
      \def\@setcopyright{}
      \def\serieslogo@{}
      \makeatother

\usepackage{fancyhdr} 
\begin{document}




\author{Weiyuan Qiu}
\address{School of Mathematical Sciences, Fudan University, Shanghai, 200433, P.R.China}
\email{wyqiu@fudan.edu.cn}

\author{Xiaoguang Wang}
\address{School of Mathematical Sciences, Fudan University, Shanghai, 200433, P.R.China}
\email{wxg688@163.com}

 \author{Yongcheng Yin}
 \address{Department of Mathematics, Zhejiang University, Hangzhou,310027, P.R.China}
 \email{yin@zju.edu.cn}






   \title[]{Dynamics of McMullen maps}


   \begin{abstract}
In this article, we develop the Yoccoz puzzle technique to study a
family of rational maps termed McMullen maps. We show that the
boundary of the immediate basin of infinity is always a Jordan curve
if it is connected. This gives a positive answer to a question of
Devaney. Higher regularity of this boundary is obtained in almost
all cases. We show that the boundary is a quasi-circle if it
contains neither a parabolic point nor a recurrent critical point.
For the whole Julia set, we show that the McMullen maps have locally
connected Julia sets except in some special cases.
 \end{abstract}





   \date{\today}


   \maketitle




\section{Introduction}

The local connectivity of Julia sets for rational maps is a central
problem in complex dynamical systems. It is well-studied for
classical types of rational maps, such as hyperbolic and
semi-hyperbolic maps and geometrically finite maps
\cite{CJY},\cite{M1},\cite{TY}. The polynomial case is also
well-known \cite{DH1},\cite{GS},\cite{K},\cite{L},\cite{M2}, \cite{P}. For
quadratic polynomials, Yoccoz proved that the Julia set is locally
connected provided all periodic points are repelling and the map is
not infinitely renormalizable \cite{H},\cite{M2}. Douady exhibited a
striking example of an infinitely renormalizable quadratic
polynomial with a non-locally connected Julia set \cite{M2}.
  For a general polynomial with connected Julia sets and without irrationally neutral cycles, Kiwi shows in \cite{K} that the local connectivity of the Julia set is equivalent to the non-existence of wandering continua.

   A powerful tool for studying the local connectivity of Julia
sets for polynomials is the so-called `Branner-Hubbard-Yoccoz
puzzle' technique introduced by Branner-Hubbard and Yoccoz
\cite{BH}. This technique uses a natural method of construction
involving finitely many  periodic external rays together with an
equipotential curve. However, for general rational maps, the
situation is different, and the construction of the Yoccoz puzzle
becomes quite involved, even impossible. Until now, the only known
rational maps that admit Yoccoz puzzle structures were cubic Newton
maps, whose Yoccoz puzzles were constructed by Roesch. In \cite{R},
Roesch applied Yoccoz puzzle techniques to show striking differences
between rational maps and polynomials. The method also leads to the
local connectivity of Julia sets except in some specific cases.

   In this article, we present the Yoccoz puzzle structure for another
family of rational maps known as McMullen maps. These maps are of
the form
$$f_\lambda:z\mapsto z^n+\lambda/z^n, \ \ \lambda\in \mathbb{C}^*=\mathbb{C}\setminus \{0\},\ n\geq3.$$ The dynamics of this family of maps have been studied by Devaney and his group \cite{D0},\cite{D1},\cite{DK},\cite{DLU}.

The Yoccoz puzzle differs for cubic Newton maps and McMullen maps in
the following way. For cubic Newton maps, the Yoccoz puzzle is
induced by a periodic Jordan curve that intersects the Julia set at
countably many points. However, for McMullen maps, the element used
to construct the Yoccoz puzzle is a periodic Jordan curve (this
curve will be called the `cut ray') that intersects the Julia set in
a Cantor set. This type of Jordan curve is induced by some
particular angle and can be viewed as an extension of the
corresponding external ray (see Section 3.2).

We denote by $B_\lambda$ the immediate basin of attraction of
$\infty$. The topology of $\partial B_\lambda$ is of special
interest. Based on Yoccoz puzzle techniques and on combinatorial and
topological analysis, we prove:

 \begin{thm}\label{11a}
 For any $n\geq3$ and any complex parameter $\lambda$, if the Julia set $J(f_\lambda)$ is not a Cantor set, then
 $\partial B_\lambda$ is a Jordan curve.
 \end{thm}

This affirmatively answers a question posed by Devaney at the
Snowbird Conference on the 25th Birthday of the Mandelbrot set
\cite{DK}. For higher regularity of $\partial B_\lambda$, we show
that $\partial B_\lambda$ is a quasi-circle except in two special
cases.

 \begin{thm}\label{11b}
 Suppose that $J(f_\lambda)$ is not a Cantor set; then
 $\partial B_\lambda$ is a quasi-circle if it contains neither
 a parabolic point nor the recurrent critical set $C_\lambda:=\{c; c^{2n}=\lambda\}$.
 \end{thm}
 
Here, the critical set $C_\lambda$ is called {\it recurrent}, if $C_\lambda\subset J(f_\lambda)$ and the set $\cup_{k\geq1}f^k_\lambda(C_\lambda)$ has an accumulation  point in $C_\lambda$. It
follows from Proposition \ref{7i} that if $\partial B_\lambda$
contains a parabolic point, then $\partial B_\lambda$  is not a
quasi-circle by the Leau-Fatou-Flower Theorem \cite{M2}. Whether
$\partial B_\lambda$ is a quasi-circle when $\partial B_\lambda$
contains the recurrent critical set $C_\lambda$ is still unknown.

For the topology of the Julia set, we show

 \begin{thm}\label{11c}
  Suppose $f_\lambda$ has no Siegel disk and the Julia set $J(f_\lambda)$ is connected, then $J(f_\lambda)$ is locally connected in the following cases:

1. The critical orbit does not accumulate on the boundary $\partial
B_\lambda$.

2. $f_\lambda$ is neither renormalizable nor $*-$renormalizable.

3. The parameter $\lambda$ is real and positive.

 \end{thm}

See Section 5 for the definitions of renormalization and
$*-$renormalization. Theorem \ref{11c} implies that the Julia set is
locally connected except in some special cases. In fact, the theorem
is stronger than the following statement:

\begin{thm}\label{11d}
Suppose $f_\lambda$ has no Siegel disk and the Julia set
$J(f_\lambda)$ is connected, then  $J(f_\lambda)$ is locally
connected if the critical orbit does not accumulate on the boundary
$\partial B_\lambda$.
\end{thm}

Theorem \ref{11d} is an analogue of Roesch's Theorem \cite{R}:

\begin{thm}\label{11e}$\mathrm{(Roesch)}$
A genuine cubic Newton map without Siegel disks has a locally
connected Julia set provided the orbit of the non-fixed critical
point does not accumulate on the boundary of any invariant basin of
attraction.
 \end{thm}

We exclude the case $n=2$ because it is impossible to find a
non-degenerate critical annulus for the Yoccoz puzzle constructed in
this paper. The existence of a non-degenerate critical annulus is
technically necessary in our proof.

The paper is organized as follows:

In Section 2, we present some basic results on McMullen maps.

In Section 3, we construct`cut rays', each of which is a type of
Jordan curve that divides the Julia set into two different parts. We
first construct a Cantor set of angles on the unit circle which is
used to generate`cut rays'. We then discuss the construction of `cut
rays' based on the work of Devaney \cite{D1}.

In Section 4, basic knowledge of Yoccoz puzzles, graphs and tableaux
are presented. The aim of this section is to find a Yoccoz puzzle
with a non-degenerate critical annulus (see Section 4.2). A natural
construction of the `modified puzzle piece' is discussed (See
Section 4.3).

In Section 5, we discuss the renormalizations of McMullen maps in
the context of the puzzle piece.

In Section 6, we present a criterion of local connectivity. We
introduce a `\textbf{BD} condition' on the boundary of the immediate
basin of attraction. Such a condition can be considered as `local
semi-hyperbolicity'. We show that existence of the `\textbf{BD}
condition' implies good topology.

In Section 7, we study the local connectivity of $\partial
B_\lambda$ in all possible cases and show that $\partial B_\lambda$
enjoys higher regularity except in two special cases.

In Section 8, we study the local connectivity of the Julia set
$J(f_\lambda)$  based on  the `Characterization of Local
Connectivity' and the `Shrinking Lemma'.

\

\textbf{Acknowledgement} The authors would like to thank the
referees for their careful reading of the manuscript and their
helpful comments. This research was supported by the National
Natural Science Foundation of China (Grants No.10831004, 10871047)
and by the Science and Technology Commission of Shanghai
Municipality (NSF Grant No.10ZR1403700).

\section{Preliminaries and Notations}
In this section, we present some basic results and notations for the
family of rational maps $$f_\lambda(z)=z^n+\lambda/z^n$$ where
$\lambda\in\mathbb{C}^*$ and $n\geq3$. This type of map is known as
a `McMullen map' because it was first studied by McMullen, who
proved that when $|\lambda|$ is sufficiently small the Julia
set of $z\mapsto z^2+\lambda z^{-3}$ is a Cantor set of circles \cite{McM1}.

  For any $\lambda\in\mathbb{C}^*$, the map $f_\lambda$ has a superattracting fixed point at $\infty$. The immediate basin of $\infty$ is denoted by $B_\lambda$, and the component of $f_{\lambda}^{-1}(B_\lambda)$ that contains $0$ is denoted by $T_\lambda$. The set of all critical points of $f_\lambda$ is $\{0,\infty\} \cup C_\lambda$, where $C_\lambda=\{c ; c^{2n}=\lambda\}$. Besides
$\infty$, there are only two critical values for $f_\lambda$:
$v_\lambda^+=2\sqrt{\lambda}$ and $v_\lambda^-=-2\sqrt{\lambda}$. In
fact, there is only one critical orbit (up to a sign). Let
$P(f_\lambda)=\overline{\bigcup_{n\geq1}f_\lambda^k(C_\lambda)}\cup
\{\infty\}$ be the postcritical set.

The B\"{o}ttcher map $\phi_\lambda$ for $f_\lambda$ is defined in a
neighborhood of $\infty$ by
$\phi_\lambda(z)=\displaystyle\lim_{k\rightarrow\infty}(f^k_\lambda(z))^{n^{-k}}$.
The B\"{o}ttcher map is unique if we require
$\phi'_\lambda(\infty)=1$. It is known that the B\"{o}ttcher map
$\phi_\lambda$ can be extended to a domain ${\rm
Dom}(\phi_\lambda)\subset B_\lambda$ such that $\phi_\lambda: {\rm
Dom}(\phi_\lambda)\rightarrow \{z\in \mathbb{\bar{C}}: |z|>R\}$ is a
conformal isomorphism for some largest number $R\geq1$. In
particular, if $B_\lambda$ contains no critical point other than
$\infty$, then ${\rm Dom}(\phi_\lambda)=B_\lambda$; if  $B_\lambda$
contains a critical point  $c\in\{0\} \cup C_\lambda$, then by `The
Escape Trichotomy' Theorem \ref{1c}), the Julia set $J(f_\lambda)$
is a Cantor set.

The Green function $G_\lambda: B_\lambda\rightarrow (0,\infty]$ is
defined by
$$G_\lambda(z)=\lim_{k\rightarrow\infty}n^{-k}\log|f^k_\lambda(z)|.$$
By definition, $G_\lambda(f_\lambda(z))=nG_\lambda(z)$ for $z\in
B_\lambda$ and $G_\lambda(z)=\log|\phi_\lambda(z)|$ for $z\in {\rm
Dom}(\phi_\lambda)$. The Green function $G_\lambda$ can be extended
to $A_\lambda=\bigcup_{k\geq 0} f_\lambda^{-k}(B_\lambda)$ by
defining
$$G_\lambda(z)=n^{-k}G_\lambda(f_\lambda^k(z)) \text{ for } z\in
f_\lambda^{-k}(B_\lambda).$$

In the following, for a set $E$ in $\mathbb{\bar{C}}$ and $a\in
\mathbb{C}$, let $aE=\{az; z\in E\}$, $a+E=\{a+z; z\in E\}$,
$\bar{E}$ be the closure of $E$ and $\textrm{int}(E)$ be the
interior of $E$.

\begin{lem}[Symmetry of the Dynamical Plane]\label{1a}  Let $\omega$ satisfy $\omega^{2n}=1$; then,

1. $\omega J(f_\lambda)=J(f_\lambda)$.

2. $G_\lambda(\omega z)=G_\lambda(z)$ for $z\in A_\lambda$.

3. $\omega {\rm Dom}(\phi_\lambda)={\rm Dom}(\phi_\lambda)$, and $
\phi_\lambda(\omega z)=\omega\phi_\lambda(z)$ for $z\in {\rm
Dom}(\phi_\lambda)$.
\end{lem}

\begin{proof}
 For 1, because $A_\lambda=\{z\in\mathbb{\bar{C}}; f_\lambda^k(z)\text{ tends to infinity as }
 k\rightarrow \infty\}$ and $f^k_\lambda(\omega z)=\pm f^k_\lambda(z)$ for
 $k\geq1$, $f_\lambda^k(\omega z)$ tends toward infinity if and only if
 $f_\lambda^k(z)$  tends toward infinity as $k\rightarrow \infty$. Thus,
 $\omega A_\lambda=A_\lambda$. The conclusion follows from the fact
 that $J(f_\lambda)=\partial A_\lambda$.

2. By the definition of $G_\lambda$.

3. Because ${\rm Dom}(\phi_\lambda)$ is the connected component of
$\{z\in B_\lambda;G_\lambda(z)>\log R\}$ that contains $\infty$, we
conclude that $\omega {\rm Dom}(\phi_\lambda)={\rm
Dom}(\phi_\lambda)$. Note that $\phi_\lambda(\omega z)$ and
$\omega\phi_\lambda(z)$ are two Riemann mappings of ${\rm
Dom}(\phi_\lambda)$ onto $\{z\in\mathbb{\bar{C}}; |z|>R \}$ with the
same derivative at $\infty$, we have $\phi_\lambda(\omega
z)=\omega\phi_\lambda(z)$ by the uniqueness of the Riemann mapping
theorem.
\end{proof}

The non-escape locus of this family is defined by
$$M=\{\lambda\in\mathbb{C}^*; \  f_\lambda^k(v^+_\lambda)\text{ does not  tend to infinity  as }
 k\rightarrow \infty\}.$$

\begin{lem}[Symmetry of the Parameter Plane]\label{1b} The non-escape locus $M$ satisfies:

1. $M$ is symmetric about the real axis.

2. $\nu M=M$ with $\nu^{n-1}=1$.

3. For any line $\ell \in \{\epsilon \mathbb{R};
\epsilon^{2n-2}=1\}$, $M$ is symmetric about $\ell$.
\end{lem}

\begin{proof}
1. Because $\overline{f_\lambda(\bar{z})}=f_{\bar{\lambda}}(z)$, the
Critical orbit of $f_\lambda$ and the critical orbit of
$f_{\bar{\lambda}}$  are symmetric under the map $z\mapsto \bar{z}$,
they either both remain bounded or both tend to infinity. Thus, $M$
is symmetric about the real axis.

2. Let $\nu={e}^{2\pi i/(n-1)}$ and $\varphi(z)={e}^{\pi i/(n-1)}
z$. For $k\geq1$,
\begin{equation*}
\varphi^{-1}\circ f^k_{\nu\lambda}\circ \varphi(z)=\begin{cases}
 (-1)^k f^k_{\lambda}(z),\ \  &n \text{ odd},\\
f^k_{\lambda}(z),\ \ &n \text{ even}.
 \end{cases}
 \end{equation*}
Thus, the critical orbit of $f_\lambda$ tends toward infinity if and
only if the critical orbit of $f_{\nu\lambda}$ tends toward
infinity. Equivalently, $\lambda\in M$ if and only if $\nu\lambda\in
M$.

3. The conclusion follows from 1 and 2.
\end{proof}

From Lemma \ref{1b}, $f_\lambda$ and $f_{\lambda {e}^{2\pi
i/(n-1)}}$ have the same dynamical properties and their Julia sets
are identical up to a rotation. Thus, the fundamental domain of the
parameter plane is $\{\lambda\in \mathbb{C}^*; \arg \lambda \in [0,
\frac{2\pi}{n-1})\}$.

The following theorem of Devaney, Look and Uminsky gives a
classification of Julia sets of different topological types
\cite{DLU}.

\begin{thm}[Devaney-Look-Uminsky]\label{1c} The Escape
Trichotomy.

1. If $v_\lambda^+\in B_\lambda$, then $J(f_\lambda)$ is a Cantor
set.

2. If $v_\lambda^+\in T_\lambda\neq B_\lambda$, then $J(f_\lambda)$
is a Cantor set of circles.

3. If $f^k_\lambda(v_\lambda^+)\in T_\lambda\neq B_\lambda$ for some
$k\geq 1$, then $J(f_\lambda)$ is a Sierpi\'nski curve, which is
locally connected.

In all other cases, the critical orbits remain bounded and the Julia
set $J(f_\lambda)$ is connected.
\end{thm}

\begin{figure}
\begin{center}
\includegraphics[height=7cm]{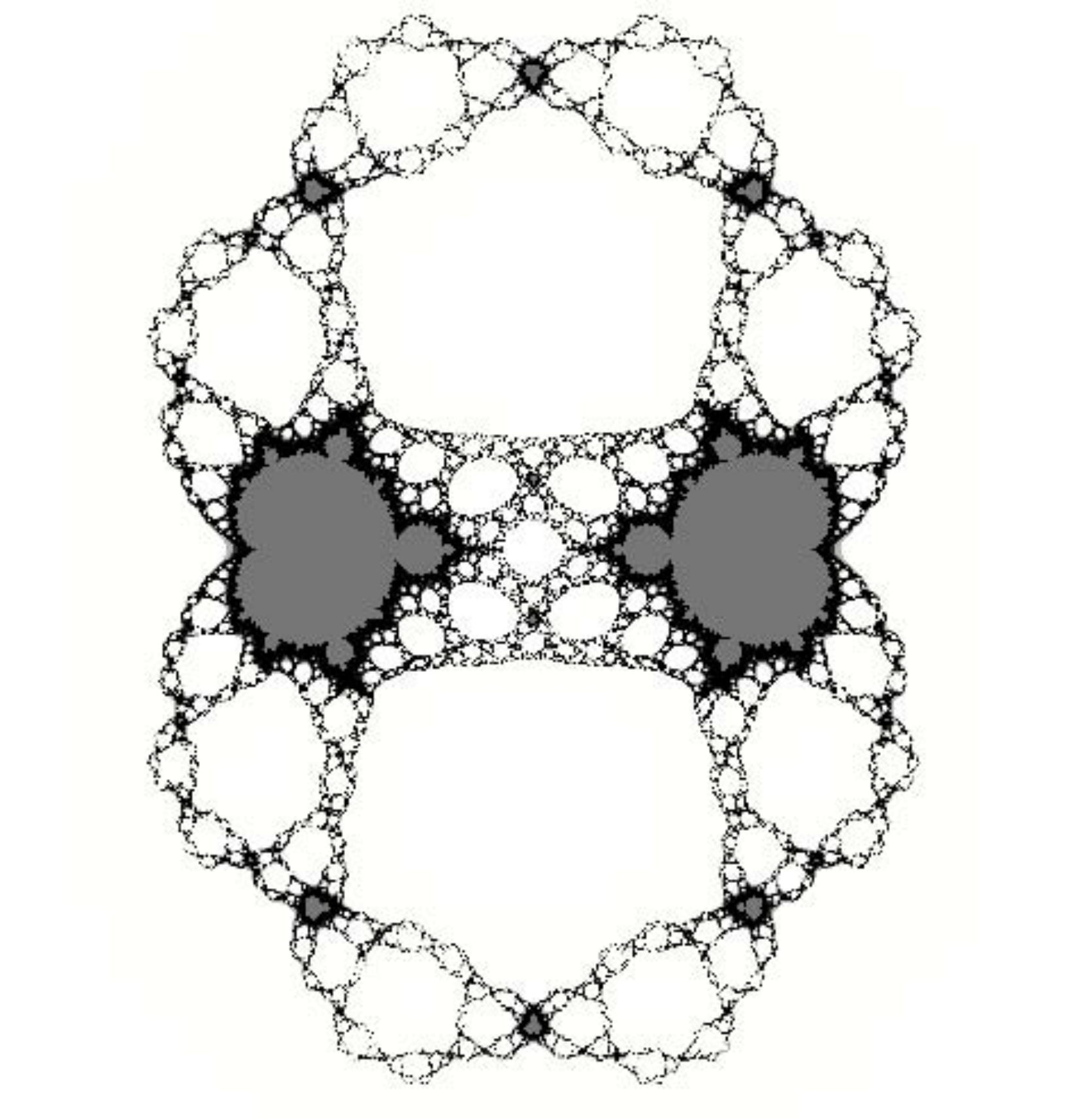}
\caption{Parameter plane for McMullen maps when $n=3$.}
\end{center}
\end{figure}

For $n\geq3$, it is known that the unbounded component of
$\mathbb{C}^*-M$ consists of the parameters for which the Julia set
is a Cantor set. This region is called a \textit{Cantor set locus}
(see Figure 1). The component of $\mathbb{C}^*-M$ that contains a
punctured neighborhood of $0$ is the region in which the Julia set
$J(f_\lambda)$ is a Cantor set of circles; this is referred to as
the \textit{McMullen domain} in honor of McMullen, who first
discovered this type of Julia set. The complement of these two
regions is the \textit{connected locus}. The small copies of the
quadratic Mandelbrot set correspond to the renormalizable
parameters, while the `holes' in the connected locus are always
called \textit{Sierpi\'nski holes} according to Devaney. These
regions correspond to the parameters for which the Julia set is a
Sierpi\'nski curve.

We will see later that, when the critical orbit tends to $\infty$,
the boundary $\partial B_\lambda$ is a quasi-circle if it is
connected. Thus, this case is already well studied.

In this paper, we will restrict our attention to the parameters
$\lambda\in \mathcal{H}=\{\lambda\in\mathbb{C}^*; \arg \lambda\in
(0,\frac{2\pi}{n-1})\}$ for the most part because of the symmetry of
the parameter plane. For these parameters, we can develop Yoccoz
puzzle techniques to study the local connectivity of Julia set.
However, for real parameters, Yoccoz puzzle theory cannot be applied
because of the absence of critical puzzle pieces. The real positive
parameters will be considered separately in Section 7.3.

 Therefore, if there is no further assumption, most discussions are based on the following:

\textbf{Hypothesis:} \textit{ $\lambda\in \mathcal{H}$ and the
critical orbits remain bounded, or equivalently, $C_\lambda\cap
A_\lambda=\emptyset$.}

\subsection{Notations}

Let $c_0=c_0(\lambda)=\sqrt[2n]{\lambda}$ be the critical point that
lies on $\mathbb{R}^+$ when $\lambda\in \mathbb{R}^+$ and varies
analytically as $\lambda$ ranges over $\mathcal{H}$. Let
$c_k=c_0{e}^{k\pi i/n}$ for $1\leq k\leq 2n-1$. The critical points
$c_k$ with $k$ even are mapped to $v_\lambda^+=2\sqrt{\lambda}$
while the critical points $c_k$ with $k$ odd are mapped to
$v_\lambda^-=-2\sqrt{\lambda}$.

\begin{figure}
\begin{center}
\includegraphics[height=7cm]{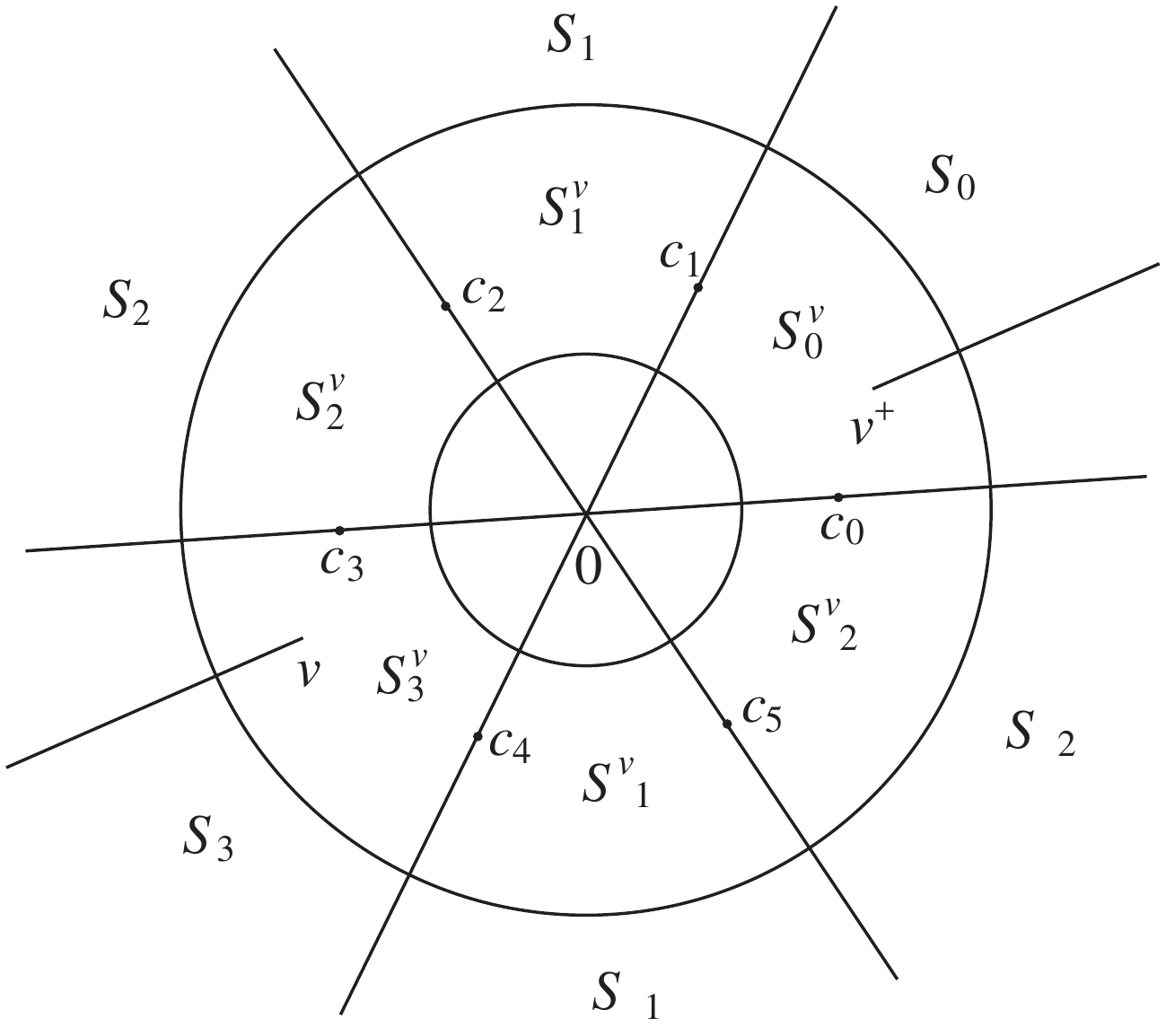}
\caption{Sectors in the dynamical plane when $n=3$.}
\end{center}
\end{figure}

Let $\ell_k=c_k \mathbb{R}^+ (\mathbb{R}^+:=[0, +\infty])$ be the
real straight line connecting the origin to $\infty$ and passing
through $c_k$ for $0\leq k\leq 2n-1$. We call $\ell_k$ a
\textit{critical ray}. The closed sector bounded by $\ell_k$ and
$\ell_{k+1}$ is denoted by $S_k$ for $0\leq k\leq n$. Define
$S_{-k}=-S_k$ for $1\leq k\leq n-1$. Therefore, the sectors are
arranged counterclockwise about the origin as
$S_0,S_1,\cdots,S_n,S_{-1},\cdots,S_{-(n-1)}$ (see Figure 2).

The critical value $v_\lambda^+$ always lies in $S_0$ because $\arg
c_0<\arg v_\lambda^+<\arg c_1$ for all $\lambda\in\mathcal{H}$.
Correspondingly, the critical value $v_\lambda^-$ lies in $S_n$. It
is easy to confirm that the image of $\ell_k$ under $f_\lambda$ is a
straight ray connecting one of the critical values to $\infty$; this
ray is called a \textit{critical value ray}. As a consequence,
$f_\lambda$ maps the interior of each of the sectors of
$\{S_{\pm1},\cdots,S_{\pm(n-1)}\}$ univalently onto a region
$\Upsilon_\lambda$, which can be identified as the complex sphere
$\mathbb{\bar{C}}$ minus two critical value rays.

Let $\mathcal{P}$ denote the set of all components of
$\bigcup_{k\geq0}f_\lambda^{-k}(B_\lambda)$. For $U\in \mathcal{P}$
and $v>0$, let $\mathbf{e}(U,v)=\{z\in U; G_\lambda(z)=v\}$ be the
equipotential curve. The annulus bounded by
$\mathbf{e}(B_\lambda,v)$ and $\mathbf{e}(T_\lambda,v)$ is denoted
by $Q_v$. We may choose a $v$ large enough that $\partial Q_v$
intersects with every critical ray  at exactly two points (to see
this, notice that the B\"ottcher map $\phi_\lambda:
B_\lambda\rightarrow \mathbb{\overline{C}}-\mathbb{\overline{D}}$
acts like the identity map near $\infty$; thus,
$\mathbf{e}(B_\lambda,v)$ looks like a circle when $v$ is large. The
curve $\mathbf{e}(T_\lambda,v)$ also looks like a circle because
$f_\lambda(\mathbf{e}(T_\lambda,v))=\mathbf{e}(B_\lambda,nv)$ and
$f_\lambda$ acts like $z\mapsto \lambda/z^n$ near zero). The bounded
and unbounded components of $\overline{\mathbb{C}}\setminus
\mathbf{e}(B_\lambda,v)$ are denoted by $\mathbf{V}(v)$ and
$\mathbf{U}(v)$, respectively.

Now, we define radial rays of $U$ for every
$U\in\mathcal{P}\setminus \{B_\lambda\}$. In the Hypothesis section,
we see that there is a unique Riemann mapping $\phi_{T_\lambda}:
T_\lambda\rightarrow\mathbb{D}$, such that
$$\phi_{T_\lambda}(z)^{-n}=\phi_{\lambda}(f_\lambda(z)), \ z\in T_\lambda ;\  \phi'_{T_\lambda}(0)=1/\sqrt[n]{\lambda}.$$
The radial ray $R_{T_\lambda}(\theta)$ of angle $\theta$ is defined
as $\phi_{T_\lambda}^{-1}((0,1)e^{2\pi i\theta})$. For $U\in
\mathcal{P}\setminus \{B_\lambda,T_\lambda\}$, there is a smallest
integer $k\geq1$, such that $f_\lambda^k: U\rightarrow T_\lambda$ is
a conformal map. The radial ray $R_U(\theta)$ is defined as the
pullback of $R_{T_\lambda}(\theta)$ under $f_\lambda^k$.

  Let $\mathbb{I}=\{0,n,\pm1,\dots,\pm(n-1)\}$ be an index set.
$S^v_k=\overline{Q_v}\cap S_k$ for $k\in \mathbb{I}$ and
$S^v=\bigcup_{k\in \mathbb{I}\setminus\{0,n\}} S^v_k$. The set of
all points with orbits that remain in $S^v$ under all iterations of
$f_\lambda$ is denoted by $\Lambda_\lambda$. Obviously,
$\Lambda_\lambda=\bigcap_{k\geq0}f_\lambda^{-k}(S^v)$.

For any $k\in \mathbb{I}\setminus\{0, n\}$, the map $f_\lambda:
\mathrm{int}(S_k)\rightarrow\Upsilon_\lambda$ is a conformal map;
its inverse is denoted by
$h_k:\Upsilon_\lambda\rightarrow\mathrm{int}(S_k)$.

Given a point $z\in \Lambda_\lambda$, suppose $f_\lambda^k(z)\in
S_{s_k}$ for $k\geq0$ and define the itinerary of $z$ as
$\mathbf{s}_\lambda(z)=(s_0,s_1,s_2,\cdots)$. The itinerary is
always well defined in the set $\Lambda_\lambda$ because if some
iteration $f_\lambda^k(z)$ lies on the boundary of two adjacent
sectors, then the next iteration $f_\lambda^{k+1}(z)$ will lie
inside $S_0\cup S_n$.

Let $\Sigma=\{\mathbf{s}=(s_0,s_1,s_2,\cdots);
 s_k\in \mathbb{I}\setminus\{0, n\} \text{ for\ every\ }k\geq0\}$ be the space of one-sided sequences of the symbols $\pm1,\dots,\pm(n-1)$.
For $\mathbf{s}=(s_0,s_1,s_2,\cdots)\in \Sigma$,and the shift map
$\sigma:\Sigma\rightarrow\Sigma$ is defined by
$\sigma(\mathbf{s})=(s_1,s_2,\cdots)$. If there is an integer $p>0$
such that $s_{k+p}=s_{k}$ for all $k\geq0$, we say the itinerary
$\mathbf{s}$ is periodic and the least integer $p$ is called the
period of $\mathbf{s}$. In this case, $\mathbf{s}$ is also denoted
by $(\overline{s_0,\cdots,s_{p-1}})$.

It is obvious that
$\mathbf{s}_\lambda(f_\lambda(z))=\sigma(\mathbf{s}_\lambda(z))$ for
$z\in\Lambda_\lambda$.

 \begin{lem}\label{2a}  The set $\Lambda_\lambda$ is a Cantor set, and the itinerary map $\mathbf{s}_\lambda: \Lambda_\lambda\rightarrow \Sigma$ is
bijective. Moreover, $\Lambda_\lambda\subset J(f_\lambda)$.
\end{lem}
\begin{proof} First, note that for any $\lambda\in \mathcal{H}$, $S^v$
is a compact subset of $\Upsilon_\lambda$. With respect to the
hyperbolic metric of $\Upsilon_\lambda$ and by the Schwarz Lemma,
there is a number $\delta \in (0,1)$ such that for any
$\mathbf{s}=(s_0,s_1,s_2,\cdots)\in \Sigma$ and any $m\geq0$,
 $$ \mathrm{Hyper. diam}\big(h_{s_0}\circ\cdots\circ h_{s_m}(S^v)\big)\leq
\mathrm{Hyper. diam }(S^v) \cdot \delta ^m.$$

Thus, $\bigcap_{k\geq0}h_{s_0}\circ\cdots\circ h_{s_k}(S^v)$
consists of a single point, say $z_\mathbf{s}$. Therefore,
$\Lambda_\lambda$ is a Cantor set, and the map $\mathbf{s}_\lambda:
\Lambda_\lambda\rightarrow \Sigma$ defined by
$\mathbf{s}_\lambda(z_\mathbf{s})=\mathbf{s}$ is bijective.

When ${\bf{s}}=(\overline{s_0,\cdots,s_{m-1}})\in \Sigma$ is a
periodic itinerary of period $m$, then $z_{\bf{s}}$ is a fixed point
of $h=h_{s_0}\circ\cdots\circ h_{s_{m-1}}$. Because $h: {\rm
int}(S^v)\to h({\rm int}(S^v))\subset {\rm int}(S_{s_0}^v)\Subset
{\rm int}(S^v)$ is strictly contractive, it follows by the Schwarz
lemma that the fixed point $z_{\bf{s}}$ is attracting. Therefore,
$z_{\bf{s}}$ is a repelling periodic point of $f_{\lambda}$.

To show $\Lambda_\lambda\subset J(f_\lambda)$, it suffices to prove
that any point of $\Lambda_\lambda$ can be approximated by a
sequence of repelling periodic points in $\Lambda_\lambda$. Suppose
$z\in \Lambda_\lambda$. 
For any $\varepsilon>0$, there is an integer $m>0$ such that
$\mathrm{Hyper. diam }(S^v) \cdot \delta ^m<\varepsilon$. Take a
periodic itinerary $\mathbf{s}\in\Sigma$ with first $m$ symbols that
are the same as those of $\mathbf{s}_\lambda(z)$. (Notice that such
an itinerary always exists.) Because the map $\mathbf{s}_\lambda$ is
bijective, there is a unique point $w\in \Lambda_\lambda$ with
$\mathbf{s}_\lambda(w)=\mathbf{s}$. The hyperbolic distance between
$z$ and $w$ is smaller than $\varepsilon$.  The previous argument
implies that $w$ is periodic and repelling.
\end{proof}

\section{Cut Rays in the Dynamical Plane}

In this section, we will construct the `cut ray', a type of Jordan
curve that cuts the Julia set into two different parts. The
construction is due to R. Devaney \cite{D1}. We give some additional
properties that will be used in our paper.

We first construct a Cantor set of angles on the unit circle and use
these angles to generate `cut rays' as in \cite{D1}. These angles
can be considered as a combinatorial invariant when the parameter
$\lambda$ ranges over $\mathcal{H}$.

To begin, we identify the unit circle
$\mathbb{S}=\mathbb{R}/\mathbb{Z}$ with $(0,1]$. We say that three
angles satisfy $t_1\leq t_2\leq t_3$ on $\mathbb{S}$ if
$t_1,t_2,t_3$ are in counterclockwise order.

\subsection{A Cantor set on the unit circle} In the following, we
construct a subset $\Theta$ of (0,1]. The set $\Theta$ is a Cantor
set and is used to generate `cut rays' in the next section.

First, define a map $\tau:  (0,1]\rightarrow (0,1]$ by $\tau
(\theta)=n\theta  \text{ mod 1} $. Let
$\Theta_k=(\frac{k}{2n},\frac{k+1}{2n}]$ for $0\leq k \leq n$ and
$\Theta_{-k}=\Theta_k+\frac{1}{2}$ for $1\leq k\leq n-1$. Obviously,
$(0,1]=\bigcup_{k\in\mathbb{I}}\Theta_k $.

Define a map $\chi:\mathbb{I}\rightarrow\mathbb{N}$ by
\begin{equation*}
\chi(k)=\begin{cases}
 k,\ \  &\text{ if } 0\leq k\leq n,\\
n-k,\ \ &\text{ if } -(n-1)\leq k\leq -1.
 \end{cases}
 \end{equation*}

For $k\in \mathbb{I}$, we have
\begin{equation*}
\tau(\Theta_k)\supset\begin{cases}
\bigcup_{j=1}^{n-1}\Theta_j,\ \  &\text{ if } \chi(k) \text{ is even},\\
\bigcup_{j=1}^{n-1}\Theta_{-j},\ \ &\text{ if } \chi(k) \text{ is
odd}.
 \end{cases}
 \end{equation*}

For $\theta\in (0,1]$, suppose $\tau^k(\theta)\in \Theta_{s_k}$ for
$k\geq0$ and define the itinerary $\mathbf{s}(\theta)$ of $\theta$
by $\mathbf{s}(\theta)=(s_0,s_1,s_2,\cdots)$.

Let $\Theta$ be the set of all angles $\theta\in (0,1]$ with orbits
that remain in $\mathcal{E}=\bigcup_{k=1}^{n-1}(\Theta_k\cup
\Theta_{-k})$ under all iterations of $\tau$. The set $\Theta$ can
be written as $\Theta=\bigcap_{k\geq0}\tau^{-k}(\mathcal{E})
=\bigcap_{k\geq0}\tau^{-k}(\overline{\mathcal{E}}).$ One can easily
verify that $\Theta$ is a Cantor set.

The image of $\Theta$ under the itinerary map is denoted by
$\Sigma_0=\{\mathbf{s}(\theta); \theta\in\Theta\}$. One can easily
verify that $\Sigma_0$ is a subspace of $\Sigma$ that consists of
all elements $\mathbf{s}=(s_0,s_1,s_2,\cdots)\in \Sigma$ such that
for $k\geq0$, if $\chi(s_k)$ is even, then $s_{k+1}\in
\{1,\cdots,n-1\}$; if $\chi(s_k)$ is odd, then $s_{k+1}\in
\{-1,\cdots,-(n-1)\}$.


The itinerary map $\mathbf{s}:\Theta\rightarrow\Sigma_0$ is
bijective because for any $\mathbf{s}=(s_0,s_1,s_2,\cdots)\in
\Sigma_0$, the intersection
$\bigcap_{k\geq0}\tau^{-k}(\Theta_{s_k})$ consists of a single
point. In the following, we first construct an inverse map for
$\mathbf{s}$ (Lemma \ref{3a}).

Let $\mathbf{s}=(s_0,s_1,s_2,\cdots)\in \Sigma$. We define a map
$\kappa:\Sigma\rightarrow (0,1]$ by
$$\mathbf{\kappa}(\mathbf{s})=\frac{1}{2}\bigg(\frac{\chi(s_0)}{n}+\sum_{k\geq1}\frac{|s_k|}{n^{k+1}}\bigg).$$

\begin{lem}\label{3a} $\kappa(\Sigma)=\Theta$ and $\kappa(\mathbf{s}(\theta))=\theta$ for all $\theta\in\Theta$.
\end{lem}
\begin{proof}
First, we show $\kappa(\mathbf{s}(\theta))=\theta$ for
$\theta\in\Theta$. Let $\mathbf{s}(\theta)=(s_0,s_1,s_2,\cdots)$ and
$\hat{\theta}=\kappa(\mathbf{s}(\theta))$. Because
$\mathbf{s}:\Theta\rightarrow\Sigma_0$ is bijective, it suffices to
show that $\mathbf{s}(\hat{\theta})=\mathbf{s}(\theta)$.

It follows that $\hat{\theta}\in \Theta_{s_0}$ because
$$\frac{\chi(s_0)}{2n}<\hat{\theta}\leq\frac{1}{2}\bigg(\frac{\chi(s_0)}{n}+\sum_{k\geq1}\frac{n-1}{n^{k+1}}\bigg)
=\frac{\chi(s_0)}{2n}+\frac{1}{2n}.$$

For $k\geq1$,
\begin{equation*}
\tau^k(\hat{\theta})=\begin{cases}
\frac{1}{2}(\chi(s_0)+|s_1|+\cdots+|s_{k-1}|)+
\frac{1}{2}\sum_{j\geq k}\frac{|s_j|}{n^{j-k+1}},\ \  &\text{if } n \text{ is odd},\\
\frac{|s_{k-1}|}{2}+ \frac{1}{2}\sum_{j\geq
k}\frac{|s_j|}{n^{j-k+1}},\ \ &\text{if } n \text{ is even}.
 \end{cases}
 \end{equation*}

Because $\mathbf{s}(\theta)=(s_0,s_1,s_2,\cdots)\in\Sigma_0$, we
have for $j\geq1$,
\begin{equation*}
\frac{|s_j|}{2}=\begin{cases}
\frac{1}{2}(\chi(s_j)-\chi(s_{j-1})) \text{ mod }1,\ \  &\text{if } n \text{ is odd},\\
\frac{1}{2}\chi(s_j) \text{ mod }1,\ \ &\text{if } n \text{ is
even}.
 \end{cases}
 \end{equation*}
and
$$\frac{\chi(s_{j-1})}{2}+\frac{|s_j|}{2n}=\frac{\chi(s_{j})}{2n}  \text{ mod }1.$$

Thus, we have

$$\tau^k(\hat{\theta})=\frac{\chi(s_{k-1})}{2}+
\frac{1}{2}\sum_{j\geq
k}\frac{|s_j|}{n^{j-k+1}}=\frac{\chi(s_{k})}{2n}+
\frac{1}{2}\sum_{j\geq k+1}\frac{|s_j|}{n^{j-k+1}}.$$

This means $\tau^k(\hat{\theta})\in \Theta_{s_k}$ for $k\geq1$.
Therefore, $\theta$ and $\hat{\theta}$ have the same itinerary.

In the following, we show $\kappa(\Sigma)=\Theta$. First, by the
previous argument, $\Theta=\kappa(\Sigma_0)\subset\kappa(\Sigma)$.
Conversely, for any $\mathbf{s}=(s_0,s_1,s_2,\cdots)\in\Sigma$,
there is a unique sequence of symbols $\epsilon_1,\epsilon_2,\cdots
\in \{\pm1\}$, such that
$\mathbf{s^*}=(s_0,\epsilon_1s_1,\epsilon_2s_2,\cdots)\in\Sigma_0$.
Thus, $\kappa(\mathbf{s})=\kappa(\mathbf{s^*})\in\Theta$.
\end{proof}

\begin{rem}
For any $\mathbf{s}=(s_0,s_1,s_2,\cdots)\in\Sigma$, one can verify
that
 $$\kappa^{-1}(\kappa(\mathbf{s}))=\{(s_0,\pm s_1,\pm
s_2,\cdots)\}.$$
\end{rem}

\begin{lem}\label{3b}  The set $\Theta$ satisfies:

1. $\tau(\Theta)=\Theta$.

2. $\Theta+\frac{1}{2}=\Theta$.

3. Periodic angles are dense in $\Theta$.

\end{lem}
\begin{proof}

1. It is obvious that $\tau(\Theta)\subset\Theta$. $\tau$ is
surjective because $\tau^{-1}(\theta)\cap \mathcal{E}\neq\emptyset$
for all $\theta\in\Theta$.

2. First note that $\mathcal{E}+\frac{1}{2}=\mathcal{E}$ mod 1. For
$k\geq1$, because $\tau^k(\theta+\frac{1}{2})=\tau^k(\theta)$ when
$n$ is even and
$\tau^k(\theta+\frac{1}{2})=\tau^k(\theta)+\frac{1}{2}$ when $n$ is
odd, we have $\tau^k(\theta+\frac{1}{2})\in \mathcal{E}$ if and only
if $\tau^k(\theta)\in \mathcal{E}$. Thus, $\theta\in\Theta$ if and
only if $\theta+\frac{1}{2}\in\Theta$.

3. Let $\theta\in \Theta$ with itinerary
$\mathbf{s}(\theta)=(s_0,s_1,s_2,\cdots)$. For any $k\geq1$, either
$(\overline{s_0,\cdots,s_k})\in\Sigma_0$, or there is a symbol
$s_{k+1}^*\in \{\pm 1,\cdots,\pm (n-1)\}$ such that
$(\overline{s_0,\cdots,s_k,s_{k+1}^*})\in\Sigma_0$. If
$(\overline{s_0,\cdots,s_k})\in\Sigma_0$ , let
$\theta_k=\kappa((\overline{s_0,\cdots,s_k}))$. Else, let
$\theta_k=\kappa((\overline{s_0,\cdots,s_k,s_{k+1}^*}))$. It's
obvious that $\theta_k$ is periodic. By Lemma \ref {3a},
$\theta_k\in \Theta$ and
$$|\theta-\theta_k|\leq C(n)n^{-k} ( \  \rightarrow0 \text { \ as \ } k\rightarrow\infty),$$
where $C(n)$ is a constant, depending only on $n$, which implies
that periodic angles are dense in $\Theta$.
\end{proof}

\begin{rem}
The Hausdorff dimension of $\Theta$ is $\frac{\log(n-1)}{\log n}$.
\end{rem}

For $\lambda\in \mathcal{H}$ and $k\in \mathbb{I}$, let
$\Theta_k^\lambda=\Theta_k+\frac{\arg
c_0(\lambda)}{2\pi}=\Theta_k+\frac{\arg \lambda}{4n\pi}$ mod 1.
Recall that for $\lambda\in \mathcal{H}$, $\arg \lambda\in(0,
\frac{2\pi}{n-1})$. It is easy to check that

\begin{equation*}
\tau(\Theta_k^\lambda)\supset\begin{cases}
 \bigcup_{j=1}^{n-1}\Theta_j^\lambda,\ \  &\text{ if } \chi(k) \text{ is even},\\
\bigcup_{j=1}^{n-1}\Theta_{-j}^\lambda,\ \ &\text{ if } \chi(k)
\text{ is odd}.
 \end{cases}
 \end{equation*}

 Again, we define $\Theta^\lambda$ as the set of all angles in $(0,1]$ whose orbits remain in
$\mathcal{E}^\lambda=\bigcup_{k=1}^{n-1}(\Theta_k^\lambda\cup
\Theta_{-k}^\lambda)$ under all iterations of $\tau$. Thus,
$\Theta^\lambda=\bigcap_{k\geq0}\tau^{-k}(\mathcal{E}^\lambda).$ For
$\theta\in (0,1]$, suppose $\tau^k(\theta)\in \Theta_{s_k}^\lambda$
for $k\geq0$ and define the itinerary of $\theta$ by
$\mathbf{s}^\lambda(\theta)=(s_0,s_1,s_2,\cdots)$. It is easy to
show that the itinerary map
$\mathbf{s}^\lambda:\Theta^\lambda\rightarrow \Sigma_0$ is
bijective.

\begin{lem}\label{3b1}  $\Theta^\lambda=\Theta$ and for any
$\theta\in\Theta$, $\mathbf{s}^\lambda(\theta)=\mathbf{s}(\theta)$.
\end{lem}
\begin{proof}
It suffices to show that if
$\mathbf{s}^\lambda(\alpha)=\mathbf{s}(\beta)$ for
$\alpha\in\Theta^\lambda$ and $\beta\in\Theta$, then $\alpha=\beta$.

First, note that $\Theta_k^\lambda\cap\Theta_k\neq\emptyset$ for any
$k\in \mathbb{I}$. Suppose
$\mathbf{s}^\lambda(\alpha)=\mathbf{s}(\beta)=(s_0,s_1,s_2,\cdots)$,
and let $A_m=\bigcap_{0\leq k\leq
m}\tau^{-k}(\Theta^\lambda_{s_k}\cap\Theta_{s_k})$ for $m\geq 0$. By
induction, we see that $A_m$ is a connected interval of the form
$(a_m,b_m]$ with $a_{m+1}>a_m, b_{m+1}<b_m$ and
$n(b_{m+1}-a_{m+1})=b_m-a_m$ for $m\geq0$. Thus, $A_{m+1}\subset
\overline{A_{m+1}}\subset A_{m}$ and
$\bigcap_{k\geq0}A_m=\bigcap_{k\geq0}\overline{A_m}$ consists of a
single point, say $\theta$. On the other hand,
$$\{\theta\}=\bigcap_{k\geq0}A_m=\Big(\bigcap_{k\geq0}\tau^{-k}(\Theta^\lambda_{s_k})\Big)
\bigcap
\Big(\bigcap_{k\geq0}\tau^{-k}(\Theta_{s_k})\Big)=\{\alpha\}\cap\{\beta\}.$$

Thus, we have $\alpha=\beta=\theta$.
\end{proof}

\subsection{Cut rays}

In this section, for any $\lambda\in \mathcal{H}$ and any $\theta\in
\Theta$, we will construct a Jordan curve, say
$\Omega_\lambda^\theta$, that cuts the dynamical plane of
$f_\lambda$ into two parts. The curve will meet the Julia set
$J(f_\lambda)$ in a Cantor set of points. This kind of Jordan curve
$\Omega_\lambda^\theta$ will be called a 'cut ray' of angle
$\theta$. In the following, we construct such rays following a
slightly different presentation from Devaney's in \cite{D1}.

Recall that the itinerary map $\mathbf{s}_\lambda:
\Lambda_\lambda\rightarrow\Sigma$ from a Cantor set onto a symbolic
space is bijective. We first extend the definition of
$\mathbf{s}_\lambda$ to a larger set. Let
$E_\lambda=\bigcap_{k\geq0}f_\lambda^{-k}(\bigcup_{j\in\mathbb{I}\setminus\{0,n\}}S_j)$
be the set of all points in the dynamical plane with orbits that
remain in $\bigcup_{j\in\mathbb{I}\setminus\{0,n\}}S_j$ under all
iterations of $f_\lambda$. By definition, $E_\lambda$ is a compact
subset of $\mathbb{\overline{C}}$ containing $0$ and $\infty$.  The
assumption $\lambda\in \mathcal{H}$ implies that $E_\lambda$
contains no critcal points other than $0$ and $\infty$.

Let $O_\lambda=\cup_{k\geq0}f_\lambda^{-k}(\infty)$ be the grand
orbit of $\infty$. The map $\mathbf{s}_\lambda:
\Lambda_\lambda\rightarrow\Sigma$ can be extended to
$\mathbf{s}_\lambda: E_\lambda\setminus O_\lambda\rightarrow\Sigma$
as follows: for any $z\in E_\lambda\setminus O_\lambda$, suppose
$f_\lambda^k(z)\in S_{s_k}$ for $k\geq0$; the itinerary of $z$ is
then defined by $\mathbf{s}_\lambda(z)=(s_0,s_1,s_2,\cdots)$. One
can see that the map $\mathbf{s}_\lambda: E_\lambda\setminus
O_\lambda\rightarrow\Sigma$ is well-defined. (In fact, if
$f_\lambda^n(z)$ lies on the intersection of two sectors, then
$f_\lambda^{n+1}(z)$ will land on the critical value ray).

Given an angle $\theta\in \Theta$ with itinerary
$\mathbf{s}(\theta)=(s_0,s_1,s_2,\cdots)$, it is easy to check that
when $n$ is odd,
$\mathbf{s}(\theta+1/2)=(-s_0,-s_1,-s_2,\cdots)=-\mathbf{s}(\theta)$
and that when $n$ is even,
$\mathbf{s}(\theta+1/2)=(-s_0,s_1,s_2,\cdots)$. We consider the set
of all points in $E_\lambda\setminus O_\lambda$ with itineraries
that take the form $(s_0,\pm s_1,\pm s_2,\cdots)$. The closure of
this set is denoted by $\omega_\lambda^\theta$:
$$\omega_\lambda^\theta=\overline{\{z\in E_\lambda\setminus
O_\lambda; \mathbf{s}_\lambda(z) =(s_0,\pm s_1,\pm
s_2,\cdots)\}}=\overline{\{z\in E_\lambda\setminus O_\lambda;
\kappa(\mathbf{s}_\lambda(z))=\theta\}}.$$

According to Devaney, the set $\omega_\lambda^\theta$ is called a
'full ray' of angle $\theta$. Let
$\Omega_\lambda^\theta=\omega_\lambda^\theta\cup\omega_\lambda^{\theta+1/2}$;
we call the set $\Omega_\lambda^\theta$ a 'cut ray' of angle
$\theta$ (or $\theta+1/2$). One may verify that
$$\Omega_\lambda^\theta=\overline{\{z\in E_\lambda\setminus
O_\lambda; \mathbf{s}_\lambda(z)
 =(\pm s_0,\pm s_1,\pm s_2,\cdots)\}}
=\bigcap_{k\geq0} f_\lambda^{-k}(S_{s_k}\cup S_{-s_k}).$$

We first give an intuitive description of the cut ray
$\Omega_\lambda^\theta$. For $m\geq0$, let
$$\Omega_{\lambda,m}^\theta=\bigcap_{0\leq k\leq
m}f^{-k}_\lambda(S_{s_k}\cup S_{-s_k}).$$ Note that the set
$\Omega_{\lambda,0}^\theta$ is a union of the two closed sectors
$S_{s_0}$ and $S_{-s_0}$. $\Omega_{\lambda,1}^\theta$  is a string
of four closed disks that lie inside $\Omega_{\lambda,0}^\theta$.
Inductively, $\Omega_{\lambda,m}^\theta$ is a string of $2^{m+1}$
closed disks that are contained in $\Omega_{\lambda,m-1}^\theta$,
and each of these disks meets exactly two others at the preimages of
$\infty$. Hence, $\Omega_{\lambda,m}^\theta$ is a connected and
compact set. One can show that $\Omega_{\lambda,m}^\theta$ converges
to $\Omega_{\lambda}^\theta=\cap_{k\geq0} \Omega_{\lambda,k}^\theta$
in Hausdorff topology as $m\rightarrow\infty$ (because a shrinking
sequence of compact sets always converges in Hausdorff topology).
Roughly, the set $\Omega_{\lambda,m}^\theta$ becomes thinner when
$m$ becomes larger and $\Omega_{\lambda,m}^\theta$ finally shrinks
to $\Omega_{\lambda}^\theta$. It is therefore conjectured that
$\Omega_{\lambda}^\theta$ is a Jordan curve. (A rigorous proof of
this fact will be given in Proposition \ref {3e1}).

By construction, the cut ray satisfies:

$\bullet$ \ $\Omega_\lambda^\theta=-\Omega_\lambda^\theta$.

$\bullet$ \ $\Omega_\lambda^\theta\setminus \{0,\infty\}$ is
contained in the interior of $S_{s_0}\cup S_{-s_0}$.

$\bullet$ \ $f_\lambda:\Omega_\lambda^{\theta}
 \rightarrow\Omega_\lambda^{\tau(\theta)}$
is a two-to-one map.

$\bullet$ \ $\bigcup_{\theta\in
\Theta}\Omega_\lambda^{\theta}=E_\lambda$.

\begin{lem}\label{3c} Let $\lambda\in \mathcal{H}$; then, there is a
constant $v>0$ such that for any $\theta\in\Theta$,
$$\overline{R_\lambda(\theta)}\cap \mathbf{U}(v)=\{z\in E_\lambda\cap \mathbf{U}(v);
\mathbf{s}_\lambda(z)=\mathbf{s}(\theta)\}.$$
\end{lem}
\begin{proof}
For any small number $\varepsilon>0$, we define
$\Theta_{k,\varepsilon}^\lambda=[\frac{\chi(k)}{2n}+\frac{\arg\lambda}{4n\pi}+\varepsilon,
\frac{\chi(k)+1}{2n}+\frac{\arg\lambda}{4n\pi}-\varepsilon]$,
$S_{k,\varepsilon}={\{z\in S_k\setminus\{0,\infty\}; \arg z \in
\Theta_{k,\varepsilon}^\lambda\}}\cup\{0,\infty\}$ for $k\in
\mathbb{I}\setminus \{0,n\}$. It is obvious that $S_{k,\varepsilon}$
is a closed subset of $S_k$. One can verify that there is an
$\varepsilon>0$ such that
$\Theta^\lambda=\bigcap_{j\geq0}\tau^{-j}(\bigcup_{k=1}^{n-1}(\Theta_{k,\varepsilon}^\lambda\cup
\Theta_{-k,\varepsilon}^\lambda))$ and
$E_\lambda=\bigcap_{k\geq0}f_\lambda^{-k}(\bigcup_{j\in\mathbb{I}\setminus\{0,n\}}S_{j,\varepsilon})$.
Thus, for any $\theta\in \Theta$ with
$\mathbf{s}(\theta)=(s_0,s_1,\cdots)$, the cut ray
$\Omega_\lambda^\theta=\bigcap_{k\geq0}
f_\lambda^{-k}(S_{s_k,\varepsilon}\cup S_{-s_k,\varepsilon})$. We
fix such $\varepsilon$ (notice that $\varepsilon$ is independent of
$\theta\in \Theta$).

Because $\phi_\lambda'(\infty)=1$, we may choose $v=v(\varepsilon)$
large enough such that $|\arg z-\arg\phi_\lambda(z)|<\varepsilon$
for all $z\in\mathbf{U}(v)$. We define a map $\zeta: \mathbf{U}(v)
\rightarrow \mathbb{S}$ by $\zeta(z)=\frac{\arg
\phi_\lambda(z)}{2\pi}$. The map $\zeta$ satisfies $\zeta\circ
f_\lambda=\tau\circ \zeta$.

If $z\in \overline{R_\lambda(\theta)}\cap \mathbf{U}(v)$ and $z\neq
\infty$, then for any $k\geq0$, $\arg\phi_\lambda(f_\lambda^k(z))\in
\Theta_{s_k,\varepsilon}^\lambda$. We conclude that $\arg
f^k_\lambda(z)\in \Theta_{s_k}^\lambda$. Or, equivalently,
$f^k_\lambda(z)\in S_{s_k}$ for all $k \geq 0$. Thus,
$\mathbf{s}_\lambda(z)=\mathbf{s}(\theta)$.

On the other hand, for any $\infty\neq z\in E_\lambda\cap
\mathbf{U}(v)$ with $\mathbf{s}_\lambda(z)=\mathbf{s}(\theta)$, we
know from the above that $f^k_\lambda(z)\in S_{s_k,\varepsilon}$ for
all $k \geq 0$, thus $\arg f^k_\lambda(z)\in
\Theta_{s_k,\varepsilon}^\lambda$. It turns out that $\arg
\phi_\lambda (f^k_\lambda(z))=\tau^k(\zeta(z))\in
\Theta_{s_k}^\lambda$. By Lemma \ref{3b1},
$\mathbf{s}^\lambda(\zeta(z))=\mathbf{s}(\theta)=\mathbf{s}^\lambda(\theta)$.
Thus, we have $\zeta(z)=\theta$; this means $z\in
{R_\lambda(\theta)}\cap \mathbf{U}(v)$.
\end{proof}

\begin{pro}\label{3d} For any $\lambda\in \mathcal{H}$ and any $\theta\in\Theta$,
the external ray $R_\lambda(\theta)$ lands at a unique point
$p_\lambda(\theta)\in\partial B_\lambda$ and
$\overline{R_\lambda(\theta)}=\{z\in E_\lambda\setminus O_\lambda;
\mathbf{s}_\lambda(z)=\mathbf{s}(\theta)\}\cup \{\infty\}=\{z\in
(E_\lambda\setminus O_\lambda)\cap B_\lambda;
\mathbf{s}_\lambda(z)=\mathbf{s}(\theta)\}\cup
\{p_\lambda(\theta)\}\cup \{\infty\}$.
\end{pro}
\begin{proof}
Suppose $\mathbf{s}(\theta)=(s_0,s_1,s_2,\cdots)$. Let
$\ell_\lambda(v,\theta)=\{z\in R_\lambda(\theta); v\leq
G_\lambda(z)\leq nv\}$ be the portion of $R_\lambda(\theta)$ that
lies between two equipotential curves $\mathbf{e}(B_\lambda, v)$ and
$\mathbf{e}(B_\lambda, nv)$. Based on Lemma \ref {3c}, we can assume
$v$ large enough such that for any $\beta\in\Theta$,
$\overline{R_\lambda(\beta)}\cap \mathbf{U}(v)=\{z\in E_\lambda\cap
\mathbf{U}(v); \mathbf{s}_\lambda(z)=\mathbf{s}(\beta)\}$. By
pulling back $\ell_\lambda(v,\tau(\theta))$ by $f_\lambda^{-1}$  to
$S_{s_0}$, we can extend the portion of
$\overline{R_\lambda(\theta)}$, say
$\gamma_0=\overline{R_\lambda(\theta)}\cap \mathbf{U}(v)$, to a
longer one $\gamma_1=h_{s_0}(\ell_\lambda(v,\tau(\theta)))\cup
\gamma_0$. Obviously, $\gamma_1\subset S_{s_0}\cap
\overline{R_\lambda(\theta)}$. Continuing inductively, suppose we
have already constructed a portion $\gamma_k$ of
$\overline{R_\lambda(\theta)}$; we then add a segment
$h_{s_0}\circ\cdots\circ
h_{s_k}(\ell_\lambda(v,\tau^{k+1}(\theta)))$ to $\gamma_k$ and
obtain $\gamma_{k+1}=\gamma_k\cup h_{s_0}\circ\cdots\circ
h_{s_k}(\ell_\lambda(v,\tau^{k+1}(\theta)))$. By construction, one
can confirm that $h_{s_0}\circ\cdots\circ
h_{s_k}(\ell_\lambda(v,\tau^{k+1}(\theta)))\subset S_{s_0}\cap
\overline{R_\lambda(\theta)}$, and that for any $z\in
h_{s_0}\circ\cdots\circ
h_{s_k}(\ell_\lambda(v,\tau^{k+1}(\theta)))$,
$\mathbf{s}_{\lambda}(z)=(s_0,s_1,s_2,\cdots)$. It turns out that
$$R_\lambda(\theta)\setminus \gamma_0=\bigcup_{k\geq0} h_{s_0}\circ\cdots\circ
h_{s_k}(\ell_\lambda(v,\tau^{k+1}(\theta))).$$\

In the following, we show that the external ray $R_\lambda(\theta)$
lands at $\partial B_\lambda$. Because $h_k:
\Upsilon_\lambda\rightarrow \Upsilon_\lambda$ contracts the
hyperbolic metric $\rho_\lambda$ of $\Upsilon_\lambda$ for any $k\in
\mathbb{I}\setminus\{0,n\}$, there is a constant $\delta\in (0,1)$
such that
$$\rho_\lambda(h_k(x), h_k(y))\leq \delta \rho_\lambda(x, y),\ \ \ \
\forall x,y \in \overline{\mathbf{V}(nv)}\cap\big(\cup_{j\in
\mathbb{I}\setminus \{0,n\}}S_j\big),\forall k\in
\mathbb{I}\setminus \{0,n\}.$$

Notice that $\bigcup_{\alpha\in \Theta}\ell_\lambda(v,\alpha)=
E_\lambda\cap \{z\in B_\lambda;  v\leq G_\lambda(z)\leq nv\}$ is a
compact subset of
 $\Upsilon_\lambda$, with
respect to the hyperbolic metric of $\Upsilon_\lambda$ we have
$$ \mathrm{Hyper. length}\big( h_{s_0}\circ\cdots\circ
h_{s_k}(\ell_\lambda(v,\tau^{k+1}(\theta)))\big)=\mathcal{O}(\delta^k).$$
This implies that $R_\lambda(\theta)\setminus \gamma_0$ has finite
hyperbolic length in $\Upsilon_\lambda$; thus, the external ray
$R_\lambda(\theta)$ lands at $\partial B_\lambda$. Let
$p_\lambda(\theta)$ be the landing point. It is easy to confirm that
$\mathbf{s}_\lambda(p_\lambda(\theta))=\mathbf{s}(\theta)$ and
$p_\lambda(\theta)\in \partial B_\lambda\cap \Lambda_\lambda$. Thus,
we have \bess \overline{R_\lambda(\theta)} &\subset&\{z\in
(E_\lambda\setminus O_\lambda)\cap B_\lambda;
\mathbf{s}_\lambda(z)=\mathbf{s}(\theta)\}\cup
\{p_\lambda(\theta)\}\cup \{\infty\}\\
&\subset&\{z\in E_\lambda\setminus O_\lambda;
\mathbf{s}_\lambda(z)=\mathbf{s}(\theta)\}\cup \{\infty\}\eess

Finally, we show $\overline{R_\lambda(\theta)}\supset\{z\in
E_\lambda\setminus O_\lambda;
\mathbf{s}_\lambda(z)=\mathbf{s}(\theta)\}\cup \{\infty\}$. For any
$x\in\{z\in E_\lambda\setminus O_\lambda;
\mathbf{s}_\lambda(z)=\mathbf{s}(\theta)\}$, we consider the orbit
of $x$.

If the orbit of $x$ remains bounded, then based on Lemma \ref {2a},
we have $x\in \Lambda_\lambda$. Because
$\mathbf{s}_\lambda|_{\Lambda_\lambda}: \Lambda_\lambda\rightarrow
\Sigma$ is bijective and
$\mathbf{s}_\lambda(x)=\mathbf{s}_\lambda(p_\lambda(\theta))=\mathbf{s}(\theta)$,
we conclude $x=p_\lambda(\theta)\in \overline{R_\lambda(\theta)}$.

If the orbit of $x$ tends toward $\infty$, then by Lemma \ref {3c},
there is an integer $M\geq1$ such that $f_\lambda^M(x)\in
R_\lambda(\tau^M(\theta))$. Note that for any $j\geq 0$, the above
argument implies $R_\lambda(\tau^j(\theta))\subset S_{s_j}$. Because
$f_\lambda(R_\lambda(\tau^{k-1}(\theta)))=R_\lambda(\tau^{k}(\theta))$
and $h_{s_{k-1}}$ is the inverse branch of $f: {\rm
int}(S_{s_{k-1}})\rightarrow\Upsilon_\lambda$, we conclude that for
all $k\geq1$,
$h_{s_{k-1}}(R_\lambda(\tau^k(\theta)))=R_\lambda(\tau^{k-1}(\theta))$
and $h_{s_{k-1}}(f_\lambda^k(x))=f_\lambda^{k-1}(x)$. It turns out
that $x\in R_\lambda(\theta)$ and
$\overline{R_\lambda(\theta)}\supset\{z\in E_\lambda\setminus
O_\lambda; \mathbf{s}_\lambda(z)=\mathbf{s}(\theta)\}\cup
\{\infty\}$.
\end{proof}

\begin{pro}\label{3e} For any $\lambda\in \mathcal{H}$ and any
$\theta\in\Theta$ with itinerary
$\mathbf{s}(\theta)=(s_0,s_1,s_2,\cdots)$, the cut ray
$\Omega_\lambda^\theta$ satisfies:

1. $\Omega_\lambda^\theta$ meets the Julia set $J(f_\lambda)$ in a
Cantor set of points. More precisely, $\Omega_\lambda^\theta\cap
J(f_\lambda)=(\kappa\circ\mathbf{s}_\lambda|_{\Lambda_\lambda})^{-1}(\{\theta,\theta+\frac{1}{2}\})$.

2.  $\Omega_\lambda^\theta$ meets the Fatou set $F(f_\lambda)$ in a
countable union of external rays and radial rays together with the
preimages of $\infty$ that lie in the closure of these rays. More
precisely, \bess \Omega_\lambda^\theta\cap
B_\lambda&=&R_\lambda(\theta)\cup
R_\lambda(\theta+\frac{1}{2})\cup\{\infty\}\\
 \Omega_\lambda^\theta\cap
T_{\lambda}&=&\begin{cases}
 h_{-s_0}(R_\lambda(\tau(\theta)))\cup
 h_{s_0}(R_\lambda(\tau(\theta)+\frac{1}{2}))\cup \{0\},\ \  &\text{ if } n \text{ is odd},\\
 h_{s_0}(R_\lambda(\tau(\theta)+\frac{1}{2}))\cup
 h_{-s_0}(R_\lambda(\tau(\theta)+\frac{1}{2}))\cup \{0\},\ \ &\text{ if } n \text{ is even}.
 \end{cases}
 \eess
For any $U\in \mathcal{P}\setminus \{B_\lambda,T_\lambda\}$ with
 $U\cap \Omega_\lambda^\theta\neq\emptyset$, $U$ is of the form
 $h_{b_0}\circ\cdots\circ h_{b_{k-1}}(T_\lambda)$, where $k\geq1$
 and $(b_0,\cdots,b_{k-1})\in \{(\pm s_0,\cdots,\pm s_{k-1})\}$.
 Moreover,
\bess &&\Omega_\lambda^\theta\cap U=
h_{b_0}\circ\cdots\circ h_{b_{k-1}}(\Omega_\lambda^{\tau^k(\theta)}\cap T_\lambda)\\
&=&\begin{cases}
 h_{b_0}\circ\cdots\circ h_{b_{k-1}}\Big( h_{-s_k}(R_\lambda(\tau^{k+1}(\theta)))\cup
  h_{s_k}(R_\lambda(\tau^{k+1}(\theta)+\frac{1}{2}))\cup \{0\}\Big),\ &\text{if } n \text{ is odd},\\
  h_{b_0}\circ\cdots\circ h_{b_{k-1}}\Big( h_{-s_k}(R_\lambda(\tau^{k+1}(\theta)+\frac{1}{2}))\cup
  h_{s_k}(R_\lambda(\tau^{k+1}(\theta)+\frac{1}{2}))\cup \{0\}\Big),\ &\text{if } n \text{ is even}.
 \end{cases}
 \eess
\end{pro}

See Figure 3 for the combinatorial structure of a part of a cut ray.

\begin{figure}
\begin{center}
\includegraphics[height=7.5cm]{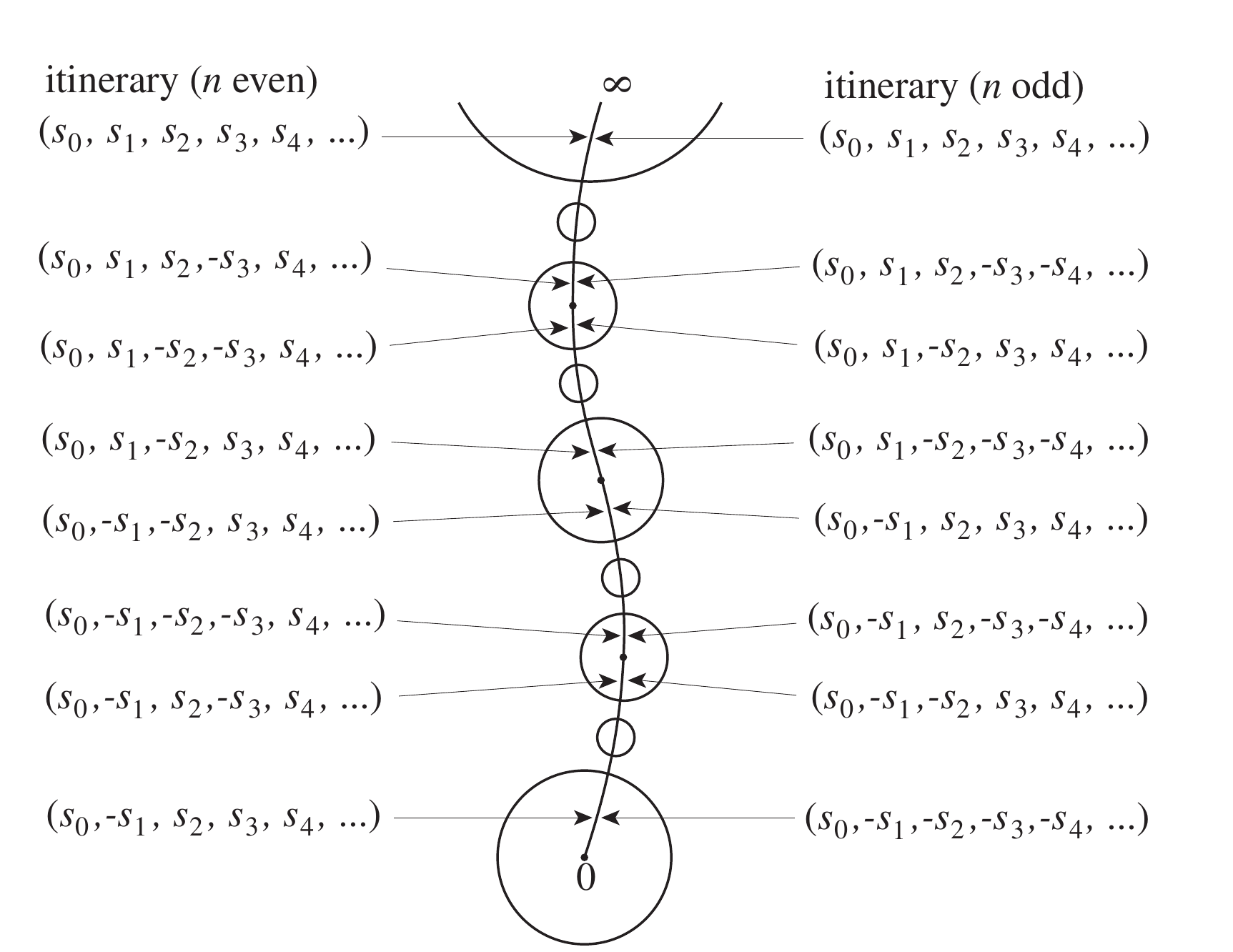}
\caption{Combinatorial structure of a full ray
$\omega_\lambda^\theta$ with
$\mathbf{s}(\theta)=(s_0,s_1,s_2,\cdots)$.}
\end{center}
\end{figure}

\begin{proof}
1. For $z\in  \Omega_\lambda^\theta$, first note that $z\in
\Omega_\lambda^\theta\cap J(f_\lambda)$ if and only if the orbit of
$z$ remains bounded, if and only if $z\in\Lambda_\lambda$ and
$\mathbf{s}_\lambda(z)\in \{(\pm s_0,\pm s_{1},\pm
s_{2},\cdots)\}=\kappa^{-1}(\{\theta,\theta+\frac{1}{2}\})$. Thus,
we have $\Omega_\lambda^\theta\cap
J(f_\lambda)=(\kappa\circ\mathbf{s}_\lambda|_{\Lambda_\lambda})^{-1}(\{\theta,\theta+\frac{1}{2}\})$.

2. Let $U$ be a Fatou component such that $U\cap
\Omega_\lambda^\theta\neq\emptyset$. Then, by 1, $U$ is eventually
mapped onto $B_\lambda$.

\textbf{Case 1: $U=B_\lambda$. }  By Proposition \ref {3d},
 $\Omega_\lambda^\theta\cap
B_\lambda\supset R_\lambda(\theta)\cup
R_\lambda(\theta+\frac{1}{2})\cup\{\infty\}$.  On the other hand,
for any $z\in (\Omega_\lambda^\theta\cap B_\lambda)\setminus
\{\infty\}$, there is an integer $M\geq1$ such that
$f_\lambda^M(z)\in \mathbf{U}(v)$, where $v$ is a positive constant
chosen by Lemma \ref {3c}. Because
$\mathbf{s}_\lambda(f_\lambda^M(z))\in \{(\pm s_M,\pm s_{M+1},\pm
s_{M+2},\cdots)\}$, we conclude that the itinerary of
$f_\lambda^M(z)$ must be the same as that of some angle
$\beta\in\Theta$. Thus,
\begin{equation*}
\mathbf{s}_\lambda(f_\lambda^M(z))=\begin{cases}
(s_M,s_{M+1},s_{M+2},\cdots) \text{ or }(-s_M,-s_{M+1},-s_{M+2},\cdots),\ \  &\text{ if } n \text{ is odd},\\
(s_M,s_{M+1},s_{M+2},\cdots),\ \ &\text{ if } n \text{ is even}.
 \end{cases}
 \end{equation*}

\begin{figure}
\begin{center}
\includegraphics[height=8cm]{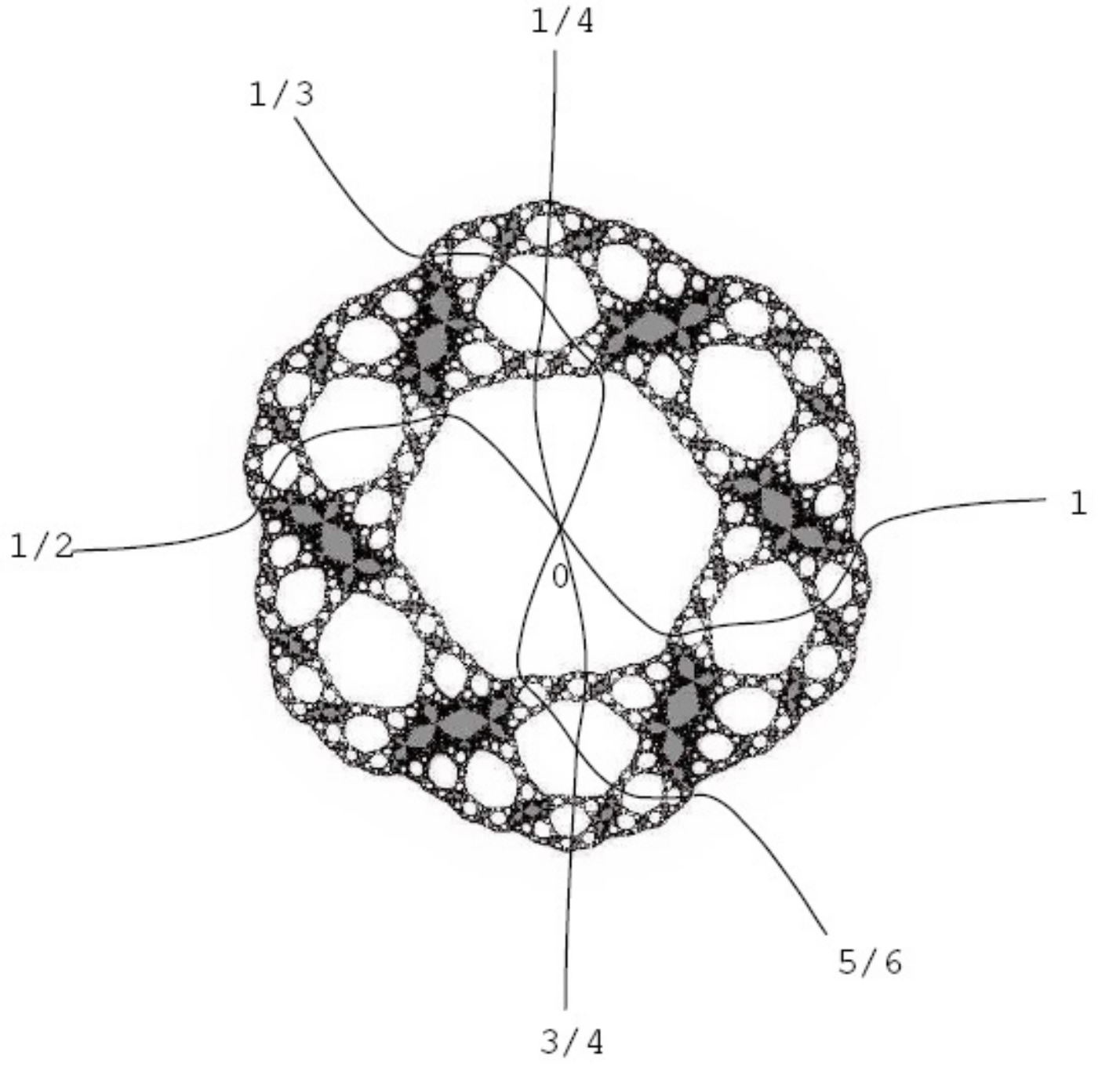}
\caption{Cut rays with angles $1/4,1/3, 1/2$ when $n=3$.}
\end{center}
\end{figure}

\textbf{Case 1.1. $n$ is odd. } By Proposition \ref {3d},
$f_\lambda^M(z)\in R_\lambda(\tau^M(\theta))\cup
R_\lambda(\tau^M(\theta)+\frac{1}{2})$. Note that
$f_\lambda^{-1}(R_\lambda(\tau^M(\theta)))\cap (S_{s_{M-1}}\cup
S_{-s_{M-1}})\cap B_\lambda=R_\lambda(\tau^{M-1}(\theta))$,
$f_\lambda^{-1}(R_\lambda(\tau^M(\theta)+\frac{1}{2}))\cap
(S_{s_{M-1}}\cup S_{-s_{M-1}})\cap
B_\lambda=R_\lambda(\tau^{M-1}(\theta)+\frac{1}{2})$. We conclude
that $f_\lambda^{M-1}(z)\in R_\lambda(\tau^{M-1}(\theta))\cup
R_\lambda(\tau^{M-1}(\theta)+\frac{1}{2})$. It turns out that $z\in
R_\lambda(\theta)\cup R_\lambda(\theta+\frac{1}{2})$ by induction.
So in this case, $\Omega_\lambda^\theta\cap
B_\lambda=R_\lambda(\theta)\cup
R_\lambda(\theta+\frac{1}{2})\cup\{\infty\}.$

\textbf{Case 1.2. $n$ is even. } By Proposition \ref {3d},
$f_\lambda^M(z)\in R_\lambda(\tau^M(\theta))$. Because
$f_\lambda^{-1}(R_\lambda(\tau^M(\theta)))\cap (S_{s_{M-1}}\cup
S_{-s_{M-1}})\cap B_\lambda=R_\lambda(\tau^{M-1}(\theta))\cup
R_\lambda(\tau^{M-1}(\theta)+\frac{1}{2})$, we have
$f_\lambda^{M-1}(z)\in R_\lambda(\tau^{M-1}(\theta))\cup
R_\lambda(\tau^{M-1}(\theta)+\frac{1}{2})$. If $M=1$, then $z\in
R_\lambda(\theta)\cup R_\lambda(\theta+\frac{1}{2})$, and the proof
is done. If $M>1$, then we claim $f_\lambda^{M-1}(z)\in
R_\lambda(\tau^{M-1}(\theta))$. This is because
$f_\lambda^{-1}(R_\lambda(\tau^{M-1}(\theta)+\frac{1}{2}))\cap
(S_{s_{M-2}}\cup S_{-s_{M-2}})\cap B_\lambda=\emptyset$. Again, by
induction, we have  $z\in R_\lambda(\theta)\cup
R_\lambda(\theta+\frac{1}{2})$ in this case.

\textbf{Case 2. $U=T_\lambda$. } In this case, if $n$ is odd, then
$f_\lambda(\Omega_\lambda^\theta\cap T_\lambda\cap
S_{s_0})=\Omega_\lambda^{\tau(\theta)}\cap B_\lambda\cap
S_{-s_1}=R_\lambda(\tau(\theta)+\frac{1}{2})\cup \{\infty\}$ and
$f_\lambda(\Omega_\lambda^\theta\cap T_\lambda\cap
S_{-s_0})=\Omega_\lambda^{\tau(\theta)}\cap B_\lambda\cap
S_{s_1}=R_\lambda(\tau(\theta))\cup \{\infty\}$. So
$\Omega_\lambda^\theta\cap T_\lambda=
h_{-s_0}(R_\lambda(\tau(\theta)))\cup
 h_{s_0}(R_\lambda(\tau(\theta)+\frac{1}{2}))\cup \{0\}$; if $n$ is even, then
$f_\lambda(\Omega_\lambda^\theta\cap T_\lambda\cap
S_{s_0})=f_\lambda(\Omega_\lambda^\theta\cap T_\lambda\cap
S_{-s_0})=\Omega_\lambda^{\tau(\theta)}\cap B_\lambda\cap
S_{-s_1}=R_\lambda(\tau(\theta)+\frac{1}{2})\cup \{\infty\}$. So
$\Omega_\lambda^\theta\cap T_\lambda=
 h_{s_0}(R_\lambda(\tau(\theta))+\frac{1}{2})\cup
 h_{-s_0}(R_\lambda(\tau(\theta)+\frac{1}{2}))\cup \{0\}.$

 \textbf{Case 3. $U\in \mathcal{P}\setminus \{B_\lambda,T_\lambda\}$.
 } In this case, there is a smallest integer $k\geq1$ such that
 $f_\lambda^k(U)=T_\lambda$.
 Because $f_\lambda^k:U\rightarrow T_\lambda$ is a conformal map and
 for any $0\leq j\leq k-1$, $f_\lambda^j(U)$ lies inside some sector
 $S_{k_j}$, we conclude $U$ must take the form
 $h_{b_0}\circ\cdots\circ h_{b_{k-1}}(T_\lambda)$ for some
 $(b_0,\cdots,b_{k-1})\in \{(\pm s_0,\cdots,\pm s_{k-1})\}$. By
 pulling back $f_\lambda^k(U\cap
 \Omega_\lambda^\theta)=\Omega_\lambda^{\tau^k(\theta)}\cap T_\lambda$ via
 $f_\lambda^{k}$, we have $\Omega_\lambda^\theta\cap U=
h_{b_0}\circ\cdots\circ
h_{b_{k-1}}(\Omega_\lambda^{\tau^k(\theta)}\cap T_\lambda)$. The
conclusion follows by case 2.
\end{proof}

\begin{pro}\label{3e1} For any $\lambda\in \mathcal{H}$ and any $\theta\in\Theta$, the cut ray $\Omega_\lambda^\theta$  is a Jordan curve.
\end{pro}
\begin{proof} Suppose $\mathbf{s}(\theta)=(s_0,s_1,s_2,\cdots)$.
 For $k\geq0$, define
$$\widehat{\Omega}_{\lambda,0}^{\tau^k(\theta)}=\Omega_\lambda^{\tau^k(\theta)}\cup
S^v_{s_k}\cup S^v_{-s_k},\ \
\widehat{\Omega}_{\lambda,k}^\theta=\bigcap_{0\leq j\leq
k}f_\lambda^{-j}(\widehat{\Omega}_{\lambda,0}^{\tau^j(\theta)}).$$

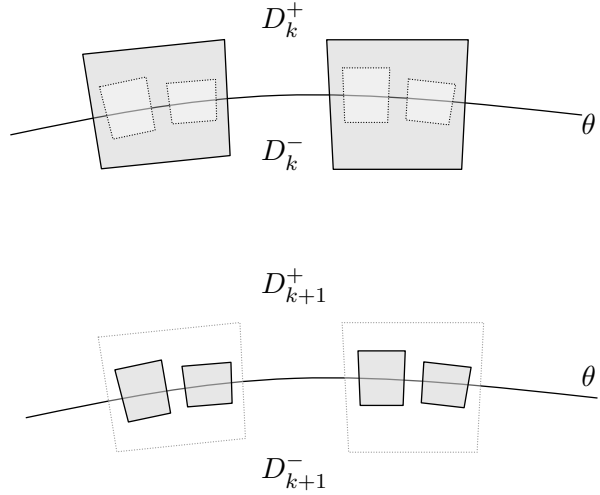
\begin{figure}[h]
\centering{
\includegraphics[height=7cm]{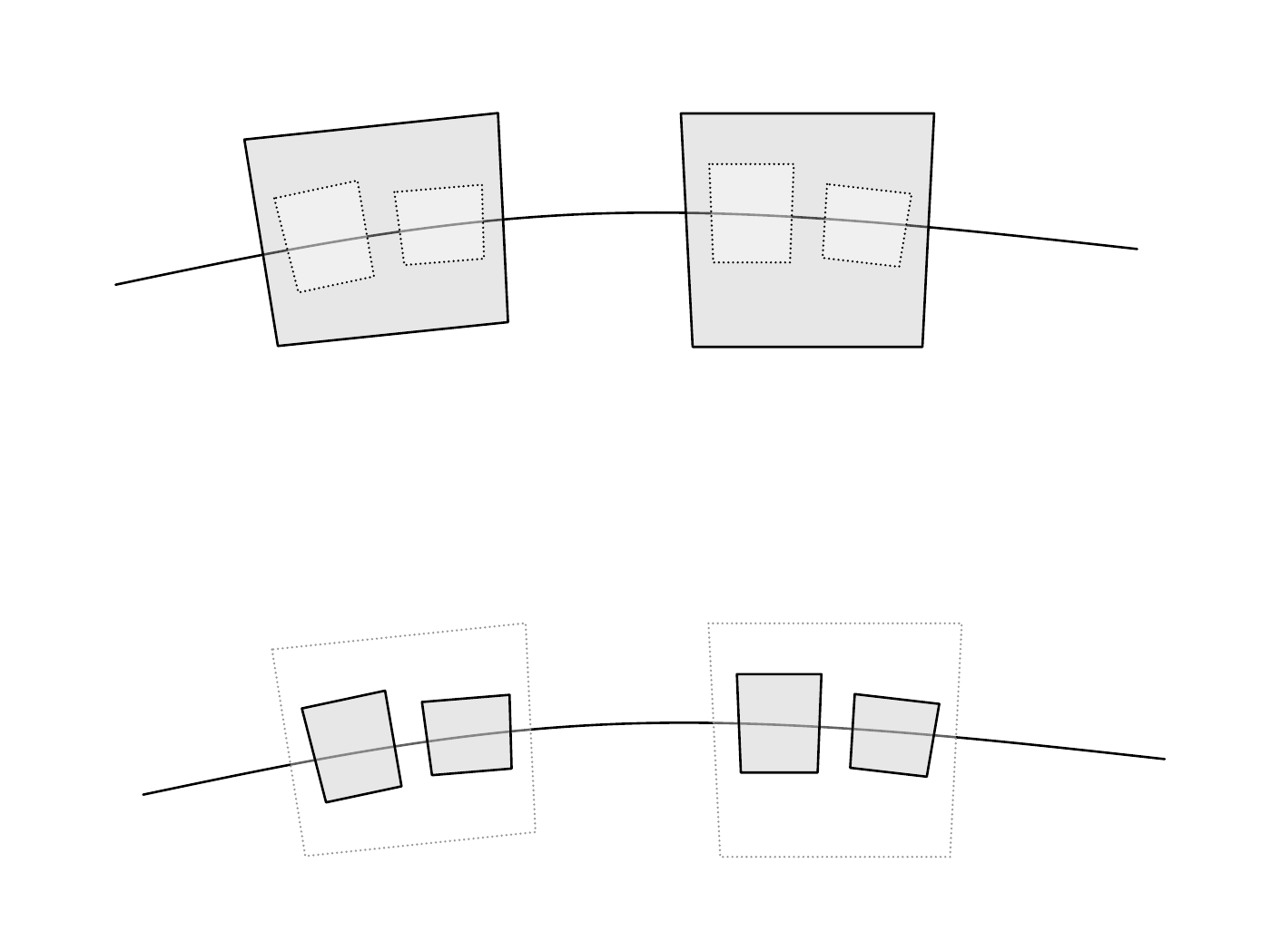}
\put(-150,180){$D_k^+$} \put(-150,130){$D_k^-$}
\put(-30,140){$\theta$} \put(-30,45){$\theta$}
 \put(-150,80){$D_{k+1}^+$}
\put(-150,10){$D_{k+1}^-$}
 \caption{The cut ray union  with the shadow regions is $\widehat{\Omega}_{\lambda,k}^\theta$ (resp.
 $\widehat{\Omega}_{\lambda,k+1}^\theta$).
 The two components of $\mathbb{\overline{C}}-\widehat{\Omega}_{\lambda,k}^\theta$ (resp.
  $\mathbb{\overline{C}}-\widehat{\Omega}_{\lambda,k+1}^\theta$) are $D_k^+$ and $D_k^-$ (resp. $D_{k+1}^+$ and $D_{k+1}^-$) .}}
\end{figure}

The set $\widehat{\Omega}_{\lambda,k}^\theta$ is connected and
compact, and it contains $\Omega_\lambda^\theta$. It is easy to
check that $\widehat{\Omega}_{\lambda,k}^\theta\supset
\widehat{\Omega}_{\lambda,k+1}^\theta$ and $\bigcap_{k\geq0}
\widehat{\Omega}_{\lambda,k}^\theta=\Omega_\lambda^\theta$. In the
following discussion, we can assume $k$ is sufficiently large that
$\widehat{\Omega}_{\lambda,k}^\theta$ avoids the critical values
$v_\lambda^{\pm}$. Let $D_k^+$ be the component of
$\overline{\mathbb{C}}\setminus \widehat{\Omega}_{\lambda,k}^\theta$
that contains $v_\lambda^+$ and $D_k^-$ be the component of
$\overline{\mathbb{C}}\setminus \widehat{\Omega}_{\lambda,k}^\theta$
that contains $v_\lambda^-$. Let $D_\infty^+=\bigcup_{k\geq0}D_k^+$
and $D_\infty^-=\bigcup_{k\geq0}D_k^-$; then, $D_\infty^+\cup
D_\infty^-\cup \Omega_\lambda^\theta=\overline{\mathbb{C}}$.

We first construct a Cantor set on
$\mathbb{S}=\mathbb{R}/\mathbb{Z}$. Let $E_1=(5/24, 13/24),
E_2=(17/24, 25/24)$ be two open intervals on $\mathbb{S}$ and
$\zeta$ be the map $t\mapsto 3t $ mod $\mathbb{Z}$. By definition,
$\zeta(E_i)\supset \overline{E_1\cup E_2}$. Let $T_k=\bigcap_{0\leq
j\leq k }\zeta^{-j}(E_1\cup E_2)$. Then, $T_k\supset T_{k+1}$ and
$T_k$ has $2^{k+1}$ components. The intersection $\bigcap_{k\geq0}
T_k$ is denoted by $T_\infty$. Because
$T_\infty=\bigcap_{k\geq0}\zeta^{-k}(E_1\cup
E_2)=\bigcap_{k\geq0}\zeta^{-k}(\overline{E_1}\cup \overline{E_2})$,
we conclude that $T_\infty$ is a Cantor set.

Now, we define two sequences of Jordan curves
$\{\gamma_k^+:\mathbb{S}\rightarrow \partial D_k^+\},\
\{\gamma_k^-:\mathbb{S}\rightarrow \partial D_k^-\}$ in the
following manner: for $k$ large enough,

1. $\gamma_{k+1}^+|_{\mathbb{S}\setminus
T_k}=\gamma_{k}^+|_{\mathbb{S}\setminus
T_k}=\gamma_{k}^-|_{\mathbb{S}\setminus
T_k}=\gamma_{k+1}^-|_{\mathbb{S}\setminus T_k}$.

2. $\gamma_{k}^+(\mathbb{S}\setminus T_k)=\Omega_\lambda^\theta\cap
\partial D_k^+ =\Omega_\lambda^\theta\cap
\partial D_k^-
=\gamma_{k}^-(\mathbb{S}\setminus T_k)$.

3. $\gamma_{k}^+(T_k)=\partial D_k^+\setminus
\Omega_\lambda^\theta,\ \gamma_{k}^-(T_k)=\partial D_k^-\setminus
\Omega_\lambda^\theta$.

In the following, we show that each sequence of maps
$\{\gamma_k^+:\mathbb{S}\rightarrow \partial D_k^+\},\
\{\gamma_k^-:\mathbb{S}\rightarrow \partial D_k^-\}$ converges in
the spherical metric. By construction,
$\gamma_{k+1}^+|_{\mathbb{S}\setminus
T_k}=\gamma_{k}^+|_{\mathbb{S}\setminus T_k}$, and for any component
$W$ of $T_k$, $\gamma_{k+1}^+(W)$ and $\gamma_{k}^+(W)$ are
contained in the same component of $\bigcap_{0\leq j \leq
k}f_\lambda^{-j}(S^v_{s_j}\cup S^v_{-s_j})$. Because the spherical
metric and the  hyperbolic metric are comparable in any compact
subset of $\Upsilon_\lambda$, we conclude by Lemma \ref {2a} that
$$\max_{t\in \mathbb{S}}\
   \mathrm{dist}_{\overline{\mathbb{C}}}\big(\gamma_{k+1}^+(t),\gamma_{k}^+(t))=\mathcal{O}(\delta^k\big),$$
where $\mathrm{dist}_{\overline{\mathbb{C}}}$ is the spherical
metric and $\delta\in (0,1)$ is a constant. Thus, the sequence
$\{\gamma^+_k\}$ has a limit map $\gamma^+_\infty:
\mathbb{S}\rightarrow \partial D_\infty^+$ that is continuous and
surjective. Similarly, the sequence $\{\gamma^-_k\}$ also has a
limit map $\gamma^-_\infty: \mathbb{S}\rightarrow \partial
D_\infty^-$ that is continuous and surjective. The limit maps
$\gamma^+_\infty$ and $\gamma^-_\infty$ satisfy
$\gamma^+_\infty|_{\mathbb{S}\setminus
T_\infty}=\gamma^-_\infty|_{\mathbb{S}\setminus T_\infty}$. By
continuity, $\gamma^+_\infty$ and  $\gamma^-_\infty$ are identical
on $\mathbb{S}$. This implies that $\partial D_\infty^+=\partial
D_\infty^-=\Omega_\lambda^\theta$ and $\Omega_\lambda^\theta$ is
locally connected.

To finish, we show that $\Omega_\lambda^\theta$ is a Jordan curve following the idea in \cite{PT}.
Let $\Phi: \mathbb{D}\rightarrow D_\infty^+$ be a Riemann mapping.
Because $\partial D_\infty^+$ is locally connected, $\Phi$ has an
extension from $\overline{\mathbb{D}}$ to $\overline{D_\infty^+}$.
If two distinct radial segments $\Phi((0,1)e^{2\pi i \theta_1})$ and
$\Phi((0,1)e^{2\pi i \theta_2})$ converge on the same point $p$,
then the Jordan curve $\Phi((0,1)e^{2\pi i \theta_1})\cup
\Phi((0,1)e^{2\pi i \theta_2})\cup\{\Phi(0), p\}$ separates  a
section of the boundary $\partial D_\infty^+$ from $ D_\infty^-$.
But this is a contradiction because $ D_\infty^+$ and $ D_\infty^-$
share a common boundary.
\end{proof}

\begin{pro}\label{3f} For $\lambda\in \mathcal{H}$ and
$\theta\in\Theta$, all periodic points on $\Omega_\lambda^\theta\cap
J(f_\lambda)$ are repulsive.
\end{pro}

\begin{proof} Suppose $\mathbf{s}(\theta)=(s_0,s_1,s_2,\cdots)$. Let
$z\in \Omega_\lambda^\theta\cap J(f_\lambda)$ be a periodic point
with period $p$. The itinerary of $z$ is then of the form
$(\overline{a_0,a_1,\cdots, a_{p-1}})$, where $a_j\in \{\pm s_j\}$
for $0\leq j\leq p-1$. Let $a_k=a_{k \ {\rm  mod } \ p }$ for
$k\geq0$ and $S^v_{a_0 \cdots a_{s}}=\bigcap_{0\leq k\leq
s}f_\lambda^{-k}(S^v_{a_k})$. By Lemma \ref {2a}, the hyperbolic
diameter of $S^v_{a_0 \cdots a_{s}}$ is $\mathcal{O}(\delta^s)$ when
$s$ is large. We can therefore choose an $N$ sufficiently large that
$f_\lambda^p:{\rm int}(S^v_{a_0 \cdots a_{N}})\rightarrow {\rm
int}(S^v_{a_p \cdots a_{N}})={\rm int}(S^v_{a_0 \cdots a_{N-p}})$ is
a conformal map. Because $z\in {\rm int}(S^v_{a_0 \cdots
a_{N}})\subset S^v_{a_0 \cdots a_{N}}\subset {\rm int}(S^v_{a_0
\cdots a_{N-p}})$, we conclude $|(f_\lambda^p)'(z)|>1$ by the
Schwarz Lemma. Thus, $z$ is a repelling periodic point.
 \end{proof}

Proposition \ref {3e} tells us the combinatorial structure of the
cut ray $\Omega_\lambda^\theta$. The following proposition shows
that the iterated preimages of $\Omega_\lambda^\theta$ have the same
combinatorial structure as $\Omega_\lambda^\theta$ provided that
$\Omega_\lambda^\theta$ does not meet the critical orbit.

\begin{pro}\label{3g} For $\lambda\in \mathcal{H}$ and
$\theta\in\Theta$, suppose the cut ray $\Omega_\lambda^\theta$ does
not meet the critical orbit. Then, for any $\alpha\in
\bigcup_{k\geq0}\tau^{-k}(\theta)$, there is a unique ray
$\omega_\lambda^\alpha$ such that:

1. $\omega_\lambda^\alpha$ is a continuous curve connecting $0$ with
$\infty$.

2. $\omega_\lambda^{\alpha+{1}/{2}}=-\omega_\lambda^\alpha$.

3. $
f_\lambda(\omega_\lambda^\alpha)=\omega_\lambda^{\tau(\alpha)}\cup
\omega_\lambda^{\tau(\alpha)+{1}/{2}}$.

4. $\omega_\lambda^\alpha\cap B_\lambda=R_\lambda(\alpha)\cup
\{\infty\}$.
\end{pro}

For this reason, we still call $\omega_\lambda^\alpha$ a full ray of
angle $\alpha$ and $\Omega_\lambda^\alpha=\omega_\lambda^\alpha\cup
\omega_\lambda^{\alpha+{1}/{2}}$ a cut ray of angle $\alpha$ (or
$\alpha+\frac{1}{2}$).

\begin{proof} The proof is based on an inductive argument. Suppose $\alpha\in
\bigcup_{k\geq0}\tau^{-k}(\theta)$ is an angle such that the full
ray $\omega_\lambda^\alpha$ and the cut ray $\Omega_\lambda^\alpha$
satisfy 1,2,3,4. Then, for $\beta\in\tau^{-1}(\alpha)$, we define
$\omega_\lambda^\beta$ by lifting $\Omega_\lambda^\alpha$ in the
following way:
$$f_\lambda(\omega_\lambda^\beta)=\Omega_\lambda^\alpha,\ \
\omega_\lambda^\beta\cap B_\lambda=R_\lambda(\beta)\cup
\{\infty\}.$$
 The ray $\omega_\lambda^\beta$ is unique because we require $\omega_\lambda^\beta\cap B_\lambda=R_\lambda(\beta)\cup
\{\infty\}$. Also, by uniqueness of lifting maps, we conclude
$\omega_\lambda^{\beta+\frac{1}{2}}=-\omega_\lambda^\beta$ by the
fact $R_\lambda(\beta+\frac{1}{2})=-R_\lambda(\beta)$ and
$\Omega_\lambda^{\alpha}=-\Omega_\lambda^\alpha$.

In the following, we show that $\omega_\lambda^\beta$ connects
$\infty$ and $0$. If not, then $\omega_\lambda^\beta$ must be a
curve connecting $\infty$ with itself, hence a Jordan curve. This
implies that $\omega_\lambda^\beta$ does not meet $0$. Because
$\Omega_\lambda^{\alpha}=-\Omega_\lambda^\alpha$, all curves in the
set $\mathcal{C}=\{e^{k\pi i/n}\omega_\lambda^\beta,\
H_\lambda(e^{k\pi i/n}\omega_\lambda^\beta); 0\leq k< 2n\}$ are
preimages of $\Omega_\lambda^\alpha$, where
$H_\lambda(z)=\sqrt[n]{\lambda}/z$. Because $\Omega_\lambda^\alpha$
does not meet the critical orbit, we conclude that for any
$\gamma_1,\gamma_2\in \mathcal{C}$ with $\gamma_1\neq\gamma_2$,
$\gamma_1$ and $\gamma_2$ are disjoint outside $\{0,\infty\}$. This
means $\#\mathcal{C}=4n$. However, this is a contradiction because
the degree of $f_\lambda$ is $2n$.
\end{proof}

Recall that for any $\theta\in \Theta$ with itinerary
$\mathbf{s}(\theta)=(s_0,s_1,s_2,\cdots)$, the cut ray
$\Omega_{\lambda}^\theta$ contains at least two points, $0$ and
$\infty$, and $\Omega_{\lambda}^\theta\setminus \{0,\infty\}$ is
contained in the interior of $S_{s_0}\cup S_{-s_0}$. Now, given two
angles $\alpha,\beta\in \Theta$ with $\Omega_{\lambda}^\alpha\neq
\Omega_{\lambda}^\beta$, suppose
$\mathbf{s}(\alpha)=(s^\alpha_0,s^\alpha_1,s^\alpha_2,\cdots),\
\mathbf{s}(\beta)=(s^\beta_0,s^\beta_1,s^\beta_2,\cdots)$. Let
$\mathbf{J}(\alpha,\beta)$ be the first integer $k\geq0$ such that
$|s^\alpha_k|\neq |s^\beta_k|$. Note that the intersection
$\Omega_{\lambda}^\alpha\cap \Omega_{\lambda}^\beta$ consists of at
least two points $0$ and $\infty$. Furthermore, if
$\mathbf{J}(\alpha,\beta)=0$, then $\Omega_{\lambda}^\alpha\cap
\Omega_{\lambda}^\beta=\{0,\infty\}$. The following proposition
tells us the number of intersection points in the general case.

\begin{pro}\label{3h} Let $\alpha,\beta\in \Theta$ with $\Omega_{\lambda}^\alpha\neq
\Omega_{\lambda}^\beta$; then, the intersection
$\Omega_{\lambda}^\alpha\cap\Omega_{\lambda}^\beta$ consists of
$2^{\mathbf{J}(\alpha,\beta)+1}$ points.
\end{pro}
\begin{proof} We consider the orbit of
$\Omega_{\lambda}^\alpha\cap\Omega_{\lambda}^\beta$ under
$f_\lambda$:
$$\Omega_{\lambda}^\alpha\cap\Omega_{\lambda}^\beta\rightarrow
 \Omega_{\lambda}^{\tau(\alpha)}\cap\Omega_{\lambda}^{\tau(\beta)}\rightarrow\cdots \rightarrow
 \Omega_{\lambda}^{\tau^{\mathbf{J}(\alpha,\beta)}(\alpha)}\cap\Omega_{\lambda}^{\tau^{\mathbf{J}(\alpha,\beta)}(\beta)}$$
Note that for any $0\leq k\leq
 \mathbf{J}(\alpha,\beta)-1$, $f_\lambda:
 \Omega_{\lambda}^{\tau^{k}(\alpha)}\cap\Omega_{\lambda}^{\tau^{k}(\beta)}
 \rightarrow\Omega_{\lambda}^{\tau^{k+1}(\alpha)}\cap\Omega_{\lambda}^{\tau^{k+1}(\beta)}$
 is a two-to-one map; thus, we have
 $$\#(\Omega_{\lambda}^\alpha\cap\Omega_{\lambda}^\beta)=2\#(\Omega_{\lambda}^{\tau(\alpha)}\cap\Omega_{\lambda}^{\tau(\beta)})
 =\cdots=2^{\mathbf{J}(\alpha,\beta)}\#(\Omega_{\lambda}^{\tau^{\mathbf{J}
 (\alpha,\beta)}(\alpha)}\cap\Omega_{\lambda}^{\tau^{\mathbf{J}(\alpha,\beta)}(\beta)})=2^{\mathbf{J}(\alpha,\beta)+1}.$$
\end{proof}

\begin{rem}
From the proof of proposition \ref {3h}, we know that any two
distinct cut rays $\Omega_{\lambda}^\alpha$ and
$\Omega_{\lambda}^\beta$ intersect at the preimages of $\infty$.
More precisely,
$\Omega_{\lambda}^\alpha\cap\Omega_{\lambda}^\beta\subset
\bigcup_{0\leq k\leq
\mathbf{J}(\alpha,\beta)+1}f_\lambda^{-k}(\infty)$, and for $2\leq
k\leq \mathbf{J}(\alpha,\beta)+1$, the intersection
$\Omega_{\lambda}^\alpha\cap\Omega_{\lambda}^\beta\cap
f_\lambda^{-(k-1)}(0)$ consists of $2^{k-1}$ points.
\end{rem}

\section{Puzzles, Graphs and Tableaux}

\subsection{The Yoccoz Puzzle}

Let $X_\lambda=\mathbb{\bar{C}}\setminus\{z\in B_\lambda;
G_\lambda(z)\geq 1\}=\mathbf{V}(1)$.
 Given $N$ periodic angles
$\theta_1,\cdots,\theta_N$ that lie in different periodic cycles of
$\Theta$, let
$$g_\lambda(\theta_1,\cdots,\theta_N)=\bigcup_{k\geq0}\Big(\Omega_\lambda^{\tau^k(\theta_1)}
\cup \cdots\cup\Omega_\lambda^{\tau^k(\theta_N)}\Big).$$ Obviously,
$g_\lambda(\theta_1,\cdots,\theta_N)$ is $f_\lambda$-invariant. The
graph $\mathbf{G}_\lambda(\theta_1,\cdots,\theta_N)$ generated by
$\theta_1,\cdots,\theta_N$  is defined as follows:
$$\mathbf{G}_\lambda(\theta_1,\cdots,\theta_N)=\partial X_\lambda \cup \Big(X_\lambda\cap g_\lambda(\theta_1,\cdots,\theta_N)\Big).$$

The Yoccoz Puzzle induced by the graph
$\mathbf{G}_\lambda(\theta_1,\cdots,\theta_N)$ is constructed in the
following way. The Yoccoz Puzzle of depth zero consists of all
connected components of $X_\lambda\setminus
\mathbf{G}_\lambda(\theta_1,\cdots,\theta_N)$, and each component is
called a puzzle piece of depth zero. The Yoccoz Puzzle of greater
depth can be constructed by induction as follows: if
$P_d^{(1)},\cdots,P_d^{(m)}$ are the puzzle pieces of depth $d$,
then the connected components of the set $f_\lambda^{-1}(P_d^{(i)})$
are the puzzle pieces $P_{d+1}^{(j)}$ of depth $d+1$. One can verify
that the puzzle pieces of depth $d$ consist of all connected
components of $f_\lambda^{-d}(X_\lambda\setminus
\mathbf{G}_\lambda(\theta_1,\cdots,\theta_N))$ and that each puzzle
piece is a disk.

In applying the Yoccoz puzzle theory, we should avoid a situation in
which the critical orbits touch the set
$\mathbf{G}_\lambda(\theta_1,\cdots,\theta_N)$. If the critical
orbits touch the graph
$\mathbf{G}_\lambda(\theta_1,\cdots,\theta_N)$, we say the graph
$\mathbf{G}_\lambda(\theta_1,\cdots,\theta_N)$ is touchable. In this
case, we cannot find a sequence of shrinking puzzle pieces such that
each piece contains a critical point in its interior (that is to
say, we cannot find a non-degenerate critical annulus that plays a
crucial role in the Yoccoz puzzle theory). For this reason, because
there are infinite periodic angles in $\Theta$, we can change the
$N$-tuple $(\theta_1,\cdots,\theta_N)$ to another $N$-tuple
$(\theta'_1,\cdots,\theta'_N)$ to make the graph not touchable.

\begin{figure}
\begin{center}
\includegraphics[height=7cm]{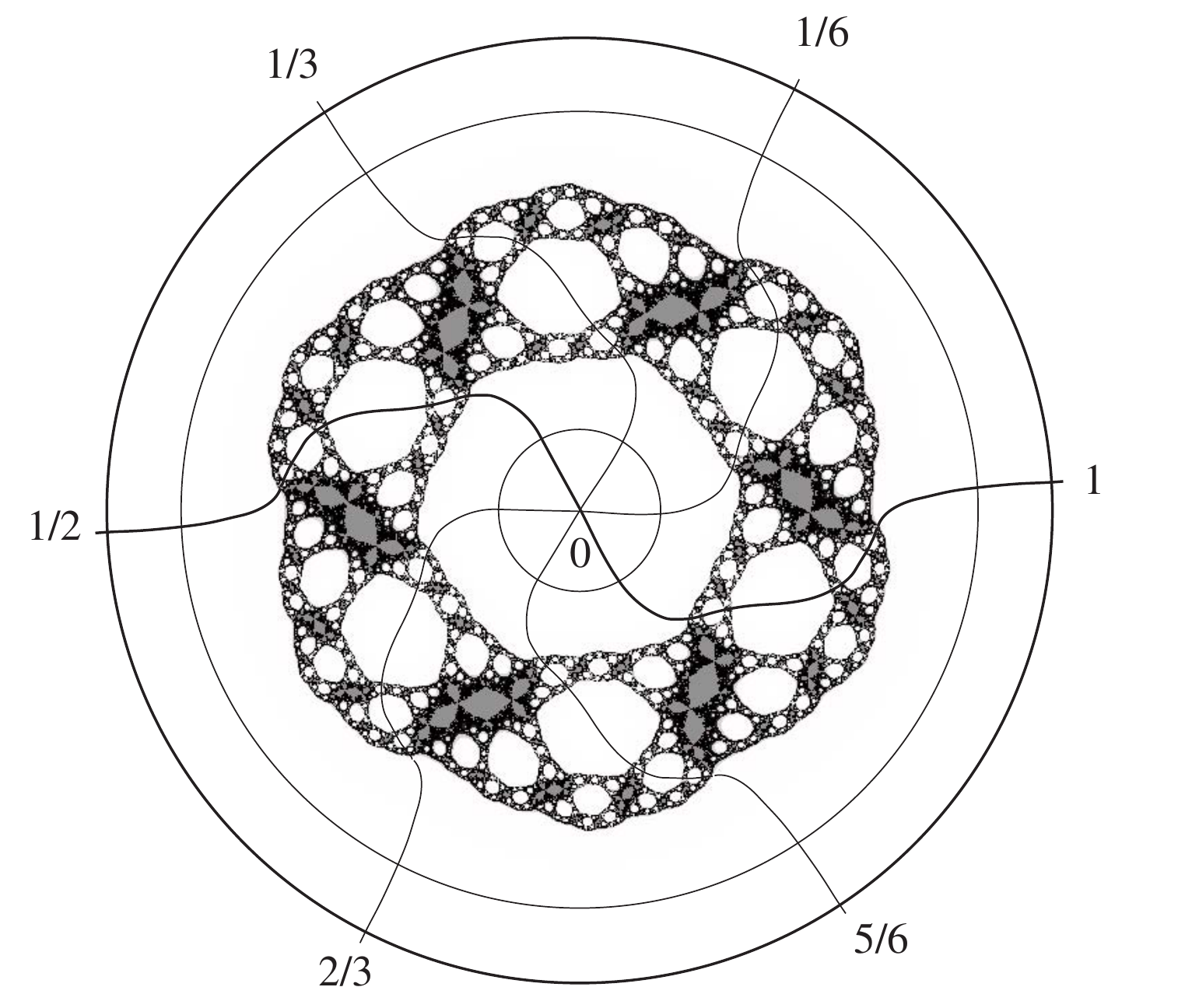}
\caption{A graph with Yoccoz puzzle to depth one ($n=3$ and
$\mathbf{G}_\lambda=\mathbf{G}_\lambda(1/2)$).}
\end{center}\label{f5}
\end{figure}

 Let $J_0$ be the set of all points on the Julia set $J(f_\lambda)$ with orbits that eventually meet the graph $\mathbf{G}_\lambda(\theta_1,\cdots,\theta_N)$. Then
$J_0= \bigcup _{k\geq0}
f_\lambda^{-k}(\mathbf{G}_\lambda(\theta_1,\cdots,\theta_N)\cap
J(f_\lambda))$. For any $z\in
\mathbb{\overline{C}}\setminus(A_\lambda\cup J_0)$, there is a
unique sequence of puzzle pieces $P_0(z)\supset P_1(z)\supset
P_2(z)\supset \cdots$ that contain $z$. By Proposition \ref {3f}, if
$f_\lambda$ has a non-repelling cycle in $\mathbb{C}$, say
$\mathcal{C}=\{z,f_\lambda(z),\cdots, f_\lambda^p(z)=z\}$, then this
cycle must avoid the graph
$\mathbf{G}_\lambda(\theta_1,\cdots,\theta_N)$. This implies that
$\mathcal{C}\subset\mathbb{\overline{C}}\setminus(A_\lambda\cup
J_0)$. Thus, for any $d\geq0$ and any $x\in \mathcal{C}$, the puzzle
piece $P_d(x)$ is well defined.

\begin{lem}\label{4a} Suppose the graph
$\mathbf{G}_\lambda(\theta_1,\cdots,\theta_N)$ is not touchable,
then for any $z\in \mathbb{\overline{C}}\setminus(A_\lambda\cup
J_0)$, the puzzle pieces satisfy:
$$-{P}_0
(z)={P}_0 (-z);\ \  \omega {P}_d (z)={P}_d (\omega z),\
\omega^{2n}=1,\ d\geq 1.$$
\end{lem}
\begin{proof} By the definition of the graph
$\mathbf{G}_\lambda(\theta_1,\cdots,\theta_N)$ and the symmetry of
the Green function $G_\lambda: A_\lambda\rightarrow (0,+\infty]$
(see Lemma \ref{1a}), we have $X_\lambda\setminus
\mathbf{G}_\lambda(\theta_1,\cdots,\theta_N)=-X_\lambda\setminus
\mathbf{G}_\lambda(\theta_1,\cdots,\theta_N)$. Thus $-{P}_0
(z)={P}_0 (-z)$. Suppose that for some $d\geq 0$,
$$f_\lambda^{-d}(X_\lambda\setminus
\mathbf{G}_\lambda(\theta_1,\cdots,\theta_N))=-f_\lambda^{-d}(X_\lambda\setminus
\mathbf{G}_\lambda(\theta_1,\cdots,\theta_N)).$$

Because $f_\lambda(\omega z)=\pm f_\lambda(z)$ and $G_\lambda(\omega
z)=G_\lambda(z)$, we have $f_\lambda(z)\in
f_\lambda^{-{d}}(X_\lambda\setminus
\mathbf{G}_\lambda(\theta_1,\cdots,\theta_N))$ if and only if
$f_\lambda(\omega z)\in f_\lambda^{-{d}}(X_\lambda\setminus
\mathbf{G}_\lambda(\theta_1,\cdots,\theta_N))$. Thus
$$f_\lambda^{-{(d+1)}}(X_\lambda\setminus
\mathbf{G}_\lambda(\theta_1,\cdots,\theta_N))=\omega
f_\lambda^{-{(d+1)}}(X_\lambda\setminus
\mathbf{G}_\lambda(\theta_1,\cdots,\theta_N)).$$

The conclusion follows by induction.
\end{proof}

\begin{lem}\label{4aa} Suppose the graph
$\mathbf{G}_\lambda(\theta_1,\cdots,\theta_N)$ is not touchable,
then for any $d\geq0$ and any puzzle piece $P_d$ of depth $d$, the
intersection $\overline{P}_d\cap J(f_\lambda)$ is connected.
\end{lem}
\begin{proof} It is equivalent to prove that every connected
component of $\mathbb{\overline{C}}\setminus(\overline{P}_d\cap
J(f_\lambda))$ is simply connected. Because the Julia set
$J(f_\lambda)$ is connected, every component of
$\mathbb{\overline{C}}\setminus(\overline{P}_d\cap J(f_\lambda))$
that lies inside $P_d$ is simply connected. Therefore, we only need
to consider the components of
$\mathbb{\overline{C}}\setminus(\overline{P}_d\cap J(f_\lambda))$
that intersect with $\partial P_d$. Note that the puzzle piece $P_d$
is bounded by finitely many cut rays, say
$\Omega_\lambda^{\beta_1},\cdots, \Omega_\lambda^{\beta_s}$,
together with finitely many equipotential curves
$\mathbf{e}(U_1,v),\cdots, \mathbf{e}(U_t,v)$. By the structure of
cut rays (Proposition \ref{3e}), there is exactly one component of
$\mathbb{\overline{C}}\setminus(\overline{P}_d\cap J(f_\lambda))$
that intersects with the boundary $\partial P_d$. This component is
the union of $\mathbb{\overline{C}}\setminus\overline{P}_d$ and
countably many Fatou components that intersect with the cut rays
$\Omega_\lambda^{\beta_1},\cdots, \Omega_\lambda^{\beta_s}$. Thus,
it is also simply connected.
\end{proof}

\subsection{Admissible graphs}

Given the  point $z\in \mathbb{\overline{C}}\setminus(A_\lambda\cup
J_0)$, the difference set $A_d(z)=P_d(z)\setminus
\overline{P_{d+1}(z)}$ is an annulus, either degenerate or of
positive modulus. Here, $d$ is called the depth of $A_d(z)$. For
$d\geq1$ and $c\in C_\lambda$, the annulus $A_d(z)$ is called
off-critical, $c$-critical or $c$-semi-critical if $P_d(z)$ contains
no critical points, $P_{d+1}(z)$ contains the critical point $c$ or
$A_d(z)$ contains the critical point $c$, respectively.

 Because the critical annuli play a crucial role in our discussion, we will devote ourselves to finding a graph such that with respect to the Yoccoz puzzle induced by such a graph, the critical annulus $A_d(c)$ is non-degenerate for some $d\geq1$. By Lemma \ref{4a}, if some critical annulus $A_d(c)$ of depth $d\geq1$ is non-degenerate, then all critical
annuli of the same depth are non-degenerate. The graph that
satisfies this property is of special interest.

\begin{defi} We say the graph
$\mathbf{G}_\lambda(\theta_1,\cdots,\theta_N)$ is admissible if it
is not touchable and if with respect to the Yoccoz puzzle induced by
$\mathbf{G}_\lambda(\theta_1,\cdots,\theta_N)$ there exists a
non-degenerate critical annulus $A_d(c)$ for some critical point
$c\in C_\lambda$ and some depth $d\geq 1$. Otherwise, we say the
graph $\mathbf{G}_\lambda(\theta_1,\cdots,\theta_N)$ is
non-admissible.
\end{defi}

By definition, a non-admissible graph either is touchable or
contains no non-degenerate critical annulus of depth greater than
one with respect to its induced Yoccoz puzzle. In the definition of
an admissible graph, we require that the critical annulus $A_d(c)$
is non-degenerate for some depth $d\geq 1$ rather than $d=0$ because
the puzzle pieces of depth zero have only two-fold symmetry and the
puzzle pieces of depth greater than zero have $2n$-fold symmetry
(see Lemma \ref{4a}).

The following remark tells us that a graph may be non-admissible in
some cases.

\begin{rem}
There exist non-admissible graphs. For example, for any $n\geq3$,
suppose $f_\lambda$ is 1-renormalizable at $c_0$ (see Section 5 for
definition). Then, the graph $\mathbf{G}_\lambda(1)$ is
non-admissible because $A_d(c_0)$ is degenerate for all depths
$d\geq 1$ (see Figure 6). One should note that $A_0(c_1)$ is
non-degenerate and $A_d(c_1)=e^{\pi i/3}A_d(c_0)$ is degenerate for
all $d\geq1$.
\end{rem}

However, even if non-admissible graphs exist, we can always find an
admissible graph based on an elaborate choice. The aim of this
section is to prove the existence of admissible graphs for $n\geq3$.

\begin{pro}\label{4ab} For any $n\geq3$ and any $\lambda\in \mathcal{H}$, if $f_\lambda$ is not critically finite, then there always exists an  admissible graph.
\end{pro}

The proof is divided into three lemmas: Lemma \ref{4b}, Lemma
\ref{4b1} and Lemma \ref{4b2}. In fact, these lemmas enable us to
prove much more: when $n\geq5$, there always exist infinitely many
admissible graphs 
$f_\lambda^{k+1}(\sqrt[2n]{\lambda})=f_\lambda^{k}(\sqrt[2n]{\lambda})$
or
$f_\lambda^{k+2}(\sqrt[2n]{\lambda})=f_\lambda^{k}(\sqrt[2n]{\lambda})$
for some $k\geq1$).

\begin{lem}\label{4b} When $n=3$, there exists an admissible graph except when the critical orbit of $f_\lambda$ eventually lands at a repelling cycle of period one or two. More precisely,

1. If neither $\mathbf{G}_\lambda(1/4)$ nor
$\mathbf{G}_\lambda(1/2)$ is touchable, then at least one of the
graphs $\mathbf{G}_\lambda(1/4)$, $\mathbf{G}_\lambda(1/2)$,
$\mathbf{G}_\lambda(1/4,1/2)$ is admissible.

2. If $\mathbf{G}_\lambda(1/2)$ is touchable, then either
$\mathbf{G}_\lambda(1/4)$ is admissible or the critical orbit of
$f_\lambda$ eventually lands at a repelling cycle of period two.

3. If $\mathbf{G}_\lambda(1/4)$ is touchable, then either
$\mathbf{G}_\lambda(1/2)$ is admissible or the critical orbit of
$f_\lambda$ eventually lands at a repelling fixed point.
\end{lem}
\begin{proof} First, note that
$$f_\lambda^{-1}(\Omega_\lambda^{1/4})=\Omega_\lambda^{1/12}\cup\Omega_\lambda^{1/4}\cup\Omega_\lambda^{5/12},\
f_\lambda^{-1}(\Omega_\lambda^{1/2})=\Omega_\lambda^{1/6}\cup\Omega_\lambda^{1/3}\cup\Omega_\lambda^{1/2}.$$

1.  In this case, the full rays $\omega_\lambda^{1/12}$ and
$\omega_\lambda^{1/6}$ decompose $S_0$ into four domains:
$D_1,D_2,D_3$ and $D_4$ (see Figure 7). If neither
$\mathbf{G}_\lambda(1/4)$ nor $\mathbf{G}_\lambda(1/2)$ is
touchable, then the critical orbit has no intersection with
$\Omega_\lambda^{1/4}\cup\Omega_\lambda^{1/2}$.

 We consider the location of the critical value $v_\lambda^+$; there are four possibilities:

\text{\bf Case 1: $v_\lambda^+\in D_1$.} In this case, the annulus
$A_0(v_\lambda^+)=P_0(v_\lambda^+)\setminus
\overline{P_{1}(v_\lambda^+)}$ is non-degenerate with respect to the
Yoccoz puzzle as induced by either of the graphs
$\mathbf{G}_\lambda(1/4)$, $\mathbf{G}_\lambda(1/2)$ and
$\mathbf{G}_\lambda(1/4,1/2)$. It turns out that the critical
annulus $A_1(c)$ is non-degenerate for all $c\in C_\lambda$. Thus,
in this case, all the graphs $\mathbf{G}_\lambda(1/4)$,
$\mathbf{G}_\lambda(1/2)$, $\mathbf{G}_\lambda(1/4,1/2)$ are
admissible.

\text{\bf Case 2: $v_\lambda^+\in D_2$.} The annulus
$A_0(v_\lambda^+)=P_0(v_\lambda^+)\setminus
\overline{P_{1}(v_\lambda^+)}$ is non-degenerate with respect to the
Yoccoz puzzle induced by the graph $\mathbf{G}_\lambda(1/4)$.
Therefore, all critical annuli $A_1(c)$ are non-degenerate. Thus,
the graph $\mathbf{G}_\lambda(1/4)$ is admissible.

\text{\bf Case 3: $v_\lambda^+\in D_3$.} The annulus
$A_0(v_\lambda^+)$ is non-degenerate with respect to the Yoccoz
puzzle induced by the graph $\mathbf{G}_\lambda(1/4,1/2)$.
Therefore, all critical annuli $A_1(c)$ are non-degenerate, and the
graph $\mathbf{G}_\lambda(1/4,1/2)$ is admissible.

\text{\bf Case 4: $v_\lambda^+\in D_4$.}  Based on an argument
similar to that used above, we conclude that the graph
$\mathbf{G}_\lambda(1/2)$ is admissible.

\begin{figure}
\begin{center}
\includegraphics[height=8cm]{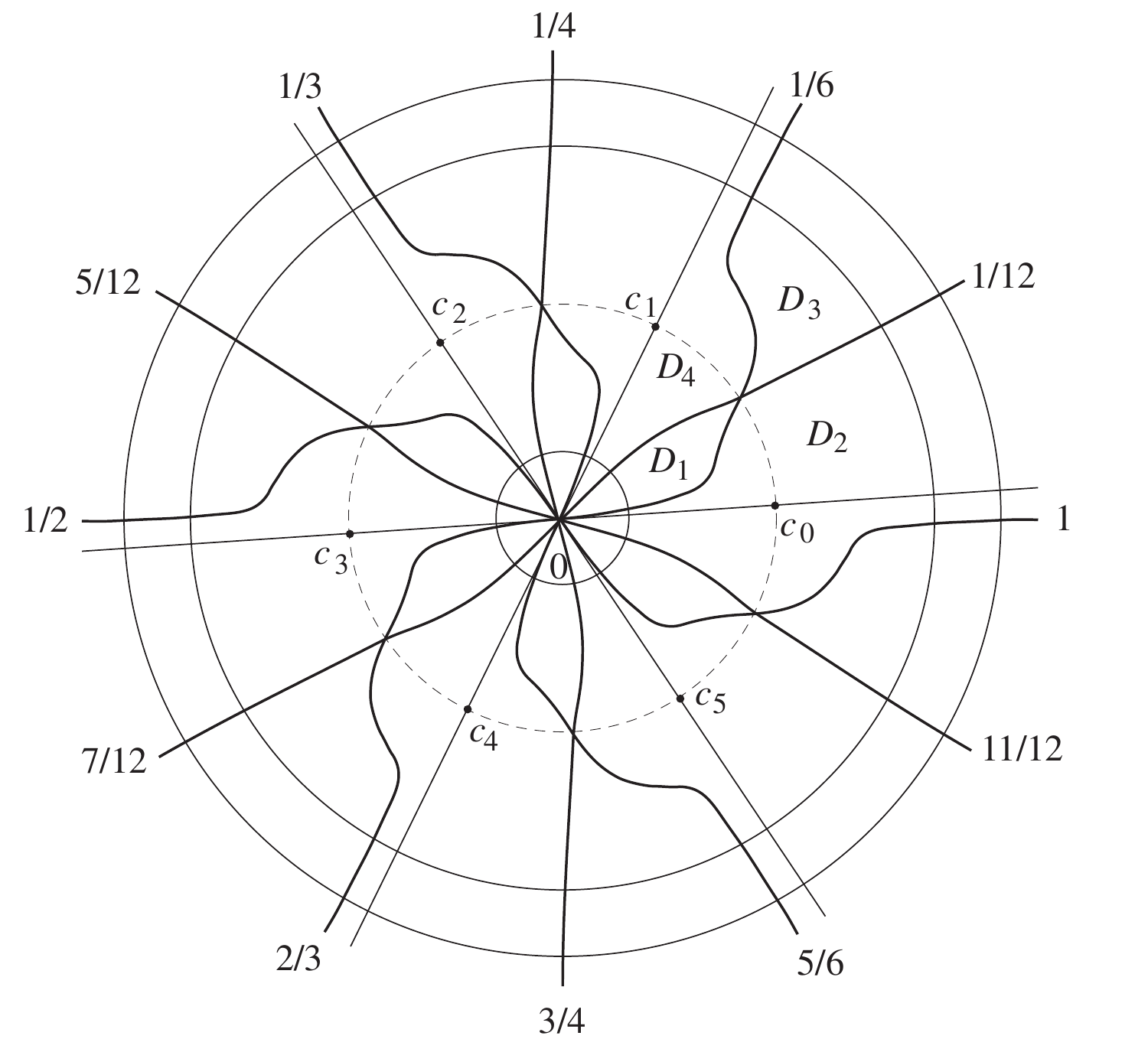}
\caption{Candidates for admissible graph when $n=3$ .}
\end{center}
\end{figure}

2. In this case, the graph $\mathbf{G}_\lambda(1/4)$ is necessarily
untouchable. First, note that the cut ray $\Omega_\lambda^{5/12}$
decomposes $\Omega_\lambda^{1/2}$ into four parts:
$\Omega_\lambda^{1/2}(2,2)$, $\Omega_\lambda^{1/2}(2,-2)$,
$\Omega_\lambda^{1/2}(-2,2)$ and $\Omega_\lambda^{1/2}(-2,-2)$,
where
$$\Omega_\lambda^{1/2}(\epsilon_0,\epsilon_1)=\overline{\{z\in\Omega_\lambda^{1/2}\setminus O_\lambda;
\mathbf{s}_\lambda(z)=(\epsilon_0,\epsilon_1,\pm2,\pm2,\cdots)} \},\
\epsilon_0,\epsilon_1=\pm2.$$ Moreover, for any $z\in
(\Omega_\lambda^{1/2}(2,2)\cup \Omega_\lambda^{1/2}(-2,-2))\cap
J(f_\lambda)$, the annulus $A_0(z)$ is non-degenerate with respect
to the Yoccoz puzzle induced by the graph $\mathbf{G}_\lambda(1/4)$.

Because $\mathbf{G}_\lambda(1/2)$ is touchable, there exist an
integer $p\geq1$ and a critical point $c\in C_\lambda$ such that
$f^p_\lambda(c)\in \Omega_\lambda^{1/2}$. Consider the itinerary of
$f^p_\lambda(c)$, say
$\mathbf{s}_\lambda(f^p_\lambda(c))=(s_0,s_1,s_2,\cdots)$. There are
two possibilities:

\text{\bf Case 1}. There is an integer $n\geq0$ such that
$(s_n,s_{n+1})=(2,2)$ or $(-2,-2)$. In this case,
$f^{n+p}_\lambda(c)\in (\Omega_\lambda^{1/2}(2,2)\cup
\Omega_\lambda^{1/2}(-2,-2))\cap J(f_\lambda)$; thus, the annulus
$A_0(f^{n+p}_\lambda(c))$ is non-degenerate. It turns out that the
critical annulus $A_{n+p}(c)$ is non-degenerate. Therefore, the
graph $\mathbf{G}_\lambda(1/4)$ is admissible.

\text{\bf Case 2}. For any integer $n\geq0$, $(s_n,s_{n+1})=(2,-2)$
or $(-2,2)$. In this case, either
$\mathbf{s}_\lambda(f^p_\lambda(c))=(2,-2,2,-2,\cdots)={(\overline{2,-2})}$
or
$\mathbf{s}_\lambda(f^p_\lambda(c))=(-2,2,-2,2,\cdots)={(\overline{-2,2})}$.
By Proposition \ref {3f}, $f^p_\lambda(c)$ lies in a repelling cycle
of period two.

3. The proof is similar to the proof of $2$. In this case, the graph
$\mathbf{G}_\lambda(1/2)$ is necessarily untouchable. First, note
that the cut ray $\Omega_\lambda^{1/3}$ decomposes
$\Omega_\lambda^{1/4}$ into four parts:
$\Omega_\lambda^{1/4}(1,-1)$, $\Omega_\lambda^{1/4}(1,1)$,
$\Omega_\lambda^{1/4}(-1,-1)$ and $\Omega_\lambda^{1/4}(-1,1)$,
where
$$\Omega_\lambda^{1/4}(\epsilon_0,\epsilon_1)=\overline{\{z\in\Omega_\lambda^{1/4}\setminus O_\lambda;
\mathbf{s}_\lambda(z)=(\epsilon_0,\epsilon_1,\pm1,\pm1,\cdots) \}},\
\ \epsilon_0,\epsilon_1=\pm1.$$ Moreover, for any $z\in
(\Omega_\lambda^{1/4}(1,-1)\cup \Omega_\lambda^{1/4}(-1,1))\cap
J(f_\lambda)$, the annulus $A_0(z)$ is non-degenerate with respect
to the Yoccoz puzzle induced by the graph $\mathbf{G}_\lambda(1/2)$.

Because $\mathbf{G}_\lambda(1/4)$ is touchable, there are an integer
$p\geq1$ and a critical point $c\in C_\lambda$ such that
$f^p_\lambda(c)\in \Omega_\lambda^{1/4}$. Consider the itinerary of
$f^p_\lambda(c)$, say
$\mathbf{s}_\lambda(f^p_\lambda(c))=(s_0,s_1,s_2,\cdots)$. There are
two possibilities:

\text{\bf Case 1}. There is an integer $n\geq0$ such that
$(s_n,s_{n+1})=(-1,1)$ or $(1,-1)$. In this case,
$f^{n+p}_\lambda(c)\in (\Omega_\lambda^{1/4}(1,-1)\cup
\Omega_\lambda^{1/4}(-1,1))\cap J(f_\lambda)$; thus, the annulus
$A_0(f^{n+p}_\lambda(c))$ is non-degenerate. It turns out that the
critical annulus $A_{n+p}(c)$ is non-degenerate. Therefore, the
graph $\mathbf{G}_\lambda(1/2)$ is admissible.

\text{\bf Case 2}. For any integer $n\geq0$, $(s_n,s_{n+1})=(1,1)$
or $(-1,-1)$. In this case, either
$\mathbf{s}_\lambda(f^p_\lambda(c))=(1,1,\cdots)={(\overline{1})}$
or
$\mathbf{s}_\lambda(f^p_\lambda(c))=(-1,-1,\cdots)={(\overline{-1})}$.
By Proposition \ref {3f}, $f^p_\lambda(c)$ is a repelling fixed
point.
\end{proof}

\begin{lem}\label{4b1} When $n=4$, if $\mathbf{G}_\lambda(1/3)$ is
not touchable, then $\mathbf{G}_\lambda(1/3)$ is admissible; if
$\mathbf{G}_\lambda(1/3)$ is touchable, then
$\mathbf{G}_\lambda(2/3,1)$ is admissible.
\end{lem}
\begin{proof}
First, note that $\mathbf{s}(1/3)=(2,2,\cdots)=(\overline{2}),\
\mathbf{s}(2/3)=(-1,-1,\cdots)=(\overline{-1})$ and $
\mathbf{s}(1)=(-3,-3,\cdots)=(\overline{-3})$. Thus,
$\Omega_\lambda^{1/3}\subset S_{2}\cup S_{-2},
\Omega_\lambda^{2/3}\subset S_{1}\cup S_{-1}$ and $
\Omega_\lambda^{1}\subset S_{3}\cup S_{-3}$ (see Figure 8). It is
easy to verify
$$f_\lambda^{-1}(\Omega_\lambda^{1/3})=\Omega_\lambda^{1/12}\cup
\Omega_\lambda^{5/24}\cup\Omega_\lambda^{1/3}\cup\Omega_\lambda^{11/24}.$$

If the graph $\mathbf{G}_\lambda(1/3)$ is not touchable, then the
critical orbit has no intersection with $\Omega_\lambda^{1/3}$. With
respect to the Yoccoz puzzle induced by $\mathbf{G}_\lambda(1/3)$,
the puzzle piece $P_1(v_\lambda^+)$ is a subset of the domain
bounded by $\omega_\lambda^{5/24}$ and $\omega_\lambda^{23/24}$
together with the equipotential curves $\mathbf{e}(B_\lambda,1/n)$
and $\mathbf{e}(T_\lambda,1/n)$. Thus, the annulus
$A_0(v_\lambda^+)$ is non-degenerate. It turns out that all critical
annuli $A_1(c)$ are non-degenerate. Therefore, the graph
$\mathbf{G}_\lambda(1/3)$ is admissible.

\begin{figure}
\begin{center}
\includegraphics[height=8cm]{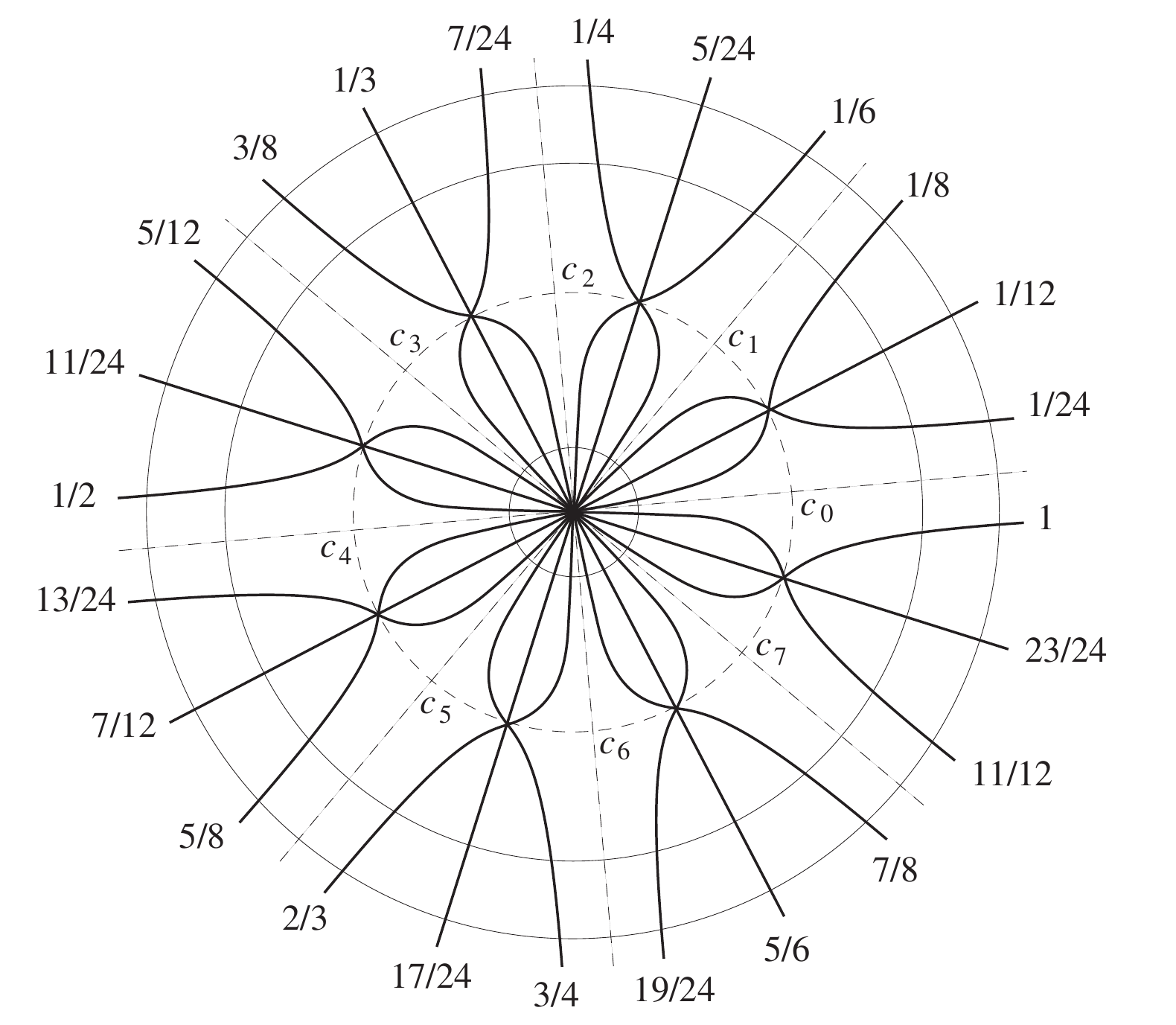}
\caption{Candidates for admissible graph when $n=4$ .}
\end{center}
\end{figure}

If the graph $\mathbf{G}_\lambda(1/3)$ is touchable, then there
exist an integer $p\geq1$ and a critical point $c\in C_\lambda$ such
that $f_\lambda^p(c)\in \Omega_\lambda^{1/3}$. Note that the
preimage of $\Omega_\lambda^{2/3}$ that lies in $S_2\cup S_{-2}$ is
$\Omega_\lambda^{7/24}$ and the preimage of $\Omega_\lambda^{1}$
that lies in $S_2\cup S_{-2}$ is $\Omega_\lambda^{3/8}$. In this
case, with respect to the Yoccoz puzzle induced by the graph
$\mathbf{G}_\lambda(2/3,1)$, the puzzle piece $P_1(f_\lambda^p(c))$
is bounded by $\Omega_\lambda^{7/24}$ and $\Omega_\lambda^{3/8}$;
thus, the annulus $A_0(f_\lambda^p(c))$ is non-degenerate. It
follows that all critical annuli $A_p(c)$ are non-degenerate, and
the graph $\mathbf{G}_\lambda(2/3,1)$ is admissible.
\end{proof}

\begin{figure}
\begin{center}
\includegraphics[height=8cm]{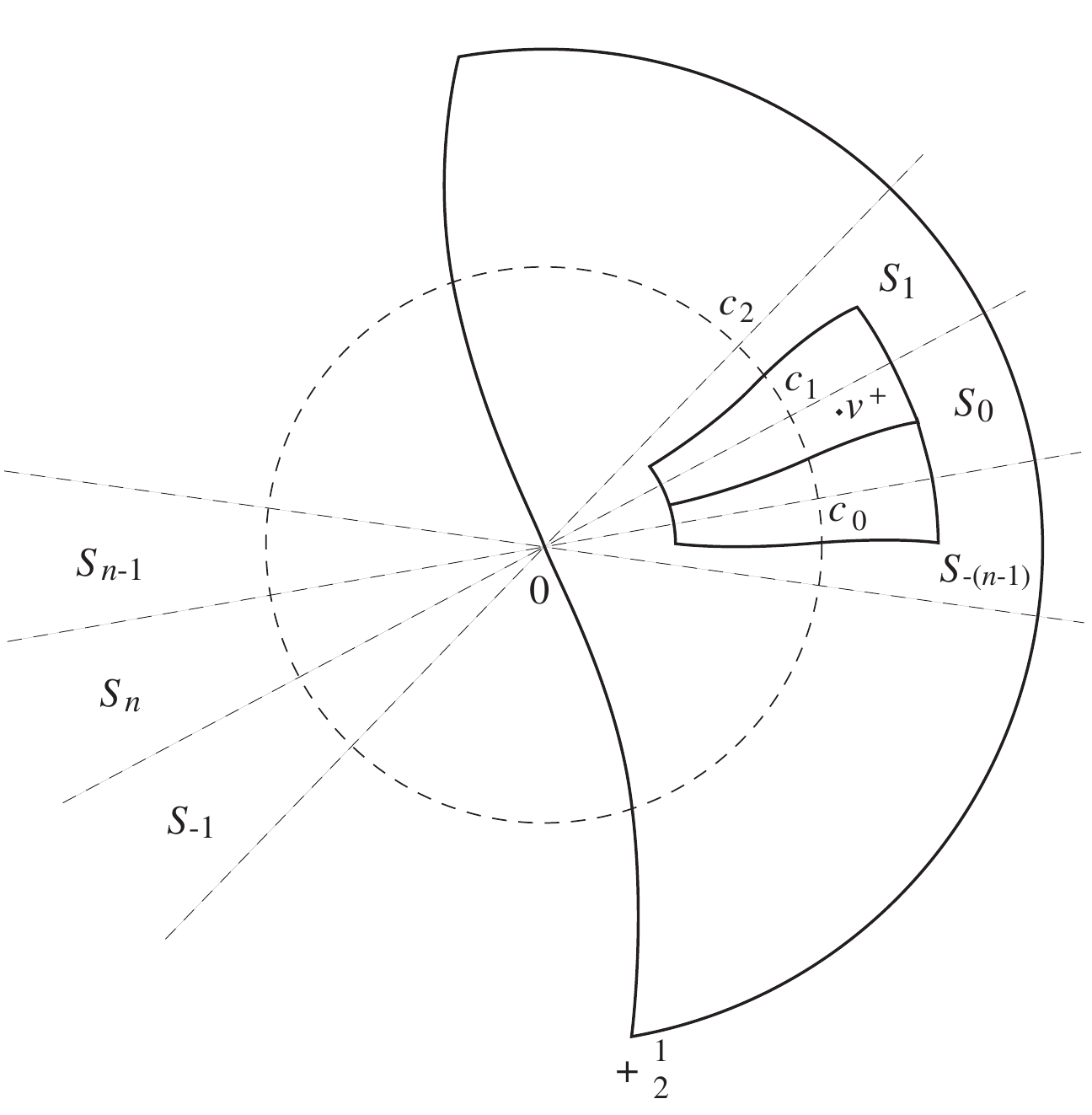}
\caption{Candidates for admissible graph when $n\geq5$ .}
\end{center}
\end{figure}

In the following, we will consider the case when $n\geq5$. Let
$$\widehat{\Theta}= \bigcap_{j\geq 0}\tau^{-j}\Big(\bigcup_{2\leq
k\leq n-2}(\Theta_k\cup\Theta_{-k})\Big)$$ be the set of all angles
in $\Theta$ whose orbits remain in $\bigcup_{2\leq k\leq
n-2}(\Theta_k\cup\Theta_{-k})$ under all iterations of $\tau$, and
let $\widehat{\Theta}_{per}$ be the set of all periodic angles in
$\widehat{\Theta}$. Based on a similar argument as for Lemma
\ref{3b}, we can show that $\widehat{\Theta}_{per}$ is a dense
subset of $\widehat{\Theta}$. By Lemma \ref{3a},one can check that
the set $\widehat{\Theta}_{per}$ can be written as
$$\widehat{\Theta}_{per}=\bigcup_{p\geq1}\{\kappa(\mathbf{s}); \mathbf{s}=
 (\overline{s_0,\cdots, s_{p-1}})\in \Sigma_0  \textrm{\ and\ } s_0,\cdots, s_{p-1}\in \{\pm2, \cdots,\pm (n-2)\}\}$$
and that any angle $\theta\in \widehat{\Theta}_{per}$ is of the form
$$ \theta=\frac{1}{2}\bigg(\frac{\chi(s_0)}{n}+\frac{|s_0|}{n(n^p-1)}+
\frac{n^p}{n^p-1}\sum_{1\leq k < p}\frac{|s_{k}|}{n^{k+1}}\bigg).$$
\begin{lem}\label{4b2} When $n\geq 5$, there are infinitely many
periodic angles $\theta\in \Theta$ such that the graph
$\mathbf{G}_\lambda(\theta)$ is admissible.
\end{lem}
\begin{proof}

We can choose an angle $\theta\in \widehat{\Theta}_{per}$ such that
the critical orbit avoids the graph $\mathbf{G}_\lambda(\theta)$.
(Note that there are infinitely many such choices of angle
$\theta$). When $n\geq5$, the set
$\cup_{j\geq0}\Omega_\lambda^{\tau^j(\theta)}-\{0,\infty\}$ lies
outside $S_1\cup S_0\cup S_{-(n-1)}$ (see Figure 9). Then, with
respect to the Yoccoz puzzle induced by the graph
$\mathbf{G}_\lambda(\theta)$, $\overline{P_1(v_\lambda^+)}$ is
contained in the interior of $S_1\cup S_0\cup S_{-(n-1)}$ and is a
proper subset of $P_0(v_\lambda^+)$. Because
$f_\lambda(P_2(c_0))=P_1(v_\lambda^+)$ and
$f_\lambda(P_1(c_0))=P_0(v_\lambda^+)$, we know that
$A_1(c_0)=P_1(c_0)\setminus \overline{P_2(c_0)}$ is non-degenerate.
Thus, the graph $\mathbf{G}_\lambda(\theta)$ is admissible.
\end{proof}

In the remainder of this section, we prove an important property of
the cut rays that are used to generate admissible graphs. Let
\begin{equation*}
\Theta_{ad}=\begin{cases}
 \{\frac{1}{4},\frac{1}{2}\},\ \  &n=3,\\
 \{\frac{1}{3},\frac{2}{3},1\},\ \  &n=4,\\
 \widehat{\Theta}_{per},\ \  &n\geq5.
 \end{cases}
 \end{equation*}
 Note that for any admissible graph
 $\mathbf{G}_\lambda(\theta_1,\cdots,\theta_N)$ constructed by Lemma \ref{4b}, Lemma
\ref{4b1} and Lemma \ref{4b2}, $\{\theta_1,\cdots,\theta_N\}\subset
 \Theta_{ad}$. In the following, we will prove

\begin{pro}\label{4b3} For any $\theta\in \Theta_{ad}$, the
intersection $\Omega_\lambda^\theta\cap \partial B_\lambda$ consists
of two points.
\end{pro}

The proof is based on the following:

\begin{lem}\label{4b4} Suppose $\theta\in\Theta$, and $\theta$
satisfies one of the following conditions:

C1. There are two sequences, $\{\theta_k^+\}_{k\geq1},
\{\theta_k^-\}_{k\geq1}\subset\Theta$ such that for all $k\geq1$,
$\theta_k^-<\theta<\theta_k^+$ and
$\mathbf{J}(\theta_k^+,\theta)=\mathbf{J}(\theta_k^-,\theta)\rightarrow\infty$
as $k\rightarrow\infty$.

C2. There is a sequence $\{\theta_k\}_{k\geq1}\subset\Theta$ such
that $\theta_1<\theta_2<\theta_3<\dots$ (or
$\theta_1>\theta_2>\theta_3>\dots$) and
$\mathbf{J}(\theta_k,\theta)=k$ for any $k\geq1$.

Then the intersection $\Omega_\lambda^\theta\cap \partial B_\lambda$
consists of two points.
\end{lem}
\begin{proof}

1. Suppose $\theta$ satisfies C1 and
$\mathbf{s}(\theta)=(s_0,s_1,s_2,\cdots)$. By Proposition \ref{3h},
the cut rays $\Omega_\lambda^{\theta_k^+}$ and
$\Omega_\lambda^{\theta_k^-}$ both intersect with
$\Omega_\lambda^{\theta}$ at $2^{\mathbf{J}(\theta_k^+,\theta)+1}$
Points; they hence decompose $\Omega_\lambda^{\theta}$ into
$2^{\mathbf{J}(\theta_k^+,\theta)+1}$ parts:
$$\Omega_\lambda^{\theta}(\epsilon_0,\epsilon_1,\cdots,
\epsilon_{\mathbf{J}(\theta_k^+,\theta)}), \epsilon_j=\pm s_j, 0\leq
j\leq \mathbf{J}(\theta_k^+,\theta).$$
 Here
$\Omega_\lambda^{\theta}(\epsilon_0,\epsilon_1,\cdots,
\epsilon_{p}):=\overline{\{z\in\Omega_\lambda^{\theta}\setminus
O_\lambda; \mathbf{s}_\lambda(z)=(\epsilon_0,\epsilon_1,\cdots,
\epsilon_{p}, \pm s_{p+1},\pm s_{p+2},\cdots)\}}.$

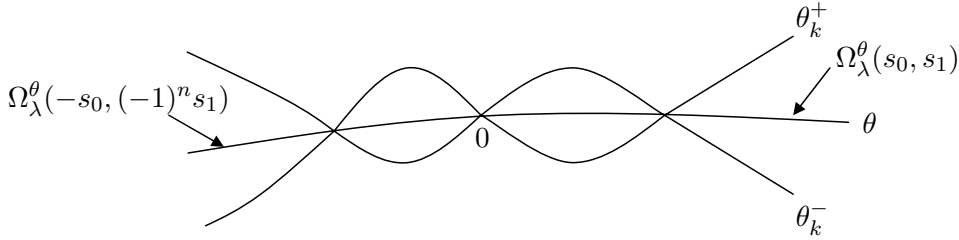
\begin{figure}[h]
\centering{
\includegraphics[height=4cm]{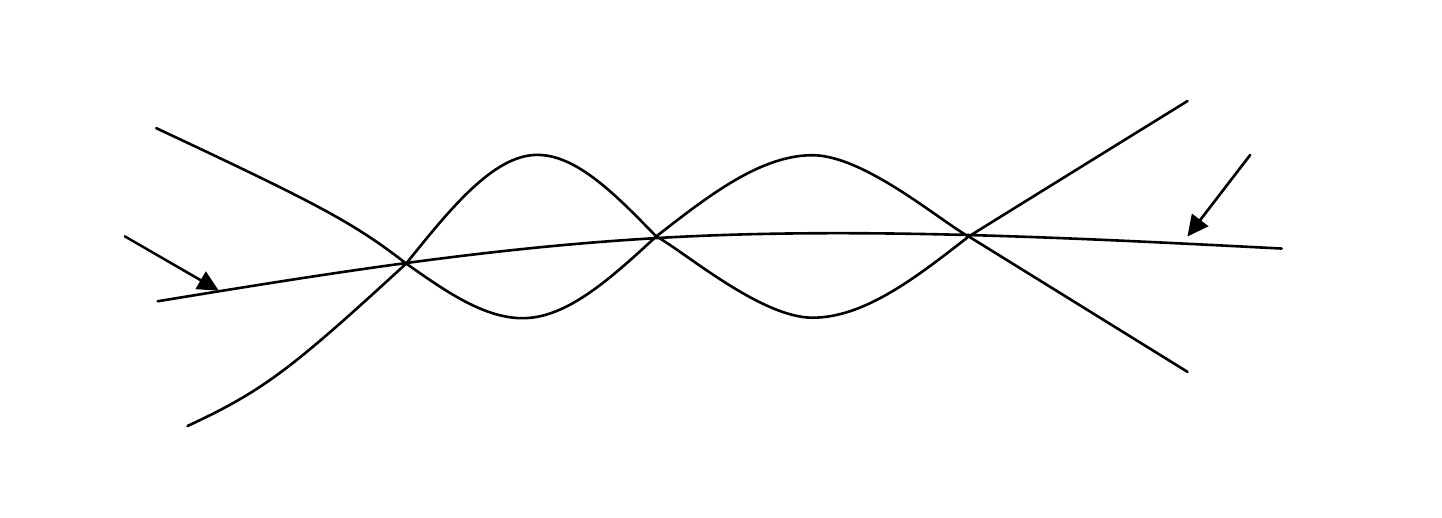}
\put(-55,20){$\theta_k^-$} \put(-30,55){$\theta$}
\put(-350,65){$\Omega_\lambda^{\theta}(-s_0,(-1)^n s_1)$}
\put(-175,50){$0$} \put(-40,80){$\Omega_\lambda^{\theta}(s_0,s_1)$}
\put(-55,95){$\theta_k^+$}
 \caption{Three cuts rays with angles $\theta_k^+>\theta>
 \theta_k^-$. In this figure,
 $\mathbf{J}(\theta_k^+,\theta)=\mathbf{J}(\theta_k^-,\theta)=1$.
 Exactly two segments of
 $\Omega_\lambda^\theta$ intersect with $\overline{B_\lambda}$: $\Omega_\lambda^{\theta}(s_0,s_1)$ and
 $\Omega_\lambda^{\theta}(-s_0,(-1)^ns_1)$.
 }}
\end{figure}

 Based on the structure of the cut rays (Proposition \ref{3e}) and because the angle $\theta$ satisfies condition C1, we conclude that of these
$2^{\mathbf{J}(\theta_k^+,\theta)+1}$ parts, only two intersect with
$\overline{B_\lambda}$: $\Omega_\lambda^{\theta}(s_0,s_1,\cdots,
s_{\mathbf{J}(\theta_k^+,\theta)})$ and
$\Omega_\lambda^{\theta}(-s_0,(-1)^n s_1,\cdots, (-1)^n
s_{\mathbf{J}(\theta_k^+,\theta)})$. We should remark that here we
use two cut rays
$\Omega_\lambda^{\theta_k^+},\Omega_\lambda^{\theta_k^-}$ with
$\mathbf{J}(\theta_k^+,\theta)=\mathbf{J}(\theta_k^-,\theta)$ to
separate the other segments of $\Omega_\lambda^\theta$ from
$\overline{B}_\lambda$ (see Figure 10). Moreover,
$\Omega_\lambda^{\theta}\cap\overline{B_\lambda}\subset\Omega_\lambda^{\theta}(s_0,s_1,\cdots,
s_{\mathbf{J}(\theta_k^+,\theta)})\cup\Omega_\lambda^{\theta}(-s_0,(-1)^n
s_1,\cdots, (-1)^n s_{\mathbf{J}(\theta_k^+,\theta)})$ for any
$k\geq1$. It turns out that \bess
\Omega_\lambda^{\theta}\cap\overline{B_\lambda}&\subset&\bigcap_{k\geq1}\Big(\Omega_\lambda^{\theta}(s_0,s_1,\cdots,
s_{\mathbf{J}(\theta_k^+,\theta)})\cup\Omega_\lambda^{\theta}(-s_0,(-1)^n
s_1,\cdots, (-1)^n s_{\mathbf{J}(\theta_k^+,\theta)})\Big)\\
&=&\{z\in\Omega_\lambda^{\theta};\mathbf{s}_\lambda(z)=(s_0,s_1,s_2,\cdots)
\text{\ or \ } (-s_0,(-1)^ns_1,(-1)^ns_2,\cdots)\}\\
&=&\overline{R_\lambda(\theta)}\cup
\overline{R_\lambda(\theta+1/2)}.\eess

By Proposition \ref{3e}, the intersection $\Omega_\lambda^\theta\cap
\partial B_\lambda$ consists of two points. These two points are the
landing points of the external rays $R_\lambda(\theta)$ and
$R_\lambda(\theta+1/2)$.

2. Now we suppose that $\theta$ satisfies C2 and
$\mathbf{s}(\theta)=(s_0,s_1,s_2,\cdots)$. We only prove the case
when $n$ is odd. The argument applies equally well to the case when
$n$ is even. Let $\{\theta_k\}_{k\geq1}\subset\Theta$ be a sequence
such that $\theta_1<\theta_2<\theta_3<\dots$ and
$\mathbf{J}(\theta_k,\theta)=k$ for any $k\geq1$. The following
facts are straightforward:

\textbf{Fact 1.} Let $z\in\Omega_\lambda^\theta$. If the itinerary
$\mathbf{s}_\lambda(z)$ is of the form
$(\epsilon_0,\cdots,\epsilon_k,s_{k+1},s_{k+2},\cdots)$ or
$(\epsilon_0,\cdots,\epsilon_k,-s_{k+1},-s_{k+2},\cdots)$ for some
$k\geq0$, then $\mathbf{s}_\lambda(f_\lambda^{k+1}(z))=\pm
(s_{k+1},s_{k+2},\cdots)=\mathbf{s}(\tau^{k+1}(\theta))\text{ or
}\mathbf{s}(\tau^{k+1}(\theta)+\frac{1}{2})$. By Proposition
\ref{3d}, $f_\lambda^{k+1}(z)\in
\overline{R_\lambda(\tau^{k+1}(\theta))}\cup
\overline{R_\lambda(\tau^{k+1}(\theta)+\frac{1}{2})}$. Thus, $z$
lies in the closure of some external ray or radial ray
$R_U(\theta_U)$ for $U\in \mathcal{P}$.

\textbf{Fact 2.} For any $k>1$, $\overline{B_\lambda}$ has no
intersection with any bounded component of
$\mathbb{\overline{C}}\setminus\bigcup_{1\leq j\leq
k}\Omega_\lambda^{\theta_j}$; see Figure 11. (The proof is almost
immediate from Proposition \ref{3d}.)

\begin{figure}[h]
\centering{
\includegraphics[height=4cm]{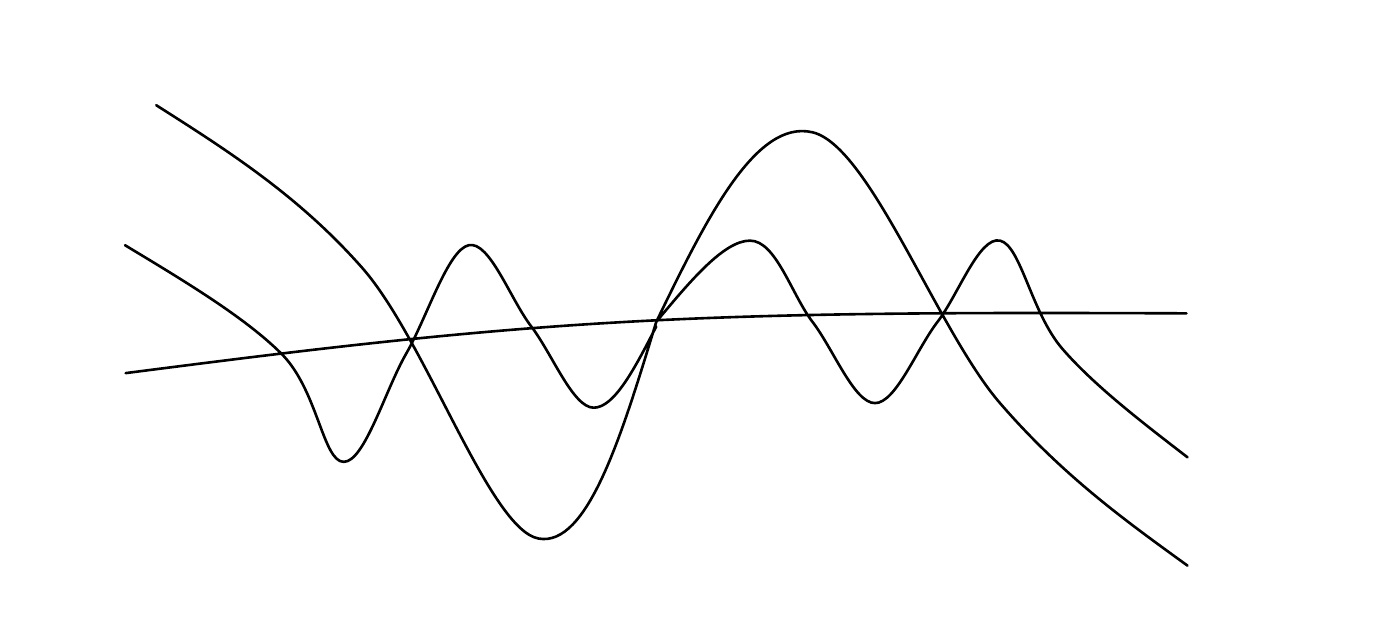}
\put(-30,10){$\theta_1$} \put(-30,45){$\cdots$}
\put(-30,30){$\theta_2$} \put(-30,60){$\theta$} \put(-100,55){$W_2$}
\put(-128,45){$0$} \put(-165,40){$W_1$}
 \caption{Cut rays with angles $\theta_1<\theta_2<\cdots<\theta$, $\mathbf{J}(\theta,\theta_1)=1, \mathbf{J}(\theta,\theta_2)=2,\cdots$.
 Moreover, $\overline{B_\lambda}$ has no intersection with the bounded components $W_1$ and $W_2$ of
$\mathbb{\overline{C}}\setminus(\Omega_\lambda^{\theta_1}\cup\Omega_\lambda^{\theta_2})$.}}
\end{figure}
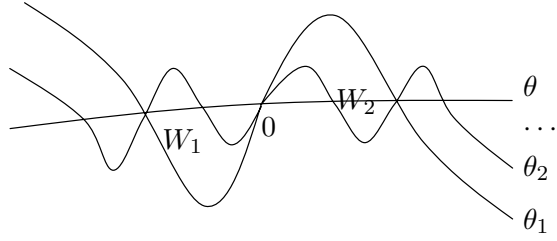

\textbf{Fact 3.} The sections of $\Omega_\lambda^\theta$ that
intersect with the unbounded component of
$\mathbb{\overline{C}}\setminus\bigcup_{1\leq j\leq
k}\Omega_\lambda^{\theta_j}$ are as follows: \bess
&\Omega_\lambda^\theta(s_0,\cdots,s_{k}),\
\Omega_\lambda^\theta(-s_0,\cdots,-s_{k}),&\\
&\Omega_\lambda^\theta(s_0,\cdots,s_{j},-s_{j+1},\cdots,-s_k),\
\Omega_\lambda^\theta(-s_0,\cdots,-s_{j},s_{j+1},\cdots,s_k),\ 0\leq
j<k.& \eess Let $\mathcal{E}_k$ be the collection of these sections.

Based on Facts 2 and 3, we have $\overline{B_\lambda}\cap
\Omega_\lambda^\theta\subset\bigcup_{E\in \mathcal{E}_k}E$ for any
$k>1$. It follows that $\overline{B_\lambda}\cap
\Omega_\lambda^\theta\subset\bigcap_{k>1}\bigcup_{E\in
\mathcal{E}_k}E=\{z\in \Omega_\lambda^\theta, \mathbf{s}_\lambda(z)$
$ \text{ is of the form } \pm \mathbf{s}(\theta)$ $ \text{ or }
 \pm(s_0,s_1,\cdots,s_k,-s_{k+1},$ $-s_{k+2},\cdots)  \text{ for some
 }k\geq0\}$.

By Fact 1, for any $z\in \overline{B_\lambda}\cap
\Omega_\lambda^\theta$, either $z\in
\overline{R_\lambda(\theta)}\cup \overline{R_\lambda(\theta+1/2)}$
or there exist $U\in\mathcal{P}\setminus\{B_\lambda\}$ and an angle
$\theta_U$ such that $z\in \overline{R_U(\theta_U)}$. In the
following, we show that the latter is impossible. In fact, if $z\in
\overline{B_\lambda}\cap\Omega_\lambda^\theta\cap\overline{R_U(\theta_U)}$,
then $z\in \partial B_\lambda\cap\partial U$. Let $p\geq0$ be the
first integer such that $f_\lambda^p(U)=T_\lambda$.

After $p$ iterations, we see that $f_\lambda^p(z)\in\partial
B_\lambda\cap\partial T_\lambda$ and $f_\lambda^p(z)$ is the landing
point of the radial ray
$R_{T_\lambda}(\theta_{T_\lambda})=f_\lambda^p(R_U(\theta_U))$. On
the other hand, $f_\lambda^{p+1}(z)$ is the landing point of the
external ray
$R_{\lambda}(\theta_{\lambda})=f_\lambda^{p+1}(R_U(\theta_U))$.
Therefore, $f_\lambda^p(z)$ is also a landing point of some external
ray $R_{\lambda}(\beta)$, $\beta\in \tau^{-1}(\theta_{\lambda})$.
Because both $R_{T_\lambda}(\theta_{T_\lambda})$ and
$R_{\lambda}(\beta)$ land at
 $f_\lambda^p(z)$, and $f_\lambda(R_{T_\lambda}(\theta_{T_\lambda}))
 =f_\lambda(R_{\lambda}(\beta))=R_{\lambda}(\theta_{\lambda})$, $f_\lambda^p(z)$ is necessarily a critical point in $C_\lambda$.

  However, the result that $f_\lambda^p(z)\in f_\lambda^p(\Omega_\lambda^\theta)\cap C_\lambda$ leads to a contradiction because for any
$\alpha\in\Theta$, the cut ray $\Omega_\lambda^\alpha$ avoids the
critical set $C_\lambda$.

Now, we are in the situation $\overline{B_\lambda}\cap
\Omega_\lambda^\theta\subset\overline{R_\lambda(\theta)}\cup
\overline{R_\lambda(\theta+1/2)}$, and the conclusion follows.
\end{proof}

\noindent\textit{Proof of Proposition \ref{4b3}.} It suffices to
verify that for any $\theta\in \Theta_{ad}$, $\theta$ satisfies
either C1 or C2 by Lemma \ref{4b4}.

When $n=3$, $\mathbf{s}(1/4)=(\overline{1,-1}),\
\mathbf{s}(1/2)=(\overline{2})$. Define two sequences of angles
$\{\alpha_k\}_{k\geq1}, \{\beta_k\}_{k\geq1}\subset \Theta$ such
that \bess \mathbf{s}(\alpha_1)&=&(1,-2,-1,1,-1,1,\cdots),\
\mathbf{s}(\beta_1)=(2,1,-1,2,2,2,\cdots),\\
\mathbf{s}(\alpha_2)&=&(1,-1,2,1,-1,1,\cdots),\ \ \
\mathbf{s}(\beta_2)=(2,2,1,-1,2,2,\cdots),\\
\mathbf{s}(\alpha_3)&=&(1,-1,1,-2,-1,1\cdots),\
\mathbf{s}(\beta_3)=(2,2,2,1,-1,2,\cdots),\\
&&\cdots
 \eess
 Then, $\alpha_1>\alpha_2>\alpha_3>\cdots$ and
$\mathbf{J}(\alpha_k,1/4)=k$ for any $k\geq1$;
$\beta_1<\beta_2<\beta_3<\cdots$ and $\mathbf{J}(\beta_k,1/2)=k$.
Thus, both $1/4$ and $1/2$ satisfy condition C2.

When $n=4$, $\mathbf{s}(1/3)=(\overline{2}),\
\mathbf{s}(2/3)=(\overline{-1}),\ \mathbf{s}(1)=(\overline{-3})$.
Define three sequences of angles $\{\alpha_k\}_{k\geq1},
\{\beta_k\}_{k\geq1},\{\gamma_k\}_{k\geq1}\subset \Theta$ such that
\bess \mathbf{s}(\alpha_1)&=&(2,1,-2,2,2,\cdots),\
\mathbf{s}(\beta_1)=(-1,-3,-1,-1,\cdots), \mathbf{s}(\gamma_1)=(-3,-1,-3,-3,\cdots),\\
 \mathbf{s}(\alpha_2)&=&(2,2,1,-2,2,\cdots),\
\mathbf{s}(\beta_2)=(-1,-1,-3,-1,\cdots),\mathbf{s}(\gamma_2)=(-3,-3,-1,-3,\cdots),\\
 \mathbf{s}(\alpha_3)&=&(2,2,2,1,-2,\cdots),\
\mathbf{s}(\beta_3)=(-1,-1,-1,-3,\cdots),\mathbf{s}(\gamma_3)=(-3,-3,-3,-1,\cdots),\\
&&\cdots \eess
 Then $\alpha_1<\alpha_2<\alpha_3<\cdots$ and
$\mathbf{J}(\alpha_k,1/3)=k$; $\beta_1>\beta_2>\beta_3>\cdots$ and
$\mathbf{J}(\beta_k,2/3)=k$; $\gamma_1<\gamma_2<\gamma_3<\cdots$ and
$\mathbf{J}(\gamma_k,1)=k$. Thus, $1/3,2/3,1$ all satisfy condition
C2.

When $n\geq5$, we can prove that for any
$\theta\in\widehat{\Theta}_{per}$, $\theta$ satisfies condition C1.
(In fact, this is true for all $\theta\in \widehat{\Theta}$).) The
proof is as follows. Suppose
$\mathbf{s}(\theta)=(s_0,s_1,s_2,\cdots)$. For any $k\geq1$, we
choose $s_{k}^-,s_{k}^+\in \{\pm1,\pm(n-1)\}$ and
$s_{k+1}^-,s_{k+1}^+\in \mathbb{I}\setminus\{0,n\}$ such that

(1) $|s_{k}^-|<|s_{k}|<|s_{k}^+|$,

(2)
$(s_0,\cdots,s_{k-1},s_{k}^-,s_{k+1}^-,s_{k+2},s_{k+3},\cdots),(s_0,\cdots,,s_{k-1},s_{k
}^+,s_{k+1}^+,s_{k+2},s_{k+3},\cdots)\in \Sigma_0$. Let
\bess\theta_k^+&=&\kappa((s_0,\cdots,s_{k-1},s_{k}^
+,s_{k+1}^+,s_{k+2},s_{k+3},\cdots)),\\
\theta_k^-&=&\kappa((s_0,\cdots,s_{k-1},s_{k}^-,s_{k+1}^-,s_{k+2},s_{k+3},\cdots)).\eess
It is easy to check that $\theta_k^-<\theta <\theta_k^+$ and
$\mathbf{J}(\theta_k^+,\theta)=\mathbf{J}(\theta_k^-,\theta)=k\rightarrow\infty$
as $k\rightarrow\infty$.
 \hfill $\Box$

\subsection{Modified puzzle piece}

 Consistent with the idea of the `thickened puzzle piece' used in \cite{M2} to study the quadratic Julia set, we construct the `modified puzzle piece' for McMullen maps. The `modified puzzle piece' can be used to study the local connectivity of $J(f_\lambda)$ in the non-renormalizable case (see
Lemma \ref{7a}). It is also used to define renormalizations (see
Remark \ref{5ba}).

Given an angle $\theta\in\Theta$ with itinerary
$\mathbf{s}(\theta)=(s_0,s_1,s_2,\cdots)$, the cut ray
$\Omega_\lambda^\theta$ is identified as
$\Omega_\lambda^\theta=\bigcap_{k\geq0}f^{-k}_\lambda(S_{s_k}\cup
S_{-s_k})$; it can be approximated by the sequence of compact sets
$\{\Omega_{\lambda,m}^\theta=\bigcap_{0\leq k\leq
m}f^{-k}_\lambda(S_{s_k}\cup S_{-s_k})\}_{m\geq0}$ in Hausdorff
topology.
 Now, we consider the set $\mathbb{\bar{C}}\setminus
 \Omega_{\lambda,m}^\theta$.
The open set $\mathbb{\bar{C}}\setminus \Omega_{\lambda,m}^\theta$
consists of two connected components, and the boundary of each
component is a Jordan curve. Denote these two boundary curves by
$\gamma_{\lambda,m}^1(\theta)$ and $\gamma_{\lambda,m}^2(\theta)$.
Let
$V_m(\theta)=\gamma_{\lambda,m}^1(\theta)\cap\gamma_{\lambda,m}^2(\theta)$
be the intersection of these two curves. It is obvious that
$V_m(\theta)$ consists of finitely many points and that
$V_m(\theta)=\Omega_\lambda^\theta \cap\Big(\bigcup_{0\leq k\leq
m+1}f_\lambda^{-k}(\infty)\Big)$. For any $v\in V_m(\theta)$, let
$D(v)$ be the connected component of $\{z\in A_\lambda;
G_\lambda(z)>1\}$ that contains $v$. Obviously, $D(v)$ is a disk.

\begin{figure}
\centering{
\includegraphics[height=6cm]{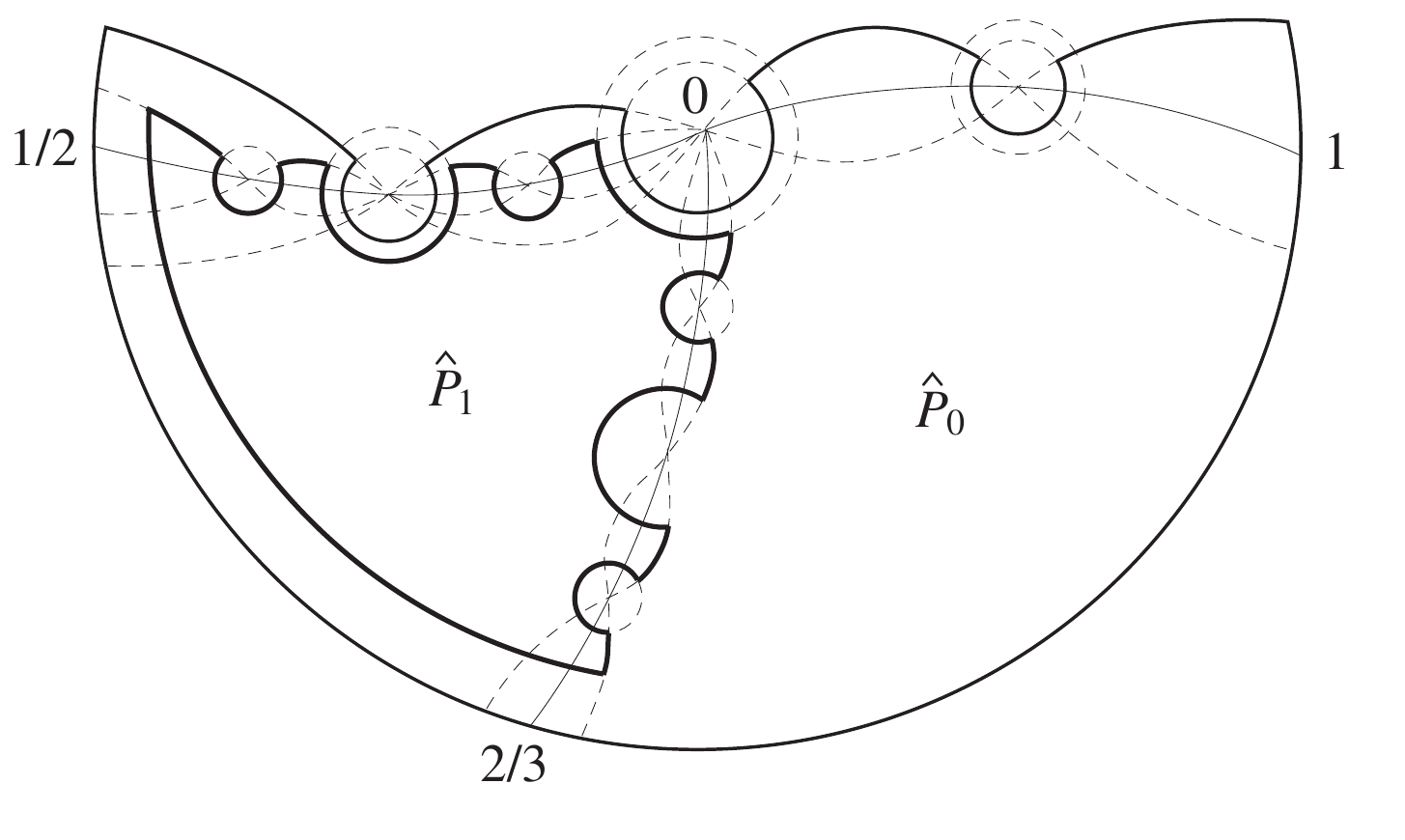}
\caption{An example of `modified puzzle pieces', to depth one.}
\label{Fig5} }
\end{figure}

In the following, we construct the `modified puzzle piece'. For the
Yoccoz puzzle induced by the graph
$\mathbf{G}_\lambda(\theta_1,\cdots,\theta_N)$, recall that each
puzzle piece $P_0$ of depth zero is contained in a unique component
of $\mathbb{\bar{C}}\setminus g_\lambda(\theta_1,\cdots,\theta_N)$.
This component is simply connected and is denoted by $Q_0$. We may
choose a $m$ large enough so that for any $\alpha,\beta\in
\{\tau^k(\theta_j); 1\leq j\leq N, k\geq 0\}$ with
$\Omega_\lambda^\alpha\neq\Omega_\lambda^\beta$,
$$\Omega_{\lambda,m}^\alpha\cap\Omega_{\lambda,m}^\beta=\Omega_{\lambda}^\alpha\cap\Omega_{\lambda}^\beta. $$
The disk $Q_0$ is bounded by some collection of cut rays, say
$\{\Omega_{\lambda}^\alpha; \alpha\in \Lambda(Q_0)\}$, where
$\Lambda(Q_0)$ is an index set induced by $Q_0$. For any $\alpha\in
\Lambda(Q_0)$, choose a curve
$\gamma(\alpha)\in\{\gamma_{\lambda,m}^1(\alpha),\gamma_{\lambda,m}^2(\alpha)\}$
such that $\gamma(\alpha)\cap Q_0=\emptyset$. Let $\widehat{Q}_0$ be
the connected component of $\mathbb{\bar{C}}\setminus
\bigcup_{\alpha\in \Lambda(Q_0)}\gamma(\alpha)$ that contains $Q_0$,
and let $V({Q}_0)=\bigcup_{\alpha\in
\Lambda(Q_0)}(V_m(\alpha)\cap\partial Q_0)$. The modified puzzle
piece $\widehat{P}_0$ of ${P}_0$ is defined as follows:
$$\widehat{P}_0=\widehat{Q}_0-\bigcup_{v\in V({Q}_0)}\overline{D(v)}.$$
Roughly speaking, we can obtain $\widehat{P}_0$ from $Q_0$ by
thickening $Q_0$ near $\partial Q_0\setminus V({Q}_0)$ and
truncating $Q_0$ near the points in $ V({Q}_0)$.  The puzzle piece
$P_0$ is not contained in $\widehat{P}_0$; for this reason, we call
$\widehat{P}_0$ the `modified puzzle piece' of ${P}_0$ rather than
the `thickened puzzle piece' of $P_0$.

Modified puzzle pieces of greater depth can be constructed by the
usual inductive procedure; if $\widehat{P}_d^{(j)}$ is the modified
puzzle piece of depth $d$, then each component of
$f_\lambda^{-1}(\widehat{P}_d^{(j)})$ is the modified puzzle piece
of depth $d+1$ (see Figure 12).

 The advantage of these modified puzzle pieces is as follows: if a puzzle piece ${P}_d^{(j)}$ contains ${P}_{d+1}^{(k)}$, then the modified puzzle
piece $\widehat{P}_d^{(j)}$ contains
$\overline{\widehat{P}_{d+1}^{(k)}}$, which can be easily proved by
induction. In other words, this construction replaces all of our
annuli with non-degenerate annuli.

For $z\in \mathbb{\overline{C}}\setminus(A_\lambda\cup J_0)$, let
$\widehat{P}_d (z)$ be the modified puzzle piece of ${P}_d (z)$. We
will only make use of modified puzzle pieces that are small enough
to satisfy the following additional restriction: if $\widehat{P}_d
(z)$ contains a critical point, then ${P}_d (z)$ must already
contain this critical point. Note that if the graph
$G_\lambda(\theta_1,\cdots,\theta_N)$ is not touchable, then this
requirement is easily satisfied for any bounded value of depth $d$
by choosing $m$ large enough, which will suffice for the
applications.

Based on construction, the puzzle piece ${P}_d (z)$ and the modified
puzzle piece $\widehat{P}_d (z)$ satisfy the following relation:
$$\overline{{P}_d (z)}\subset \widehat{P}_d (z)\cup A_\lambda,\ \
\bigcap_{d\geq 0}\overline{{P}_d (z)}\subset\bigcap_{d\geq
0}\widehat{P}_d (z).$$

The modified puzzle pieces also satisfy the following symmetry
properties: For any $z\in
\mathbb{\overline{C}}\setminus(A_\lambda\cup J_0)$,
$$-\widehat{P}_0
(z)=\widehat{P}_0 (-z);\ \  \omega\widehat{P}_d (z)=\widehat{P}_d
(\omega z),\ \omega^{2n}=1,\ d\geq 1.$$

\subsection{Tableaux}

In this section, we present some basic information on tableaux,
based on Milnor's Lecture \cite{M2}. Applications of tableaux
analysis combined with puzzle techniques can be found in \cite
{BH},\cite{H}, \cite{M2}, \cite{PQRTY}, \cite {QY}, \cite{R},
\cite{RY}, \cite{YZ} and many other papers.

Recall that $J_0$ is the set of all points on $J(f_\lambda)$ with
orbits that eventually touch the graph
$\mathbf{G}_\lambda(\theta_1,\cdots,\theta_N)$. For $x\in
\mathbb{\overline{C}}\setminus(A_\lambda\cup J_0)$, the tableau
$T(x)$ is defined as the two-dimensional array
$(P_{d,l}(x))_{d,l\geq0}$, where
$P_{d,l}(x)=f_\lambda^l(P_{d+l}(x))=P_{d}(f_\lambda^l(x))$. The
position $(d,l)$ is called critical if $P_{d,l}(x)$ contains a
critical point in $C_\lambda$. If $P_{d,l}(x)$ contains a critical
point $c\in C_\lambda$, the position $(d,l)$ is called a
$c$-position. 

For any $x\in \mathbb{\overline{C}}\setminus(A_\lambda\cup J_0)$,
the tableau $T(x)$ satisfies the following three rules:

(T1)\  For each column $l\geq0$, either the position $(d,l)$ is
critical for all $d\geq0$ or there is a unique integer $d_0\geq0$
such that the position $(d,l)$ is critical for all $d<d_0$ and not
critical for $d\geq d_0$.

(T2)\  If $P_{d,l}(x)=P_d(y)$ for some $y\in
\mathbb{\overline{C}}\setminus(A_\lambda\cup J_0)$, then
$P_{i,l+j}(x)=P_{i,j}(y)$ for $0\leq i+j\leq d$.

(T3)\  Let $T(c)$ be a tableau with $c\in C_\lambda$. Assume

 (a)
$P_{d+1-l,l}(c)=P_{d+1-l}(c')$ for some critical point $c'\in
C_\lambda$, $0\leq l<d$, and $P_{d-i,i}(c)$ contains no critical
points for $0<i<l$;

(b) $P_{d,m}(x)=P_{d}(c)$ and $P_{d+1,m}(x)\neq P_{d+1}(c)$ for some
$m>0$.

Then, $P_{d+1-l,m+l}(x)\neq P_{d+1-l}(c')$.

\begin{rem}
The tableau rule (T3) is based on the fact that every puzzle piece
of depth $d\geq1$ contains at most one critical point in
$C_\lambda$.
\end{rem}

\begin{defi}
1. The tableau $T(x)$ is non-critical if there is an integer
$d_0\geq0$ such that $(d_0,j)$ is not critical for all $j>0$.
Otherwise, $T(x)$ is called critical. (One should be careful to note
that $T(x)$ is critical does not mean $x\in C_\lambda$.)

2. The tableau $T(x)$ is called pre-periodic if there exist two
integers $l\geq0$ and $p\geq1$ such that $P_{d,l+p}(x)=P_{d,l}(x)$
for all $d\geq0$. In this case, if $l=0$, $T(x)$ is called periodic,
and the smallest integer $p\geq1$ is called the period of $T(x)$.

3. Let ${\rm Row}_c(d)$ be the $d$-th row of the tableau $T(c)$ with
$c\in C_\lambda$. We say ${\rm Row}_c(d+l)$ with $l>0$ is a child of
${\rm Row}_c(d)$ if there is a critical point $c'\in C_\lambda$ such
that $A_d(f_\lambda^l(c))=A_d(c')$ and
$f_\lambda^l:A_{d+l}(c)\rightarrow A_{d}(c')$ is a degree two
covering map.

4. Let $c\in C_\lambda$. For $d\geq1$, we say ${\rm Row}_c(d)$ is
excellent if $A_d(f_\lambda^l(c))$ is not semi-critical for all
$l\geq0$.

\end{defi}

\begin{rem}
By Lemma \ref{4a} and the fact that $f_\lambda^k(\omega z)=\pm
f_\lambda^k(z)$ for $k\geq1, \omega^{2n}=1$, we have

1. If $(d,l)$ is a critical position for some tableau $T(c)$ with
$c\in C_\lambda$, then $(d,l)$ is a critical position of $T(c')$ for
every $c'\in C_\lambda$.

2. If there is $c\in C_\lambda$ such that the tableau $T(c)$ is
critical, non-critical or pre-periodic, then for every $c'\in
C_\lambda$, the tableau $T(c')$ is critical, non-critical or
pre-periodic, respectively.

3. If ${\rm Row}_c(d)$ is excellent or has a child ${\rm
Row}_c(d+l)$ for some critical point $c\in C_\lambda$, then for
every $c'\in C_\lambda$, ${\rm Row}_{c'}(d)$ is excellent or has a
child ${\rm Row}_{c'}(d+l)$, respectively.
\end{rem}

\begin{lem}\label{4c} Suppose some tableau $T(c)$ with $c\in C_\lambda$ is critical but not pre-periodic, then

1. For every $d\geq1$, ${\rm Row}_c(d)$ has at least one child.

2. If ${\rm Row}_c(d)$ is excellent, then ${\rm Row}_c(d)$ has at
least two children.

3. If ${\rm Row}_c(d)$ is excellent and ${\rm Row}_c(d+l)$ is its
child, then ${\rm Row}_c(d+l)$ is also excellent.

4. If ${\rm Row}_c(d)$ has only one child, say ${\rm Row}_c(d+l)$,
then ${\rm Row}_c(d+l)$ is excellent.

\end{lem}
\begin{proof}
1. By hypothesis, for every $d\geq1$, we can find a smallest integer
$l>0$ such that the annulus $A_d(f_\lambda^l(c))$ is $c'$-critical
for some $c'\in C_\lambda$. The map
$f_\lambda^l:A_{d+l}(c)\rightarrow A_{d}(c')$ is a degree two
covering map, which implies that ${\rm Row}_c(d+l)$ is a child of
${\rm Row}_c(d)$.

2. Following 1, there exists $d'>d$ such that the annulus
$A_{d'}(f_\lambda^l(c))$ is $c'$-semi-critical. Because ${\rm
Row}_c(d)$ is excellent, by tableau rule (T3),
$A_{d'-t}(f_\lambda^{l+t}(c))$ is either off-critical or
semi-critical for $0<t\leq d'-d$. In particular,
$A_{d}(f_\lambda^{l+d'-d}(c))$ is off-critical. Hence, we can find a
smallest integer $l'>l+d'-d$ such that the annulus
$A_d(f_\lambda^{l'}(c))$ is critical; therefore, ${\rm Row}_c(d+l')$
is another child of ${\rm Row}_c(d)$.

3. If ${\rm Row}_c(d+l)$ is not excellent, then there is a column
$l'\geq l$ such that $A_{d+l}(f_\lambda^{l'}(c))$ is semi-critical.
By tableau rule (T3), $A_{d}(f_\lambda^{l+l'}(c))$ is also
semi-critical, which contradicts the fact that ${\rm Row}_c(d)$ is
excellent.

4. If ${\rm Row}_c(d+l)$ is not excellent, then as in (3),
$A_{d}(f_\lambda^{l+l'}(c))$ is semi-critical for some $l'\geq l$.
Suppose $l'\geq l$ is the smallest integer. We can find a smallest
integer $t>l'+l$ such that $A_{d}(f_\lambda^{t}(c))$ is
$c'$-critical for some $c'\in C_\lambda$. Then ${\rm Row}_c(d+t)$ is
also a child of ${\rm Row}_c(d)$, which is a contradiction.
\end{proof}

\begin{lem}\label{4d} Suppose some tableau $T(c)$ with $c\in C_\lambda$ is critical and pre-periodic.

1. If $n$ is odd, then there exist exactly two critical points $\pm
c'\in C_\lambda$ such that $T(c')$ and $T(-c')$ are periodic.

2. If $n$ is even, then there is a unique critical point
$\tilde{c}\in C_\lambda$ such that $T(\tilde{c})$ is periodic.
\end{lem}
\begin{proof} Because $T(c)$ is critical and pre-periodic, there exist a smallest integer $p\geq1$ and a unique critical point $c'\in C_\lambda$ such that $(d,p)$ is a $c'$-position for all $d\geq0$.

1. If $n$ is odd, there are two possibilities: either
$f_\lambda(c)=f_\lambda(c')$ or $f_\lambda(c)+f_\lambda(c')=0$.

If $f_\lambda(c)=f_\lambda(c')$, then both $T(c')$ and $T(-c')$ are
periodic with period $p$. In this case, there is an integer
$d_0\geq0$ such that for any $d\geq d_0, 0<l<p$, the position
$(d,l)$ is not critical. It is easy to check that for any
$\tilde{c}\in C_\lambda\setminus \{\pm c'\}$, the tableau
$T(\tilde{c})$ is strictly pre-periodic. In particular, if $p=1$,
then $P_d(c')=P_d(f_\lambda(c'))$ for all $d\geq0$. This means that
for any $d\geq0$, $c'$ and $f_\lambda(c')$ lie in the same puzzle
piece of depth $d$. Thus, we conclude $\{\pm c'\}=\{c_0,c_n\}$.

If $f_\lambda(c)+f_\lambda(c')=0$, then both $T(c')$ and $T(-c')$
are periodic with period $2p$. Consider the tableau $T(c')$; there
is an integer $d_0\geq0$ such that for any $d\geq d_0, 0<l<p$, the
position $(d,l)$ is not critical and for any $d\geq 0$ the position
$(d,p)$ is $(-c')$-critical. It is easy to confirm that for any
$\tilde{c}\in C_\lambda\setminus \{\pm c'\}$, the tableau
$T(\tilde{c})$ is strictly pre-periodic. In particular, if $p=1$,
then $P_d(-c')=P_d(f_\lambda(c'))$ for all $d\geq0$. Therefore, for
any $d\geq0$, $-c'$ and $f_\lambda(c')$ lie in the same puzzle piece
of depth $d$. Thus, we conclude $\{\pm c'\}=\{c_1,c_{n+1}\}$.

2. $n$ is even. In this case, based on the fact that
$f_\lambda^k(v_\lambda^+)=f_\lambda^k(v_\lambda^-)$ for all
$k\geq1$, we conclude the tableau $T(f_\lambda(c'))$ is periodic
With a period $p$ and the tableau $T(-f_\lambda(c'))$ is strictly
pre-periodic. There is thus a unique critical point $\tilde{c}\in
f_\lambda^{-1}(f_\lambda(c'))$ such that $T(\tilde{c})$ is periodic.
For this tableau, there is an integer $d_0\geq0$ such that for any
$d\geq d_0, 0<l<p$, the position $(d,l)$  is not critical. It is
easy to check that for any $c''\in C_\lambda\setminus
\{\tilde{c}\}$, the tableau $T(c'')$ is strictly pre-periodic. In
particular, if $p=1$ and $T(v_\lambda^+)$ is periodic, then
$\tilde{c}=c_0$; if $p=1$ and $T(v_\lambda^-)$ is periodic, then
$\tilde{c}=c_{n+1}$.
\end{proof}

\section{Renormalizations}

In this section, we discuss the renormalization of McMullen maps
with respect to the puzzle piece.

\begin{defi} If there exist a critical point $c$ of $f_\lambda$, an
integer $p\geq1$ and two disks $U$ and $V$ containing $c$ such that
$$\epsilon f_\lambda^p: U\rightarrow V$$ is a quadratic-like map whose Julia set is connected (here $\epsilon\in \{\pm1\}$ is a symbol), then we say $f_\lambda$ is $p$-renormalizable at $c$ if  $\epsilon=1$ and  $f_\lambda$ is $p$-$*$-renormalizable at $c$ if  $\epsilon=-1$. In the former case, the triple $(f_\lambda^p,U,V)$ is called a $p$-renormalization of $f_\lambda$ at $c$. In the latter case, the triple
$(-f_\lambda^p,U,V)$ is called a $p$-$*$-renormalization of
$f_\lambda$ at $c$.
\end{defi}

In the following, we use $K_c=\{z\in U; (\epsilon
f_\lambda^p)^k(z)\in U, \forall\ k\geq0\}=\bigcap_{k\geq0}(\epsilon
f_\lambda^p)^{-k}(U)$ to denote the small filled Julia set of the
($*$-)renormalization $(\epsilon f_\lambda^p,U,V)$. By the
straightening theorem of Douady and Hubbard \cite{DH2}, if
$(\epsilon f_\lambda^p,U,V)$ is a $p$-($*$-)renormalization of
$f_\lambda$, then $\epsilon f_\lambda^p$ is conjugated by a
quasi-conformal map $\sigma$ to a unique quadratic polynomial
$p_\mu(z)=z^2+\mu$ in a neighborhood of the filled Julia set $K_c$.
Let $\beta$ be the $\beta$-fixed point (i.e., the landing point of
the zero external ray) of $p_\mu$ and $\beta'$ be the other preimage
of $\beta$. We call $\beta_c=\sigma^{-1}(\beta)$ the $\beta$-fixed
point of the renormalization $(\epsilon f_\lambda^p,U,V)$. The other
preimage of $\beta_c$ under the renormalization is
$\beta_c'=\sigma^{-1}(\beta')$ .

In this section, we always assume that the graph
$\mathbf{G}_\lambda(\theta_1,\cdots,\theta_N)$ is admissible.

\subsection{From tableau to renormalizations}

\begin{lem}\label{5a} Suppose some tableau $T(c)$ with $c\in C_\lambda$ is
pre-periodic.

1. If $T(c)$ is non-critical, then $f_\lambda$ is critically finite.

2. If $T(c)$ is critical, then $f_\lambda$ is either renormalizable
or $*$-renormalizable.
\end{lem}
\begin{proof}
 Because $T(c)$ is pre-periodic, there exist two integers $l\geq0$ and $p\geq1$ such that
$P_d(f_\lambda^{l+p}(c))=P_{d,l+p}(c)=P_{d,l}(c)=P_d(f_\lambda^{l}(c))$
for all $d\geq0$.

1. $T(c)$ is  non-critical. In this case, the tableaux
$T(f_\lambda^{l}(c))$ and $T(f_\lambda^{l+p}(c))$ are also
non-critical. Based on Lemma \ref{7a},
$\{f_\lambda^{l+p}(c)\}=\bigcap_{d\geq0}P_d(f_\lambda^{l+p}(c))=
\bigcap_{d\geq0}P_d(f_\lambda^{l}(c))=\{f_\lambda^{l}(c)\}.$
Therefore, $f_\lambda^{l+p}(c)=f_\lambda^{l}(c)$, and $f_\lambda$ is
critically finite.

2. $T(c)$ is  critical. If $n$ is odd, then based on Lemma \ref{4d},
there are exactly two critical points $\pm c'\in C_\lambda$ such
that $T(c')$ and $T(-c')$ are periodic. Suppose the period is $p$,
and consider the tableau $T(c')$. There are two possibilities:

Case 1. There is an integer $d_0\geq0$ such that for any $d\geq d_0,
0<l<p$, the position $(d,l)$ is not critical. Then, $f^p_\lambda:
P_{d_0+p}(c')\rightarrow P_{d_0}(c')$ is a quadratic-like map and
$\{f^{kp}_\lambda(c'); k\geq0\}\subset P_{d_0+p}(c')$. Thus,
$(f^p_\lambda, P_{d_0+p}(c'), P_{d_0}(c'))$ is a $p$-renormalization
of $f_\lambda$ at $c'$. Because $f_\lambda$ is an odd function,
$(f^p_\lambda, P_{d_0+p}(-c'), P_{d_0}(-c'))$ is a
$p$-renormalization of $f_\lambda$ at $-c'$.

Case 2. $p$ is even and there is an integer $d_0\geq0$ such that for
any $d\geq d_0, 0<l<p/2$, the position $(d,l)$ is not critical, and
for any $d\geq 0$, the position $(d,p/2)$ is $(-c')$-critical. Then,
$-f^{p/2}_\lambda: P_{d_0+p/2}(c')\rightarrow P_{d_0}(c')$ is a
quadratic-like map with $\{(-1)^kf^{kp/2}_\lambda(c');
k\geq0\}\subset P_{d_0+p/2}(c')$. Thus, $(-f^{p/2}_\lambda,
P_{d_0+p/2}(c'), P_{d_0}(c'))$ is a $p/2$-$*$-renormalization of
$f_\lambda$ at $c'$. It turns out that $(-f^{p/2}_\lambda,
P_{d_0+p/2}(-c'), P_{d_0}(-c'))$ is a $p/2$-$*$-renormalization of
$f_\lambda$ at $-c'$.

If $n$ is even, then based on Lemma \ref{4d}, there is a unique
critical point $\tilde{c}\in C_\lambda$ such that $T(\tilde{c})$ is
periodic. Suppose the period is $p$; there is then an integer
$d_0\geq0$ such that for any $d\geq d_0, 0<l<p$, the position
$(d,l)$ is not critical. Then, $f^p_\lambda:
P_{d_0+p}(\tilde{c})\rightarrow P_{d_0}(\tilde{c})$ is a
quadratic-like map and $\{f^{kp}_\lambda(\tilde{c}); k\geq0\}\subset
P_{d_0+p}(\tilde{c})$. Thus, $(f^p_\lambda, P_{d_0+p}(\tilde{c}),
P_{d_0}(\tilde{c}))$ is a $p$-renormalization of $f_\lambda$ at
$\tilde{c}$. Because $f_\lambda$ is an even function,
$(-f^p_\lambda, P_{d_0+p}(-\tilde{c}), P_{d_0}(-\tilde{c}))$ is a
$p$-$*$-renormalization of $f_\lambda$ at $-\tilde{c}$.
\end{proof}

\begin{rem}\label{5ba} Lemma \ref{5a} also holds when the graph $\mathbf{G}_\lambda(\theta_1,\cdots,\theta_N)$ is not touchable. Indeed, in this case, we can use modified puzzle pieces to define renormalizations.
\end{rem}

\begin{pro}\label{5b} Suppose $f_\lambda$ has a non-repelling cycle
in $\mathbb{C}$; then $f_\lambda$ is either renormalizable or
$*$-renormalizable. In this situation, there are three
possibilities:

1. If $f_\lambda$ is renormalizable and $n$ is odd, then $f_\lambda$
has exactly two non-repelling cycles in $\mathbb{C}$.

2. If $f_\lambda$ is $*$-renormalizable and $n$ is odd, then
$f_\lambda$ has exactly one non-repelling cycle in $\mathbb{C}$.

3. If $f_\lambda$ is renormalizable and $n$ is even, then
$f_\lambda$ has exactly one non-repelling cycle in $\mathbb{C}$.
\end{pro}
\begin{proof} Let $\mathcal{C}=\{z_0,
f_\lambda(z_0),\cdots,f^q_\lambda(z_0)=z_0\}$ be the non-repelling
cycle of $f_\lambda$  in $\mathbb{C}$.  By Proposition \ref{4ab}, we
can find an admissible graph
$\mathbf{G}_\lambda(\theta_1,\cdots,\theta_N)$. By Proposition
\ref{3f}, the cycle $\mathcal{C}$ avoids the graph
$\mathbf{G}_\lambda(\theta_1,\cdots,\theta_N)$. Thus, for any
$z\in\mathcal{C}$ and any integer $d\geq0$, the puzzle piece
$P_d(z)$ is well-defined.

We claim that there exist $z\in\mathcal{C}$ and a critical point
$c\in C_\lambda$ such that $P_d(z)=P_d(c)$ for all $d\geq0$.
Otherwise, the tableau $T(z)$ is non-critical for any
$z\in\mathcal{C}$. It follows that there is an integer $d_0\geq0$
such that the map $f_\lambda^q: {P}_{d_0+q}(z_0)\rightarrow
{P}_{d_0}(z_0)$ is conformal. Based on the Schwarz lemma,
$|(f_\lambda^q)'(z_0)|>1$, which is a contradiction.

In this way, we can find a critical point $c\in C_\lambda$ with
tableau $T(c)$ that is periodic. Based on Lemma \ref{5a},
$f_\lambda$ is either renormalizable or $*$-renormalizable.

To continue, suppose the period of $T(c)$ is $p$, which is
necessarily a divisor of $q$. Based on Lemma \ref{5a}, there are
three possibilities:

(P1). $n$ is odd and $(f_\lambda^p, P_{d_0+p}(c),P_{d_0}(c))$ is a
$p$-renormalization of $f_\lambda$ at $c$. In this case,
$(f_\lambda^p, P_{d_0+p}(c),P_{d_0}(c))$ is quasi-conformally
conjugate to a polynomial $z\mapsto z^2+\mu$. Because a quadratic
polynomial has at most one non-repelling cycle (see \cite {CG} or
\cite{S}), it turns out that $\mathcal{C}$ is the only non-repelling
cycle contained in $\bigcup_{0\leq j<p}f_\lambda^j(K_c)$. On the
other hand, $-\mathcal{C}$ is the only non-repelling cycle contained
in $\bigcup_{0\leq j< p}f_\lambda^j(-K_c)$. Because there are
exactly two critical points whose tableaux are periodic in this case
and $(\cup_{0\leq j< p}f_\lambda^j(K_c))\cap(\cup_{0\leq j<
p}f_\lambda^j(-K_c))=\emptyset$, we conclude that $f_\lambda$ has
exactly two non-repelling cycles in $\mathbb{C}$.

(P2). $n$ is odd and $(-f_\lambda^{p/2}, P_{d_0+p/2}(c),P_{d_0}(c))$
is a $p/2$-$*$- reorganization of $f_\lambda$ at $c$. In this case,
the cycle $\mathcal{C}$ meets both $K_c$ and $-K_c$.
 By a similar argument as above, one sees that $\mathcal{C}$ is the only non-repelling  cycle contained in $\bigcup_{0\leq j<p}f_\lambda^j(K_c)$. Because the cycle $-\mathcal{C}$ is also contained in
$\bigcup_{0\leq j< p}f_\lambda^j(K_c)$, it turns out that
$\mathcal{C}=-\mathcal{C}$.

(P3). $n$ is even and $(f_\lambda^p, P_{d_0+p}(c),P_{d_0}(c))$ is a
$p$-renormalization of $f_\lambda$ at $c$. In this case, $c$ is the
only critical point whose tableau $T(c)$ is periodic. Based on a
similar argument as made above, we show that $\mathcal{C}$ is the
only non-repelling cycle in $\mathbb{C}$ .
\end{proof}


In the following, we discuss the case when $f_\lambda$ has an
indifferent cycle of multiplier $e^{2\pi i \theta}$. Douady
\cite{Do} conjectured that for any rational map, whenever it is
linearizable (i.e., the map is conformally conjugate to an
irrational rotation) near an indifferent fixed point of multiplier
$e^{2\pi i \theta}$, then $\theta$ must be a Brjuno number. Here, an
irrational number $\theta$ of convergents $p_k/q_k$ (rational
approximations obtained by the continued fraction expansion) is a
Brjuno number (denoted by $\mathcal{B}$) if
$$\sum_{k\geq1}\frac{\log q_{k+1}}{q_k}<+\infty.$$
According to Cremer, Siegel and Brjuno, if $\theta\in\mathcal{B}$,
then every germ $f(z)=e^{2\pi i \theta}z+\mathcal{O}(z^2)$ is
linearizable. Yoccoz \cite{Y} shows that if the quadratic polynomial
$z\mapsto e^{2\pi i \theta}z+z^2$ is linearizable, then
$\theta\in\mathcal{B}$. For a general case, Geyer \cite{G} shows
that for any $d\geq2$, if $z\mapsto z^d+c$ has an indifferent cycle
of multiplier $e^{2\pi i \theta}$ near which the map is
linearizable, then $\theta\in\mathcal{B}$. Based on these results
and Proposition \ref{5b}, we immediately establish:

\begin{pro}\label{5baa} Suppose $f_\lambda$ has an
indifferent cycle of multiplier $e^{2\pi i \theta}$; then
$f_\lambda$ is linearizable near the indifferent cycle if and only
if $\theta\in\mathcal{B}$.
\end{pro}

\subsection{Properties of renormalizations}

In this section, we assume that some tableau $T(c)$ with $c\in
C_\lambda$ is periodic with period $k$. By Lemma \ref{5a},
$f_\lambda$ is either $k$-renormalizable at $c$ or
$k/2$-$*$-renormalizable at $c$. Let $(\epsilon f^p_\lambda,
P_{d_0+p}(c), P_{d_0}(c))$ be the corresponding renormalization,
where
 \begin{equation*}
(\epsilon,p)=\begin{cases}
(1,k),\ \  &\text{ if $f_\lambda$ is $k$-renormalizable at $c$},\\
(-1,k/2),\ \ &\text{ if $f_\lambda$ is  $k/2$-$*$-renormalizable at
$c$}.
 \end{cases}
\end{equation*}
The small filled Julia set $K_c=\bigcap_{d\geq0}
\overline{P_d(c)}=\bigcap_{d\geq0} P_d(c)$.

If $K_c\cap
\partial B_\lambda\neq\emptyset$, we will show that there is a unique
external ray in $B_\lambda$ converging on $K_c$. Before the proof,
we need a classic result for quadratic polynomials:

\begin{lem}\label{5c}  Let $p_\mu(z)=z^2+\mu$ be a quadratic
polynomial with a connected filled Julia set $K$. If there is a
curve $\delta\subset \mathbb{C}\setminus K$ converging to $x\in K$
and $p_\mu(\delta)\supset \delta$, then $x$ is the $\beta$-fixed
point of $p_\mu$.
\end{lem}

Here, a curve $\delta\subset \mathbb{C}\setminus K$ converges to
$x\in K$ means that $\delta$ can be parameterized as $\delta:
[0,1)\rightarrow \mathbb{C}\setminus K$ such that
$\lim_{t\rightarrow1}\delta(t)$ exists and
$\lim_{t\rightarrow1}\delta(t)=x\in K$. See \cite{McM2} for a proof
of Lemma \ref{5c}. The conclusion also holds for quadratic-like
maps.

\begin{lem}\label{5d}  Suppose some tableau $T(c)$ with $c\in C_\lambda$  is $k$-periodic and $K_c\cap
\partial B_\lambda\neq\emptyset$, then

1. The small filled Julia sets $K_c, f_\lambda(K_c), \cdots,
f^{k-1}_\lambda(K_c)$ are pairwise disjoint.

2. There is a unique external ray $R_\lambda(t)$ in $B_\lambda$
accumulating on $K_c$. This external ray lands at $\beta_c\in K_c$
and the angle $t$ is $k$-periodic.
\end{lem}

\begin{proof}
1. \  If $f_\lambda^i(K_c)\cap f_\lambda^j(K_c)\neq \emptyset$ for
some $0\leq i<j<k$, then $K_c\cap f_\lambda^{k+i-j}(K_c)\neq
\emptyset$. Thus, $P_{d,
k+i-j}(c)=f_\lambda^{k+i-j}(P_{d+k+i-j}(c))=P_d(c)$ for all
$d\geq0$. This implies that the tableau $T(c)$ is
$(k+i-j)$-periodic, which is a contradiction.

 2. \ First, note that $f_\lambda^k(P_{d+k}(c))=P_d(c)$ for
$d\geq0$. Because $K_c\cap \partial B_\lambda\neq\emptyset$,
$P_{mk}(c)\cap B_\lambda$ is nonempty and bounded by two external
rays, say $R_\lambda(\theta_m^-)$ and $R_\lambda(\theta_m^+)$ with
$\theta_m^-<\theta_m^+$.  Let
$Q(\theta_m^-,\theta_m^+)=\overline{P_{mk}(c)\cap B_\lambda},\
m\geq1$. Because
$f_\lambda^k(Q(\theta_{m+1}^-,\theta_{m+1}^+))=Q(\theta_m^-,\theta_m^+)$,
we have
$$\theta_m^-\leq\theta_{m+1}^-\leq\cdots\leq\theta_{m+1}^+\leq\theta_{m}^+,\ \
\theta_m^+-\theta_m^-=n^k(\theta_{m+1}^+-\theta_{m+1}^-).$$ Thus,
there is a common limit $t=\lim \theta_m^+=\lim \theta_m^-$. Because
$\theta_m^-\leq t\leq\theta_m^+$ for any $m$, we have $n^kt\equiv t$
(${\rm mod }\  \mathbb{Z}$). Thus, $t$ is a periodic angle and the
external ray $R_\lambda(t)$ lands at a point $z\in K_c\cap
\partial B_\lambda$ (because rational external rays always land). Because $R_\lambda(n^jt)$ lands at $f_\lambda^j(z)\in
f_\lambda^j(K_c)\cap\partial B_\lambda$ for $0\leq j<k$ and the
small filled Julia sets $K_c, f_\lambda(K_c), \cdots,
f^{k-1}_\lambda(K_c)$ are pairwise disjoint, we conclude that the
angles $t, nt,\cdots, n^{k-1}t$ are distinct. Thus, $t$ is
$k$-periodic.

Suppose $\theta$ is another angle such that the external ray
$R_\lambda(\theta)$ accumulates on $K_c$. Then,
$\theta_m^-\leq\theta\leq\theta_m^+$ for any $m$. Thus, $\theta=\lim
\theta_m^+=\lim \theta_m^-=t$.

To finish, we show $z=\beta_c$. Because $T(c)$ is $k$-periodic,
$f_\lambda$ is either $k$-renormalizable or
$k/2$-$*$-renormalizable. In the former case,
$f_\lambda^k(R_\lambda(t))=R_\lambda(t)$. Thus, based on Lemma
\ref{5c}, $z=\beta_c$. In the latter case, because $R_\lambda(t)$ is
the unique external ray accumulating on $K_c$, we conclude that
$R_\lambda(t+1/2)=-R_\lambda(t)$ is the unique external ray
accumulating on $-K_c$. On the other hand,
$f_\lambda^{k/2}(R_\lambda(t))$ is also an external ray accumulating
on $-K_c$, and we have
$f_\lambda^{k/2}(R_\lambda(t))=R_\lambda(t+1/2)=-R_\lambda(t)$. In
this case, $-f_\lambda^{k/2}(R_\lambda(t))=R_\lambda(t)$. Again,
based on Lemma \ref{5c}, $z=\beta_c$.
\end{proof}

\section{ A Criterion of Local Connectivity}

In this section, we present a criterion for the characterization of
the local connectivity of the immediate basin of attraction. This
criterion can be applied together with Yoccoz puzzle techniques to
study the local connectivity and higher regularity of the boundary
$\partial B_\lambda$.

In the following discussion, let $f$ be a rational map of degree at
least two, $C(f)$ be the critical set of $f$ and
$P(f)=\overline{\bigcup_{k\geq1}f^k(C(f))}$ be the postcritical set.
Suppose that $f$ has an attracting periodic point $z_0$ and the
immediate basin $B$ of $z_0$ is simply connected. Let
$B(z,\delta)=\{x\in \mathbb{C}; |x-z|<\delta\}$.
\begin{defi}
We say $f$ satisfies the {\bf{BD}}(bounded degree) condition on
$\partial B$ if for any $u\in \partial B$ there is a number
$\varepsilon_u>0$ such that for any integer $m\geq0$ and any
component $U_m(u)$ of $f^{-m}(B(u,\varepsilon_u))$ intersecting with
$\partial B$, $U_m(u)$ is simply connected and the degree ${\rm
deg}(f^m: U_m(u)\rightarrow B(u,\varepsilon_u))$ is bounded by some
constant $D$ that is independent of $u, m$ and $U_m(u)$.
\end{defi}

The following is a remark on the definition: because $f^m:
U_m(u)\rightarrow B(u,\varepsilon_u)$ is a proper map between two
disks, we conclude by the Maximum Principle that for any disk
$W\subset B(u,\varepsilon_u)$ and any component $V$ of $f^{-m}(W)$
that lies inside $U_m(u)$, $V$ is also a disk.

The aim of this section is to prove the following:

\begin{pro}\label{6a}  If $f$ satisfies the {\bf{BD}} condition on
$\partial B$, then

1. $\partial B$ is locally connected.

2. If, furthermore, $\partial B$ is a Jordan curve, then $\partial
B$ is a quasi-circle.
\end{pro}

 Before presenting the proof, we introduce a distortion lemma. Let $U$ be a hyperbolic disk in $\mathbb{C}$ and $z\in U$. The shape of $U$ about $z$ is defined by:
$${\rm Shape}(U,z)=\sup_{x\in \partial U}|x-z|/\inf_{x\in \partial U}|x-z|.$$

It is obvious that ${\rm Shape} (U,z)=\infty$ if and only if $U$ is
unbounded and ${\rm Shape} (U,z)=1$ if and only if $U$ is a round
disk centered at $z$. In all other cases, $1<{\rm Shape}
(U,z)<\infty.$

Let $K$ be a connected and compact subset of $U$ containing at least
two points. For any $z_1,z_2\in K$, define the turning of $K$ about
$z_1$ and $z_2$ by:
$$\Delta(K;z_1,z_2)={\rm {diam}}(K)/|z_1-z_2|,$$
where ${\rm {diam}}(\cdot)$ is the Euclidean diameter. It is obvious
that $1\leq\Delta(K;z_1,z_2)\leq\infty$ and
$\Delta(K;z_1,z_2)=\infty$ if and only if $z_1=z_2$.

\begin{lem}\label{6b} For $i\in \{1,2\}$, let $(V_i, U_i)$ be a pair of hyperbolic disks in $\mathbb{C}$ with $\overline{U_i}\subset V_i$.  $g:V_1\rightarrow V_2$ is a proper holomorphic map of degree $d$, and $U_1$ is a component of $g^{-1}(U_2)$. Suppose ${\rm mod}(V_2\setminus
\overline{U_2})\geq m>0$. Then,

1. (Shape distortion) There is a constant $C(d,m)>0$ such that for
all $z\in U_1$,
$${\rm Shape} (U_1,z)\leq C(d,m){\rm Shape} (U_2,g(z)).$$

2. (Turning distortion) There is a constant $D(d,m)>0$ such that for
any connected and compact subset $K$ of $U_1$ with $\#K\geq2$ and
any $z_1,z_2\in K$,
$$\Delta(K;z_1,z_2)\leq D(d,m)\Delta(g(K);g(z_1),g(z_2)).$$
\end{lem}
\begin{proof}
A complete proof of 1 can be found in \cite{Wa}, Theorem 2.3.2. In
the following, we prove 2. We assume that $g(z_1)\neq g(z_2)$.
Otherwise, $\Delta(g(K);g(z_1),g(z_2))=\infty$, and the conclusion
follows. Let $\rho (x,y)$ be the hyperbolic distance in $V_2$, and
let $B_1,B_2$ be two hyperbolic disks both centered at $g(z_1)$,
with radii $\max_{\zeta\in g(K)}\rho(g(z_1),\zeta)$ and
$\rho(g(z_1),g(z_2))$, respectively. Let $\varphi: V_2\rightarrow D$
be the Riemann mapping with $\varphi(g(z_1))=0$, and let
$W=\varphi(U_2)$. Because ${\rm mod}(\mathbb{D}\setminus
\overline{W})={\rm mod}(V_2\setminus \overline{U}_2)\geq m$, we
conclude by the Gr\"otzsch Theorem that there is a constant $r(m)\in
(0,1)$ such that $W\subset D_{r(m)}$; here, we use $D_r$ to denote
the disk $\{z; |z|<r\}$.

Note that $\varphi(B_1),\varphi(B_2)$ are two round disks, say $D_R$
and $D_r$, centered at $0$. Based on Koebe distortion, there exist
three constants $C_1(m),C_2(m),C_3(m)>0$ such that
$${\rm Shape}(B_1,g(z_1))\leq C_1(m), \ {\rm Shape}(B_2,g(z_1))\leq C_2(m),$$
$$R/r\leq C_3(m)\max_{\zeta\in g(K)\cap\partial B_1}|g(z_1)-\zeta|/|g(z_1)-g(z_2)|\leq C_3(m)\Delta(g(K);g(z_1),g(z_2)).$$

For $i\in \{1,2\}$, let $W_i$ be the component of $g^{-1}(B_i)$ that
contains $z_1$. Based on the Maximum Principle, $W_1$ and $W_2$ are
simply connected. We may assume that $K\subset \overline{W}_1$
(otherwise, we can replace $B_1$ by $\widehat{B}_1$, a hyperbolic
disk centered at $g(z_1)$ with radius $\epsilon+\max_{\zeta\in
g(K)}\rho(g(z_1),\zeta)$, where $\epsilon$ is a small positive
constant and then let $\epsilon\rightarrow 0^+$). Thus, ${\rm
{diam}}(K)\leq {\rm {diam}}(W_1)\leq 2\sup_{\zeta\in
\partial W_1}|\zeta-z_1|$. Consider the location of $z_2$, which by the Maximum Principle is either $z_2\in \partial W_2$ or $z_2\in U_1\setminus
\overline{W}_2$. In either case, $|z_1-z_2|\geq \inf_{\zeta\in
\partial W_2}|\zeta-z_1|$. Thus, by Shape distortion,
\bess \Delta(K;z_1,z_2)&\leq& 2\sup_{\zeta\in
\partial W_1}|\zeta-z_1|/\inf_{\zeta\in
\partial W_2}|\zeta-z_1|\\
&=& 2{\rm Shape}(W_1,z_1) {\rm Shape}(W_2,z_1) Q(W_1,W_2,z_1)\\
&\leq& C_1(d,m){\rm Shape}(B_1,g(z_1)) {\rm Shape}(B_2,g(z_1))
Q(W_1,W_2,z_1)\\
&\leq& C_2(d,m)Q(W_1,W_2,z_1)\eess where
$Q(W_1,W_2,z_1)=\inf_{\zeta\in
\partial W_1}|\zeta-z_1|/\sup_{\zeta\in \partial W_2}|\zeta-z_1|$. To finish, in the following we show that there is a constant $c(m)>0$
such that $$Q(W_1,W_2,z_1)\leq c(m)\Delta(g(K);g(z_1),g(z_2)).$$

In fact, we only need to consider the case $Q(W_1,W_2,z_1)>1$. In
this case, the annulus $W_1\setminus \overline{W}_2$ contains the
round annulus $\{w\in \mathbb{C}; \sup_{\zeta\in
\partial W_2}|\zeta-z_1|<|w-z_1|<\inf_{\zeta\in \partial W_1}|\zeta-z_1|\}$. It turns out that
\bess\frac{1}{2\pi}\log Q(W_1,W_2,z_1)&\leq& {\rm mod}(W_1\setminus
\overline{W}_2) \leq {\rm mod}(B_1\setminus
\overline{B}_2)=\frac{1}{2\pi}\log\frac{R}{r}\\
&\leq&\frac{1}{2\pi}\log\big(C_3(m)\Delta(g(K);g(z_1),g(z_2))\big)\eess
The conclusion follows.
\end{proof}

\noindent\textit{Proof of Proposition \ref{6a}.}
By replacing $f$ with $f^k$, we assume $z_0$ is a fixed point of
$f$. Based on quasi-conformal surgery, we assume $z_0$ is a
superattracting fixed point with local degree $d={\rm
deg}(f:B\rightarrow B)\geq2$. Thus, $B$ contains no critical points
other than $z_0$. By M\"obius conjugation, we assume $z_0=\infty$.

  Because $f$ satisfies the {\bf{BD}} condition on $\partial B$, there exists a constant $\delta>0$ such that for any $u\in\partial B$, any integer $m\geq0$ and any component $U_m(u)$ of $f^{-m}(B(u,\delta))$ that intersects with $\partial B$, $U_m(u)$ is simply connected and ${\rm deg}(f^m: U_m(u)\rightarrow B(u,\delta))\leq D$. In fact, we can choose $\delta$ as the Lebesgue number of the family $\mathcal{F}=\{B(u,\varepsilon_u); u\in
\partial B\}$, which is an open covering of the boundary $\partial B$.

  The proof consists of four steps, as follows:

\textbf{Step 1.} {\it Let $V_m(z)$ be the component of
$f^{-m}(B(z,\delta/2))$ contained in $U_m(z)$ and intersecting with
$\partial B$, then
$$\lim _{m\rightarrow\infty}\sup_{z\in\partial B}{\rm
{diam}}(V_m(z))=0.$$ }

Otherwise, there is a constant $d_0\geq0$ and two sequences
$\{z_k\}\subset \partial B$ and $\{\ell_k\}$ such that ${\rm
{diam}}(V_{\ell_k}(z_k))\geq d_0$. For every $k\geq1$, choose a
point $y_k\in f^{-\ell_k}(z_k)\cap V_{\ell_k}(z_k)$. By passing to a
subsequence, we assume $y_k\rightarrow y_\infty\in \partial B$ and
$z_k\rightarrow z_\infty\in \partial B$. Based on Lemma \ref{6b},
there is a constant $C(D)$ such that
$${\rm Shape} (V_{\ell_k}(z_k),y_k)\leq C(D){\rm Shape} (B(z_k,\delta/2),z_k)=C(D).$$

Because ${\rm {diam}}(V_{\ell_k}(z_k))\geq d_0$, $V_{\ell_k}(z_k)$
contains a round disk of definite size centered at $y_k$. Therefore,
there is a constant $r_0=r_0(d_0,D)$ such that
$V_{\ell_k}(z_k)\supset B(y_\infty, r_0)$ for large $k$. Therefore,
$f^{\ell_k}(B(y_\infty,r_0))\subset B(z_k,\delta/2)\subset
B(z_\infty,\delta)$. But, this contradicts the fact that
$f^{\ell_k}(B(y_\infty,r_0))\supset J(f)$ when $k$ is large.

\textbf{Step 2.} {\it There are two constants $L>0$ and $\nu\in
(0,1)$ such that for any $z\in \partial B$ and any $k\geq 1$, ${\rm
{diam}}(V_k(z))\leq L\nu^k $.}

By step 1, there is an integer $s>0$ such that ${\rm
{diam}}(V_s(z))<\delta/4$ for all $z\in \partial B$. For each $x\in
\partial B$, we take a point $x_{kx}\in V_{ks}(x)\cap f^{-ks}(x)$.
(Notice that, in general, $V_{ks}(x)\cap f^{-ks}(x)$ consists of
finitely many points, $x_{ks}$ can be either of them). For $0\leq
j\leq k$, let $x_{js}=f^{(k-j)s}(x_{ks})$ and $U_j$ be the component
of $f^{-js}(B(x_{(k-j)s},\delta/2))$ containing $x_{ks}$. Then,
$$x_{ks}\in  V_{ks}(x)=U_k\subset \cdots\subset U_0=B(x_{ks},\delta/2).$$

For every $1\leq j<k$, $f^{js}:U_j\rightarrow
B(x_{(k-j)s},\delta/2)$ is a proper map of degree $\leq D$. Because
$f^{js}(U_{j+1})$ is contained in $B(x_{(k-j)s},\delta/4)$,
$${\rm mod}(U_j\setminus \overline{U_{j+1}})\geq \frac{1}{D}
{\rm mod}(B(x_{(k-j)s},\delta/2)\setminus
\overline{f^{js}(U_{j+1})})\geq \frac{\log 2}{2\pi D},$$
$${\rm mod}(B(x_{ks},\delta/2)\setminus
\overline{V_{ks}(x)})\geq \sum_{0\leq j<k}{\rm mod}(U_j\setminus
\overline{U_{j+1}})\geq \frac{k\log 2}{2\pi D}.$$

We know from the proof of Step 1 that ${\rm Shape}
(V_{ks}(x),x_{ks})\leq C(D)$. Therefore, there is a constant
$K(D)>0$ such that $\min_{y\in \partial V_{ks}(x)}|x_{ks}-y|\geq
K(D){\rm diam}(V_{ks}(x))$. We have

$${\rm mod}(B(x_{ks},\delta/2)\setminus
\overline{V_{ks}(x)})
\leq \frac{1}{2\pi}\log \Big(\frac{\delta}{2K(D){\rm
diam}(V_{ks}(x))}\Big).$$

It turns out that ${\rm {diam}}(V_{ks}(x))\leq
\frac{\delta}{2K(D)}2^{-k/D}$, which implies that there are two
constants $L>0$ and $\nu\in (0,1)$ such that ${\rm
{diam}}(V_{k}(x))\leq L\nu^k$ for all $k\geq1$.

\textbf{Step 3.} {\it There exists a sequence of Jordan curves
$\{\gamma_k: \mathbb{S}\rightarrow B\}$ such that $\gamma_k$
converges uniformly to a continuous and surjective map
$\gamma_\infty: \mathbb{S}\rightarrow \partial B$, where
$\mathbb{S}=\mathbb{R}/\mathbb{Z}$ is the unit circle. Hence,
$\partial B$ is locally connected.}

Recall that the B\"ottcher map $\phi: B\rightarrow
\mathbb{\overline{C}}\setminus\mathbb{\bar{D}}$ defined by
$\phi(z)=\displaystyle\lim_{k\rightarrow\infty}(f^k_\lambda(z))^{d^{-k}}$
is a conformal isomorphism, which satisfies $\phi^{-1}(r^de^{2\pi
idt})=f(\phi^{-1}(r e^{2\pi it}))$ for $(r,t)\in (1,+\infty)\times
\mathbb{S}$. Let $\ell (R,t)=\phi^{-1}([\sqrt[d]{R},R]e^{2\pi it})$
for $(R,t)\in (1,2)\times \mathbb{S}$. By the boundary behavior of
hyperbolic metric, there is a constant $C>0$ such that for
any$(R,t)\in (1,2)\times \mathbb{S}$,
 \bess {\rm Eucl. length}(\ell
(R,t))&\leq&C {\rm Hyper.length}(\ell (R,t))
 \cdot{\rm H.dist}(\phi^{-1}(R\mathbb{S}),\partial B)\\
&\leq&C (\log d)
 {\rm H.dist}(\phi^{-1}(R\mathbb{S}),\partial B)\ \ (\rightarrow0 {\text\  as \ } R\rightarrow 1), \eess
where ${\rm Hyper.length}$ is the hyperbolic length in $B$ and ${\rm
H.dist}$ is the Hausdorff distance in the sphere $\mathbb{\bar{C}}$.
Thus, we can choose $R$ sufficiently close to $1$ such that for any
$t\in \mathbb{S}$, $\ell (R,t)\subset B(z,\delta/2)$ for some $z\in
\partial B$. For $k\geq0$, define a curve $\gamma_k: \mathbb{S}\rightarrow B$ by
$\gamma_k(t)=\phi^{-1}(R^{1/d^k}e^{2\pi it})$. Because
$f^k(\gamma_{k+q}(t))=\gamma_q(d^kt)$ for $q\geq0$ and
$\gamma_0(d^kt), \gamma_1(d^kt)\in \ell (R,d^kt)\subset
B(z,\delta/2)$ for some $z\in
\partial B$, we conclude that $\gamma_{k}(t)$ and $\gamma_{k+1}(t)$ lie in the same component of $f^{-k}(B(z,\delta/2))$ that intersect with
$\partial B$. Based on Step 2,
$$\max_{t\in \mathbb{S}}|\gamma_{k+1}(t)-\gamma_{k}(t)|=\mathcal{O}(\nu^k).$$
So $\{\gamma_k: \mathbb{S}\rightarrow B\}$ is a Cauchy sequence and
hence converges to a continuous map $\gamma_\infty:
\mathbb{S}\rightarrow \partial B$.

To finish, we show $\gamma_\infty$ is surjective. Let $B_k\subset B$
be the disk bounded by $\gamma_k(\mathbb{S})$; then $B_k\Subset
B_{k+1}$ and $\cup_k B_k=B$. Each point $z\in \partial B$ can
therefore be approximated by a sequence of points
$\{z_k=\gamma_k(t_k)\}_{k\geq1}$ with $z_k\in \partial B_k$. There
is a subsequence $k_j$ such that $t_{k_j}\rightarrow t_\infty \in
\mathbb{S}$ as $j\rightarrow\infty$. We then have
$\gamma_\infty(t_\infty)=\lim_j \gamma_{k_j}(t_\infty)=\lim_j
\gamma_{k_j}(t_{k_j})=z$.  It follows that $\gamma_\infty$ is
surjective.

\textbf{Step 4.} {\it If, furthermore, $\partial B$ is a Jordan
curve, then $\partial B$ is a quasi-circle.}

Because
 $\partial B$ is a Jordan curve, the B\"otcher map
 $\phi: B\rightarrow \mathbb{\overline{C}}\setminus
 \mathbb{\overline{D}}$ can be extended to a homeomorphism
 $\phi: \overline{B}\rightarrow \mathbb{\overline{C}}\setminus
 \mathbb{{D}}$. Define a map $\psi: \mathbb{S}\rightarrow\partial B$
 by $\psi(\zeta)=\phi^{-1}(\zeta)$ for $\zeta\in \mathbb{S}$. Then
 $f(\psi(\zeta))=\psi(\zeta^d)$. Let $\varphi=\phi|_{\partial B}$ be
 the inverse of $\psi$. Both $\psi$ and $\varphi$ are uniformly  continuous; thus, for any sufficiently small positive number
 $\varepsilon$, there are two small constants $a(\varepsilon),
 b(\varepsilon)$ such that
 \bess
\forall\  (\zeta_1,\zeta_2)\in \mathbb{S}\times \mathbb{S},
|\zeta_1-\zeta_2|< a(\varepsilon)&\Rightarrow&|\psi(\zeta_1)-\psi(\zeta_2)|< \varepsilon;\\
\forall\  (z_1,z_2)\in \partial B\times \partial B, |z_1-z_2|<
b(\varepsilon)&\Rightarrow&|\varphi(z_1)-\varphi(z_2)|<
a(\varepsilon). \eess

Given two points $z_1, z_2\in \partial B$, $\partial B\setminus
\{z_1,z_2\}$ consists of two components, say $E_1$ and $E_2$. Let
$L(z_1,z_2)\in \{\overline{E}_1, \overline{E}_2\}$ be a section of
$\partial B$ such that ${\rm diam}(L(z_1,z_2))=\min\{{\rm
diam}(E_1), {\rm diam}(E_2)\}$. Thus, for any positive number
$\varepsilon\ll {\rm diam}(\partial B)$, by uniform continuity we
have
 \bess |z_1-z_2|<
b(\varepsilon)\Rightarrow {\rm diam}(L(z_1,z_2))< \varepsilon.
\label{7.a}\eess

Based on Ahlfors' characterization of quasi-circles \cite{A}, to
prove that $\partial B$ is a quasi-circle, it suffices to show that
there is a constant $C>0$ such that for any $z_1,z_2\in
\partial B$ with $z_1\neq z_2$,   $ \Delta(L(z_1,z_2);z_1,z_2)\leq
C$. In fact, if $|z_1-z_2|\geq\epsilon$ for some positive constant
$\epsilon$, then $ \Delta(L(z_1,z_2);z_1,z_2)\leq {\rm
diam}(\partial B)/\epsilon$. Therefore, we only need to consider the
case when $|z_1-z_2|$ is small. In the following, we assume
$\delta\ll {\rm diam}(\partial B)$ and $|z_1-z_2|\leq b(\delta/2)$;
it turns out that $ {\rm diam}(L(z_1,z_2))< \delta/2$.

Because $f$ is expanding on $\partial B$, there is an integer $N>0$
such that $f^k(L(z_1,z_2))=\partial B$ for all $k\geq N$. We can
therefore find a smallest integer $\ell\geq 0$ such that

$$ {\rm
diam}(f^\ell(L(z_1,z_2)))<\delta/2,\ {\rm
diam}(f^{\ell+1}(L(z_1,z_2)))\geq \delta/2.$$

On the other hand, there exist two points $w_1,w_2\in
f^\ell(L(z_1,z_2))$ such that \bess {\rm
diam}(f^{\ell+1}(L(z_1,z_2)))&=&|f(w_1)-f(w_2)|\leq\int_{[w_1,w_2]}|f'(z)||dz|\\
&\leq&M|w_1-w_2|\leq M{\rm diam}(f^\ell(L(z_1,z_2))), \eess where
$[w_1,w_2]$ is the straight segment connecting $w_1$ with $w_2$ and
$$M=\max\{|f'(z)|; {\rm Eucl.dist}(z, \partial B)\leq\delta/2\}.$$
Thus, we have
$$\frac{\delta}{2M}\leq {\rm
diam}(f^\ell(L(z_1,z_2)))= {\rm
diam}(L(f^\ell(z_1),f^\ell(z_2)))<\frac{\delta}{2}.$$ By (\ref
{7.a}), there is a constant $c(\delta,M)>0$ such that
$|f^\ell(z_1)-f^\ell(z_2)|\geq c(\delta,M)$.

Applying Lemma \ref{6b} to the situation
$(V_1,U_1)=(U_\ell(f^\ell(z_1)),V_\ell(f^\ell(z_1))),
(V_2,U_2)=(B(f^\ell(z_1),\delta),B(f^\ell(z_1),\delta/2))$ and
$g=f^\ell$, we conclude that there is a constant $C(D)>0$ such that
$$\Delta(L(z_1,z_2);z_1,z_2)\leq C(D)\Delta(f^\ell(L(z_1,z_2));f^\ell(z_1),f^\ell(z_2))\leq\frac{C(D)\delta}{2c(\delta,M)}.$$

Thus, for any $x,y\in
\partial B$ with $x\neq y$, the turning $\Delta(L(x,y);x,y)$ is
bounded by
$$\max\Bigg\{ \frac{{\rm
diam}(\partial B)}{b(\delta/2)}, \frac{C(D)\delta}{2c(\delta,M)}
\Bigg\}.$$ \hfill $\Box$

\begin{rem}
Using the same argument as \cite{CJY}, one can show further that if
$f$ satisfies {\bf{BD}} condition on $\partial B$, then $\partial B$
is a John domain.
\end{rem}

The following describes an important case in which $f$ satisfies the
{\bf{BD}} condition on $\partial B$.

\begin{pro}\label{6c}  If $\#(P(f)\cap\partial B)<\infty$ and all
periodic points in $P(f)\cap\partial B$ are repelling, then $f$
satisfies {\bf{BD}} condition on $\partial B$.
\end{pro}
\begin{proof} The proof is based on the following claim.

 {\bf Claim:} {\it For any $u\in\partial B$, there is a constant
$\varepsilon_u>0$ such that for any $m\geq0$ and any component
$U_m(u)$ of $f^{-m}(B(u,\varepsilon_u))$ that intersects with
$\partial B$, $U_m(u)$ contains at most one critical point of
$f^m$.}

The claim implies that $U_m(u)$ is simply connected by the
Riemann-Hurwitz Formula. Because the sequence $U_m(u)\rightarrow
f(U_m(u))\rightarrow\cdots\rightarrow f^{m-1}(U_m(u))\rightarrow
B(u,\varepsilon_u)$ meets every critical point of $f$ at most once,
we conclude that ${\rm deg}(f^m: U_m(u)\rightarrow
B(u,\varepsilon_u))$ is bounded by $D=\Pi_{c\in C(f)}{\rm
deg}(f,c)$.

In the following, we prove the claim.

 First, note that every point in $P(f)\cap \partial
B$ is pre-periodic; we can deconstruct $\partial B$ into three
disjoint sets $X, Y$ and $Z$, where $X=\partial B\setminus P(f)$,
$Z$ is the union of all repelling cycles in $P(f)\cap \partial B$
and $Y=(P(f)\cap \partial B)\setminus Z$.

 For any $x\in X$, choose a small number $\varepsilon_x>0$ such that $B(x,\varepsilon_x)\cap P(f)=\emptyset$. Then, for any component $W_m(x)$ of $f^{-m}(B(x,\varepsilon_x))$ intersecting with $\partial B$,
$f^m:W_m(x)\rightarrow B(x,\varepsilon_x)$ is a conformal map.

The set $Y$ consists of all strictly pre-periodic points. Thus,
there is an integer $q\geq1$ such that for any $y\in Y$,
$f^{-q}(y)\cap P(f)\cap \partial B=\emptyset$. For an open set $U$
in $\mathbb{\bar{C}}$ and a point $u\in U$, we use ${\rm Comp}_u(U)$
to denote the component of $U$ that contains $u$. For every $y\in
Y$, choose $\varepsilon_y>0$ small enough such that for any $x\in
f^{-q}(y)\cap
\partial B\subset X$, ${\rm Comp}_x(f^{-q}(B(y,\varepsilon_y)))\subset
B(x,\varepsilon_x)$ and ${\rm Comp}_x(f^{-q}(B(y,\varepsilon_y)))$
contain at most one critical point of $f^q$.

Finally, we deal with $Z$. For $z\in Z$, suppose $z$ lies in a
repelling cycle of period $p$. Choose $\varepsilon_z>0$ such that

(1) $B(z,\varepsilon_z)$ is contained in the linearizable
neighborhood of $z$ and ${\rm Comp}_z(f^{-p}(B(z,\varepsilon_z)))$
is a subset of $B(z,\varepsilon_z)$

(2) For every $u\in (f^{-p}(z)\cap
\partial B)\setminus \{z\}\subset X\cup Y$, ${\rm Comp}_u(f^{-p}(B(z,\varepsilon_z)))$
contains at most one critical point of $f^p$ and ${\rm
Comp}_u(f^{-p}(B(z,\varepsilon_z)))\subset B(u,\varepsilon_u)$.

One can easily verify that the collection of neighborhoods
$\{B(u,\varepsilon_u) , u\in \partial B\}$ are just as required.
\end{proof}

\begin{cor}\label{6d} If $f$ is critically finite, then $f$ satisfies the
{\bf{BD}} condition on $\partial B$.
\end{cor}
\begin{proof} Because $f$ is critically finite, every periodic point
of $f$ is either repelling or superattracting, which implies that
$\#(P(f)\cap\partial B)<\infty$ and all periodic points in
$P(f)\cap\partial B$ are repelling. Thus, by Proposition \ref{6c},
$f$ satisfies the {\bf{BD}} condition on $\partial B$.
\end{proof}

\section{The boundary $\partial B_\lambda$ is a Jordan curve}

In this section, we will prove Theorem \ref{11a} and Theorem
\ref{11b}. The strategy of the proof is as follows.

First, consider the McMullen maps $f_\lambda$ with parameter
$\lambda\in\mathcal{H}$. If $f_\lambda$ is critically finite, then
the Julia set is locally connected. Otherwise, by Proposition
\ref{4ab}, we can find an admissible graph
$\mathbf{G}_\lambda(\theta_1,\cdots,\theta_N)$. With respect to the
Yoccoz puzzle induced by this graph, there are two possibilities:

\text{\bf Case 1: None of $T(c)$ with $c\in C_\lambda$ is periodic.}
This case is discussed in section 7.1, and the local connectivity of
$J(f_\lambda)$ follows from Proposition \ref{7b}. The idea of the
proof is based on the combinatorial analysis for tableaux introduced
by Branner and Hubbard (see \cite {BH}, \cite {M2}) and on `modified
puzzle piece' techniques.

\text{\bf Case 2: Some $T(c)$ with $c\in C_\lambda$ is periodic.} In
this case, the map $f_\lambda$ is either renormalizable or
$*$-renormalizable. This case is discussed in section 7.2. The local
connectivity of $\partial B_\lambda$ follows from Proposition
\ref{7c}. The goal of the proof of Proposition \ref{7c} is to
construct a closed curve separating $\partial B_\lambda$ from the
small filled Julia set $K_c$.

In section 7.3, we deal with the real parameters $\lambda\in
\mathbb{R}^+$.

In section 7.4, we improve the regularity of the boundary $\partial
B_\lambda$. We first include a proof of Devaney that claims that the
local connectivity of $\partial B_\lambda$ implies that $\partial
B_\lambda$ is a Jordan curve. We then show that $\partial B_\lambda$
is a quasi-circle except in two specific cases.

In section 7.5, we present some corollaries.

\subsection{None of $T(c)$ with $c\in C_\lambda$ is periodic}

Recall that $J_0$ is the set of all points on the Julia set
$J(f_\lambda)$ whose orbits eventually meet the graph
$\mathbf{G}_\lambda(\theta_1,\cdots,\theta_N)$.

\begin{lem}\label{7a} Let $z\in J(f_\lambda)\setminus J_0$. If $T(z)$ is
non-critical, then ${\rm
End}(z):=\bigcap_{d\geq0}\overline{P_d(z)}=\{z\}$.
\end{lem}
\begin{proof}
It suffices to prove ${\rm End}(f_\lambda(z))=\{f_\lambda(z)\}$.
Because $T(z)$ is non-critical, there is an integer $d_0\geq1$ such
that for any $j>0$, the position $(d_0,j)$ is not critical.
Equivalently, for any $d\geq d_0$ and any $j\geq1$, the puzzle piece
$P_d(f_\lambda^j(z))$ contains no critical point. Let
$\{\widehat{P}_{d_0-1}^{(i)} and 1\leq i\leq M\}$ be the collection
of all modified puzzle pieces of depth $d_0-1$, numbered so that
$\widehat{P}_{d_0-1}^{(1)}=\widehat{P}_{d_0-1}(v_\lambda^+),\
\widehat{P}_{d_0-1}^{(2)}=\widehat{P}_{d_0-1}(v_\lambda^-)$, and
recall that we use $\widehat{P}_{d}(w)$ to denote the modified
puzzle piece of ${P}_{d}(w)$. Every modified puzzle piece of
depth$\geq d_0$ is contained in a unique modified  puzzle piece
$\widehat{P}_{d_0-1}^{(i)}$ of depth $d_0-1$. Let ${\rm
dist}_i(x,y)$ be the Poincar\'e metric of
$\widehat{P}_{d_0-1}^{(i)}$. For $2< i \leq M$, there are exactly
$2n$ branches of $f_\lambda^{-1}$ on $\widehat{P}_{d_0-1}^{(i)}$,
say $g_1^i,g_2^i,\cdots,g_{2n}^i$, and each $g_k^i$ on
$\widehat{P}_{d_0-1}^{(i)}$ is univalent and carries
$\widehat{P}_{d_0}^{(\alpha)}\subset\subset
\widehat{P}_{d_0-1}^{(i)}$ onto a proper subset of some
$\widehat{P}_{d_0-1}^{(j)}$. It follows that there is a uniform
constant $0<\nu<1$ such that
$${\rm dist}_j(g_k^i(x),g_k^i(y))\leq \nu {\rm dist}_i(x,y)$$
for any $x,y\in\widehat{P}_{d_0}^{(\alpha)}\subset\subset
\widehat{P}_{d_0-1}^{(i)}$ and any $2< i\leq M, 1\leq k \leq 2n$.

Let $D$ be the maximum Poincar\'e diameters of the modified puzzle
pieces of depth $d_0$. For any integer $h>0$, because the sequence
$${P}_{d_0+h}(f_\lambda(z))\rightarrow {P}_{d_0+h-1}(f_\lambda^2(z))\rightarrow
\cdots {P}_{d_0+1}(f_\lambda^h(z))\rightarrow
{P}_{d_0}(f_\lambda^{h+1}(z))$$ contains no critical point (this
follows from the assumption that $T(z)$ is non-critical), it follows
that
$$ \mathrm{Hyper. diam}\big(
{P}_{d_0+h}(f_\lambda(z))\big)\leq D\nu^h$$ with respect to the
Poincar\'e metric of $\widehat{P}_{d_0-1}(f_\lambda(z))$. Thus, we
have ${\rm End}(f_\lambda(z))=\{f_\lambda(z)\}$.
\end{proof}

\begin{pro}\label{7b} If $T(c)$ is not periodic for any $c\in
C_\lambda$, then the Julia set $J(f_\lambda)$ is locally connected.
\end{pro}
\begin{proof} Note that $T(c)$ is either critical or non-critical. First, we prove ${\rm End}(c)=\{c\}$ and ${\rm End}(z)=\{z\}$ for any $z\in J(f_\lambda)\setminus J_0$. We then deal with the points that lie in $J_0$.

\text{\bf Case 1: $T(c)$ is critical.} Because the graph is
admissible, we can find a non-degenerate annulus $A_{d_0}(c)$.
Consider the descendents of ${\rm Row}_c(d_0)$. It is obvious that
if ${\rm Row}_c(t)$ is a descendent in the $k$-th generation of
${\rm Row}_c(d_0)$, the annulus $A_{t}(c)$ is non-degenerate with
modulus ${\rm mod}(A_{d_0}(c))/2^{k}$. If ${\rm Row}_c(d_0)$ has at
least $2^k$ descendents in the $k$-th generation for each $k\geq1$,
then each of these contributes exactly ${\rm mod}(A_{d_0}(c))/2^{k}$
to the sum $\sum_d {\rm mod}(A_{d}(c))$. Hence, $\sum_d {\rm
mod}(A_{d}(c))=\infty$, as required. On the other hand, if there are
fewer descendents in some generation, then one of them, say ${\rm
Row}_c(m)$, must be an only child, hence excellent by Lemma
\ref{4c}. Again by Lemma \ref{4c}, we see that $\sum_d {\rm
mod}(A_{d}(c))=\infty$. Therefore, in either case, ${\rm
End}(c)=\{c\}$.

Now consider a point $z\in J(f_\lambda)\setminus (J_0\cup
C_\lambda)$. If $T(z)$ is non-critical, then by Lemma \ref{7a},
${\rm End}(z)=\{z\}$. If $T(z)$ is critical, then for each $d\geq1$,
there is a smallest integer $l_d\geq0$ such that both $(d,l_d)$ and
$(d,l_d+1)$ are critical positions. It follows that
$f_\lambda^{l_d}:A_{d+l_d}(z)\rightarrow A_{d}(c')$ is a conformal
map for some $c'\in C_\lambda$. In this case, $\sum_d {\rm
mod}(A_{d}(z))\geq\sum_d {\rm mod}(A_{d+l_d}(z))=\sum_d {\rm
mod}(A_{d}(c))=\infty$, hence ${\rm End}(z)=\{z\}$.

\text{\bf Case 2: $T(c)$ is non-critical.} It follows from Lemma
\ref{7a} that ${\rm End}(c)=\{c\}$. For $z\in J(f_\lambda)\setminus
(J_0\cup C_\lambda)$, we assume $T(z)$ is critical; otherwise, ${\rm
End}(z)=\{z\}$ based on Lemma \ref{7a}. Suppose $A_{d_0}(c)$ is a
non-degenerate annulus and $(d_0+1,l_1),(d_0+1,l_2),\cdots$ are all
critical positions in the $(d_0+1)$-th row of the tableau $T(z)$.
Because all tableaux $T(c)$ with $c\in C_\lambda$ are non-critical,
there is a constant $D$ such that ${\rm
deg}(f_\lambda^{l_k}:P_{d_0+l_k}(z)\rightarrow P_{d_0,l_k}(z))\leq
D$ for all $k\geq1$. Thus,
$${\rm
mod}(A_{d_0+l_k}(z))\geq D^{-1}{\rm mod}(A_{d_0}(c))$$ for all
$k\geq1$. Hence, $\sum_d {\rm mod}(A_{d}(z))\geq\sum_k {\rm
mod}(A_{d_0+l_k}(z))=\infty$ and ${\rm End}(z)=\{z\}$.

\textbf{Points that lie in $J_0$.}  For any $z\in J_0$, the orbit
$z\mapsto f_\lambda(z)\mapsto f^2_\lambda(z)\mapsto\cdots$
eventually meets the graph
$\mathbf{G}_\lambda(\theta_1,\cdots,\theta_N)$. Therefore, the
Euclidean distance between the critical set $C_\lambda$ and the
orbit $\{f_\lambda^k(z)\}_{k\geq0}$ is bounded below by some
positive number $\epsilon(z)$.
  In addition, for every $d$ large enough, $z$ lies in the common boundary of exactly two puzzle pieces of depth $d$. We denote these two puzzle pieces by $P_d'(z)$ and $P_d''(z)$. In the previous argument, we have already proved that ${\rm End}(c)=\{c\}$; this implies $\mathrm{Eucl.
diam}(P_d(c))\rightarrow0$ as $d\rightarrow\infty$. Choose a $d_0$
large enough such that
$$\mathrm{Eucl. diam}(P_{d_0}(c))<\epsilon(z)
 \leq \mathrm{Eucl. dist}(C_\lambda, \{f_\lambda^k(z)\}_{k\geq0}).$$
Then, the orbit $z\mapsto f_\lambda(z)\mapsto
f^2_\lambda(z)\mapsto\cdots$ avoids all the critical puzzle pieces
of depth $d_0$. Let $P^*_{d}(z)=\overline{P_{d}'(z)\cup P_{d}''(z)}$
for $d$ large enough. Then, the proof of Lemma \ref{7a} applies
equally well to this situation, and $\bigcap _d P^*_{d}(z)=\{z\}$
immediately follows.

\textbf{Connectivity of neighborhoods.} Let
 \begin{equation*}
P^*_{d}(z)=\begin{cases}
\overline{P_d(z)},\ \  &\text{ if  $z\in J(f_\lambda)\setminus J_0$},\\
\overline{P_{d}'(z)\cup P_{d}''(z)},\ \ &\text{ if  $z\in J_0$ and
$d$ is large.}
 \end{cases}
\end{equation*}
Based on Lemma \ref {4aa}, for every $z\in J(f_\lambda)$ and every
large integer $d$, the intersection $P^*_{d}(z)\cap J(f_\lambda)$ is
a connected and compact subset of $J(f_\lambda)$. Thus,
$\{P^*_{d}(z)\cap J(f_\lambda)\}$ forms a basis of connected
neighborhoods of $z$. Because $\bigcap (P^*_{d}(z)\cap
J(f_\lambda))=\{z\}$, the Julia set is locally connected at $z$.
Note that $z$ is arbitrarily chosen, we conclude that $J(f_\lambda)$
is locally connected.
\end{proof}

\subsection{Some $T(c)$ with $c\in C_\lambda$ is periodic}

Suppose some tableau $T(c)$ with $c\in C_\lambda$ is $k$-periodic
for some $k>0$. Based on the proof of Lemma \ref{5a}, $f_\lambda$ is
either $k$-renormalizable at $c$ or $k/2$-$*$-renormalizable at $c$.
Let $(\epsilon f^p_\lambda, P_{d_0+p}(c), P_{d_0}(c))$, where $d_0$
is a large integer, be the renormalization and
 \begin{equation*}
(\epsilon,p)=\begin{cases}
(1,k),\ \  &\text{ if $f_\lambda$ is $k$-renormalizable at $c$},\\
(-1,k/2),\ \ &\text{ if $f_\lambda$ is  $k/2$-$*$-renormalizable at
$c$}.
 \end{cases}
\end{equation*}
The small filled Julia set of the renormalization $(\epsilon
f^p_\lambda, P_{d_0+p}(c), P_{d_0}(c))$ is denoted by $K_c$. Recall
that $\beta_c$ is the $\beta$-fixed point of the renormalization and
$\beta'_c$ is the other preimage of $\beta_c$ under the map
$\epsilon f^p_\lambda|_{P_{d_0+p}(c)}$.

Assume now that $K_c\cap \partial B_\lambda\neq\emptyset$; then,
based on Lemma \ref{5d}, $\beta_c\in K_c\cap \partial B_\lambda$ and
there is a unique external ray, say $R_\lambda(\theta)$, landing at
$\beta_c$. The angle $\theta$ is of the form $\frac{m}{2^k-1}$. It
follows that $\beta'_c\in K_c\cap \partial T_\lambda$ and there is a
unique radial ray $R_{T_\lambda}(\alpha_\theta)$ in $T_\lambda$
landing at $\beta'_c$. The radial ray $R_{T_\lambda}(\alpha_\theta)$
satisfies $\epsilon
f_\lambda^p(R_{T_\lambda}(\alpha_\theta))=R_\lambda(\theta)$. Let
$$K=K_c\cup \overline{R_\lambda(\theta)}\cup
\overline{R_{T_\lambda}(\alpha_\theta)}\cup(-K_c)\cup(-\overline{R_\lambda(\theta)})
\cup(-\overline{R_{T_\lambda}(\alpha_\theta)}).$$

The set $K$ is a connected and compact subset of $\mathbb{\bar{C}}$.
Note that
$-R_{T_\lambda}(\alpha_\theta)=R_{T_\lambda}(\alpha_\theta+1/2)$.
Let $\Delta_1$ be the component of $\mathbb{\bar{C}}\setminus (K\cup
\overline{B_\lambda})$ that intersects with
$Q_{T_\lambda}(\alpha_\theta,\alpha_\theta+1/2)$ and $\Delta_2$ be
the component of $\mathbb{\bar{C}}\setminus (K\cup
\overline{B_\lambda})$ that intersects with
$Q_{T_\lambda}(\alpha_\theta+1/2,\alpha_\theta)$, where we use
$Q_{T_\lambda}(\theta_1,\theta_2)$ to denote the set
$\overline{\{\phi_{T_\lambda}(re^{2\pi it}); 0<r< 1, \theta_1\leq
t\leq \theta_2\}}$. Because $K\cup \overline{B_\lambda}$ is
connected and compact, both $\Delta_1$ and $\Delta_2$ are disks. Let
$Z_i$ be the component of $\mathbb{\bar{C}}\setminus K$ that
contains $\Delta_i$.

The aim of this section is to prove:
\begin{pro}\label{7c} Assume that $K_c\cap \partial B_\lambda\neq\emptyset$,
 then for $i\in \{1,2\}$, there is a curve $\mathcal{L}_i\subset \Delta_i\cup\{0\}$
 stemming from $T_\lambda$ and converging to $\beta_c$. More
 precisely, $\mathcal{L}_i$ can be parameterized as
  $\mathcal{L}_i: [0,+\infty)\rightarrow \Delta_i\cup \{0\}$ such
  that $\mathcal{L}_i(0)=0,
  \mathcal{L}_i((0,+\infty))\subset\Delta_i$ and
  $\lim_{t\rightarrow+\infty}\mathcal{L}_i(t)=\beta_c$.
\end{pro}
\begin{proof}
Let $\Gamma=\bigcup_{j\geq0}(\pm
f_\lambda^j(K_c\cup\overline{R_\lambda(\theta)} ))$. By Lemma
\ref{5d}, any two distinct elements in the set $\{\pm
f_\lambda^j(K_c\cup\overline{R_\lambda(\theta)}); j\geq0\}$
intersect only at the point $\infty$, which implies that
$U=\mathbb{\bar{C}}\setminus\Gamma$ is a disk.

\textbf{Step 1.} {\it There exists $G_i:U\rightarrow U\cap Z_i$, an
inverse branch of $\epsilon f_\lambda^p$ such that the sequence
$\{G_i^l; l\geq0\}$ converges locally and uniformly in $U$ to a
constant $z_i\in K_c$.}

Because $U$ has no intersection with the post-critical set of
$f_\lambda$, its preimage $f_\lambda^{-1}(U)$ has exactly $2n$
components, say $V_1,\cdots, V_{2n}$. These components are arranged
symmetrically about the origin under the rotation $z\mapsto e^{\pi
i/n}z$. For every $1\leq j\leq 2n$, $f_\lambda: V_j\rightarrow U$ is
a conformal map. Moreover, $f_\lambda^{-1}(U)\subset
\mathbb{\bar{C}}\setminus K$.

For $1\leq j\leq p-1$, let $\Omega_j\in \{V_1,\cdots, V_{2n}\}$ be
the component of $f_\lambda^{-1}(U)$ such that
$\overline{\Omega}_j\cap f_\lambda^j(K_c)\neq\emptyset$ and the
inverse of $f_\lambda:\Omega_j\rightarrow U$ is denoted by $g_j$.
For $j=0$, let $\Omega_0^i$ be the component of $f_\lambda^{-1}(U)$
such that $\overline{\Omega_0^i}\cap K_c\neq\emptyset$ and
$\Omega_0^i\subset Z_i$. The inverse of
$f_\lambda:\Omega_0^i\rightarrow U$ is denoted by $g_0^i$ for $i\in
\{1,2\}$.

Now, we define
\begin{equation*}
G_i(z)=\begin{cases}
g_0^i\circ g_1 \circ\cdots\circ g_{p-1}(\epsilon z) ,\ \ z\in U  &\text{ if $p\geq 2$},\\
g_0^i(\epsilon z) ,\ \ z\in U  &\text{ if $p=1$}.
 \end{cases}
\end{equation*}

Because $(\epsilon f^p_\lambda, P_{d_0+p}(c), P_{d_0}(c))$ is a
$p$-($*$-)renormalization of $f_\lambda$ at $c$, we have
$G_i(P_{d_0}(c)\cap U)\subset P_{d_0+p}(c)\cap Z_i$. The map
$G_i:U\rightarrow U$ is not surjective; thus, by the Denjoy-Wolff
Theorem (see \cite{M1}), the sequence $\{G_i^l; l\geq0\}$ converges
locally and uniformly in $U$ to a constant $z_i$. It follows from
$G_i(P_{d_0}(c)\cap U)\subset P_{d_0+p}(c)\cap Z_i$ that $z_i\in
K_c$.

\begin{figure}
\begin{center}
\includegraphics[height=7.5cm]{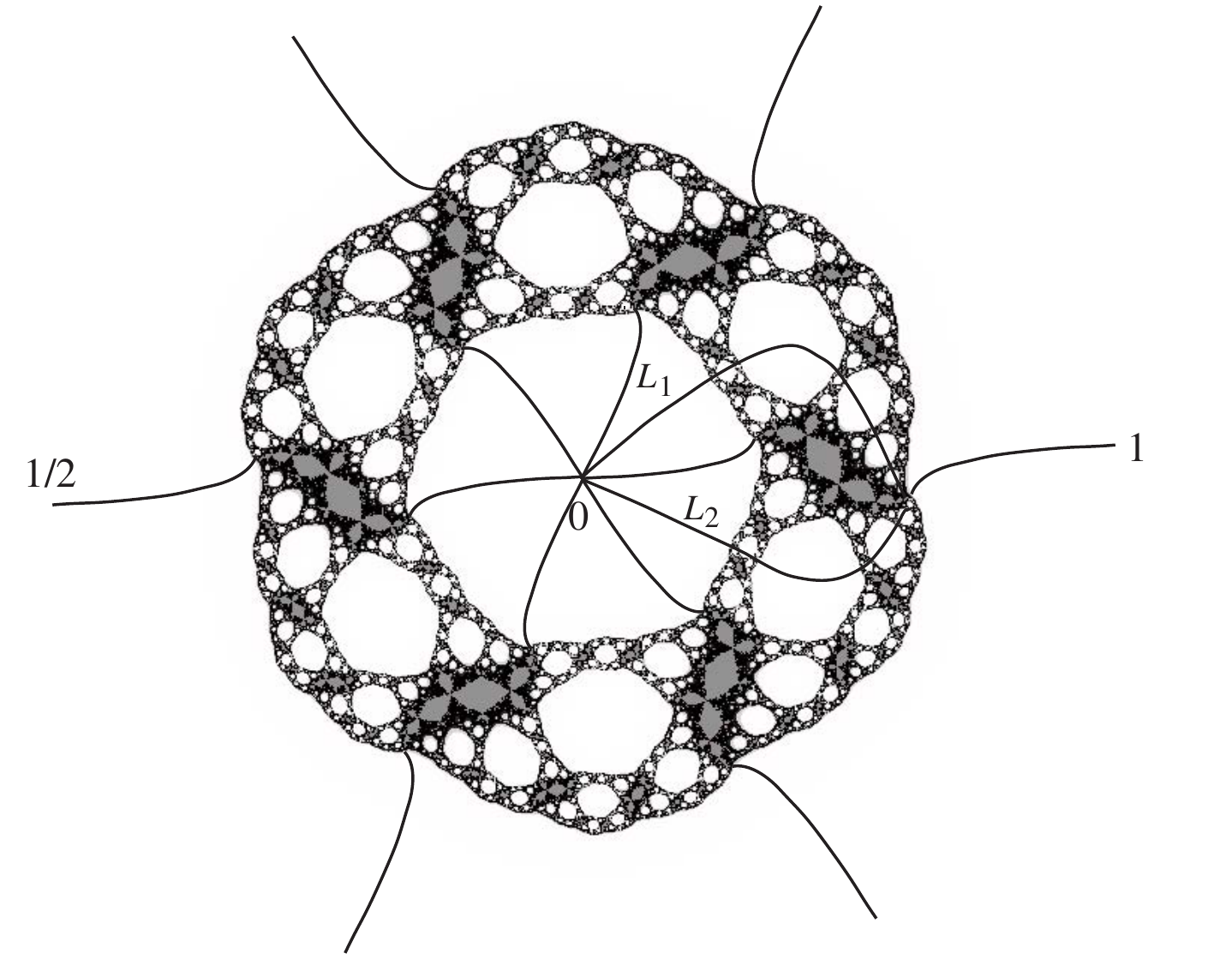}
\caption{Constructing two curves $L_1$ and $L_2$ that converge to
$\beta_c$, here $n=3$ and $f_\lambda$ is $1-$renormalizable at
$c=c_0$.}
\end{center}
\end{figure}

\textbf{Step 2.} {\it There exists a curve $C_i\subset
U\cap(\Delta_i\cup\{0\})$ connecting $0$ with $G_i(0)$ for
$i\in\{1,2\}$.}

Because the graph $\mathbf{G}_\lambda(\theta_1,\cdots,\theta_N)$ is
admissible, the filled Julia set $K_c$ is disjointed from the
boundary of any puzzle piece. Thus, for any $\alpha\in
\{\tau^s(\theta_j); 1\leq j\leq N, s\geq0\}$, $\Gamma$ is disjoint
from the cut ray $\Omega_\lambda^\alpha$ outside $\infty$. (This is
because the external ray $R_\lambda(\theta)$ has no intersection
with $g_\lambda(\theta_1,\cdots,\theta_N)$ outside $\infty$; compare
Lemma \ref{5d}). By Proposition \ref{4b3}, for any angle $\alpha\in
\{\tau^s(\theta_j); 1\leq j\leq N, s\geq0\}$ and any map $g\in
\{g_0^1,g_0^2,g_1,\cdots,g_{p-1}\}$, only one curve of
$g(\omega_\lambda^\alpha\setminus \{\infty\}),
g(\omega_\lambda^{\alpha+1/2}\setminus \{\infty\})$ intersects with
$\partial B_\lambda$, and the other curve connects $0$ with a
preimage of $0$.

Fix an angle $\alpha\in \{\tau^s(\theta_j); 1\leq j\leq N,
s\geq0\}$; we define a curve family $\mathcal{F}$ by
$$\mathcal{F}=\{\epsilon \omega_\lambda^\alpha\setminus \{\infty\} ;
 \ \epsilon^{2n}=1 \text{\ and \ } \epsilon \omega_\lambda^\alpha
  \subset \cup_{j\in\mathbb{I}\setminus\{0,n\}}S_j\}.$$

We construct the curve $C_i$ by an inductive procedure, as follows:

First, choose a curve $\zeta_{p-1}\in \mathcal{F}$ such that
$g_{p-1}(\zeta_{p-1})\cap \partial B_\lambda=\emptyset$ and let
$\gamma_{p-1}=g_{p-1}(\zeta_{p-1})$. Suppose that for some $2\leq
j\leq p-1$ we have already constructed the curves
$\gamma_{p-1},\cdots,\gamma_j$. We then choose $\zeta_{j-1}\in
\mathcal{F}$ such that $g_{j-1}(\zeta_{j-1})\cap
\partial B_\lambda=\emptyset$ and $\zeta_{j-1}\cap
\gamma_j=\emptyset$ and let $\gamma_{j-1}=g_{j-1}(\zeta_{j-1}\cup
\gamma_j)$. In this way, we can construct a sequence of curves
$\gamma_{p-1},\gamma_{p-2},\cdots,\gamma_{2},\gamma_{1}$ step by
step, and each curve has no intersection with $\partial B_\lambda$.
These curves connect $0$ with some iterated preimage of $0$. By
construction,
$$\gamma_1=\bigcup_{1\leq j\leq p-1}g_1\circ\cdots \circ g_j( \zeta_{j}).$$

We now choose $\zeta_{0}^i\in \mathcal{F}$ such that
$g_{0}^i(\zeta_{0}^i)\cap \partial B_\lambda=\emptyset$ and
$\zeta_{0}^i\cap \gamma_1=\emptyset$, and let
\begin{equation*}
C_i=\begin{cases}
g_0^i(\zeta_{0}^i\cup \gamma_1)\cup\{0\} ,\ \ &\text{ if $p\geq 2$},\\
g_0^i(\zeta_{0}^i)\cup\{0\} ,\ \  &\text{ if $p=1$}.
 \end{cases}
\end{equation*}

 The curve $C_i$ connects $0$ to $G_i(0)$ and $C_i\subset
U\cap(\Delta_i\cup\{0\})$, as required.

\textbf{Step 3.} {\it The union
$\mathcal{L}_i=\bigcup_{j\geq0}G_i^j(C_i)$ is the curve contained in
$\Delta_i\cup\{0\}$ and converging to $\beta_c$.}

By construction, $G_i(\mathcal{L}_i)\subset G_i(\mathcal{L}_i)\cup
C_i=\mathcal{L}_i$ and $\mathcal{L}_i\setminus
\{0\}\subset\Delta_i$.

To finish, we show $\mathcal{L}_i$ converges to $\beta_c$. By step
1, the sequence $\{G_i^k; k\geq0\}$ converges uniformly on any
compact subset of $U$ to a constant $z_i\in K_c$. Because $C_i$ is a
compact subset of $U$, the curve $\mathcal{L}_i$ converges to
$z_i\in K_c$ and $G_i(z_i)=z_i$. Because $\epsilon
f_\lambda^p(\mathcal{L}_i)\supset \mathcal{L}_i$, we conclude
$z_i=\beta_c$ by Lemma \ref{5c}.
\end{proof}

\begin{cor}\label{7d} If $T(c)$ is periodic for some $c\in C_\lambda$, then $\partial B_\lambda$ is locally connected.
\end{cor}
\begin{proof} We can assume that $f_\lambda$ is not geometrically finite;
 otherwise, the Julia set is locally connected (see \cite{TY}).
 Thus, $f_\lambda$ has no parabolic point.

 If $K_c\cap \partial B_\lambda=\emptyset$, then for all $j\geq0$, $f_\lambda^j(K_c)\cap \partial
B_\lambda=\emptyset$. Because $P(f_\lambda)$ is a subset of
$(\bigcup_{j\geq0}f_\lambda^j(\pm
f_\lambda(K_c)))\bigcup\{\infty\}$, we conclude $P(f_\lambda)\cap
\partial B_\lambda=\emptyset$. Based on Proposition \ref{6a} and Proposition \ref{6c}, $\partial B_\lambda$ is locally connected.

If $K_c\cap \partial B_\lambda\neq\emptyset$, then by Proposition
\ref{7c}, the closed curve
$\mathcal{L}=\mathcal{L}_1\cup\mathcal{L}_2\cup \{\beta_c\}$
separates $K_c\setminus\{\beta_c\}$ from $\partial
B_\lambda\setminus\{\beta_c\}$. In this case, for all $j\geq0$,
$f_\lambda^j(K_c)\cap
\partial B_\lambda=\{f_\lambda^j(\beta_c)\}$. Thus, $\#(P(f_\lambda)\cap
\partial B_\lambda)<\infty$, and all periodic points in $P(f_\lambda)\cap
\partial B_\lambda$ are repelling.
Again by Proposition \ref{6a} and Proposition \ref{6c}, $\partial
B_\lambda$ is locally connected.
\end{proof}

\subsection{Real case}

In this section, we will deal with real parameters. Due to the
symmetry of the parameter plane, we only need to consider
$\lambda\in\mathbb{R}^+=(0,+\infty)$. In this case, the Julia set
$J(f_\lambda)$ is symmetric about the real axis. If
$C_\lambda\subset A_\lambda$, by `The Escape Trichotomy' (Theorem
\ref{1c}), the Julia set $J(f_\lambda)$ is either a Cantor set, a
Cantor set of circles or a Sierpinski curve. In the latter two
cases, the local connectivity of $\partial B_\lambda$ is already
known. In the following discussion, we assume $C_\lambda\cap
A_\lambda=\emptyset$.

\begin{lem}\label{7e} Suppose $\lambda\in\mathbb{R}^+$ and $C_\lambda\cap
A_\lambda=\emptyset$; then, $f_\lambda$ is $1$-renormalizable at
$c_0=\sqrt[2n]{\lambda}$.
\end{lem}
\begin{proof} Let $U$ be the interior of $(S_0\cup S_{-(n-1)})
\setminus\{z\in B_\lambda\cup T_\lambda; G_\lambda(z)\geq 1\}$ and
$V=\mathbb{\bar{C}}\setminus (\{z\in B_\lambda; G_\lambda(z)\geq
n\}\cup [-\infty,v_\lambda^-])$. One can easily verify that
$f_\lambda: U\rightarrow V$ is a quadratic-like map. Because
$C_\lambda\cap A_\lambda=\emptyset$, the critical orbit
$\{f_\lambda^k(c_0); k\geq0\}$ is contained in $U\cap \mathbb{R}^+$.
This implies that $(f_\lambda,U,V)$ is a $1$-renormalization of
$f_\lambda$ at $c_0$.
\end{proof}

Let $K_{c_0}=\bigcap_{k\geq0}f_\lambda^{-k}(U)$ be the small filled
Julia set of the renormalization $(f_\lambda,U,V)$, $\beta_{c_0}$ be
the $\beta-$fixed point and $\beta'_{c_0}$ be the preimage of
$\beta_{c_0}$. It is easy to check that $K_{c_0}$ is symmetric about
the real axis and $K_{c_0}\cap \mathbb{R}^+$ is a connected and
closed interval.

\begin{pro}\label{7f}  $K_{c_0}\cap \partial
B_\lambda=\{\beta_{c_0}\}$.
\end{pro}
\begin{proof}
As with the proof of Proposition \ref{7c}, the idea of the proof is
to construct a Jordan curve $\mathcal{C}$ that separates
$K_{c_0}\setminus \{\beta_{c_0}\}$ from $\partial
B_\lambda\setminus\{\beta_{c_0}\}$,.

We first show that $\beta_{c_0}$ is the landing point of the zero
external ray $R_\lambda(0)$. Note that rational external rays (i.e.,
external rays with a rational angle) always land. Let $z_0$ be the
landing point of $R_\lambda(0)$. Obviously, $R_\lambda(0)\subset
\mathbb{R}^+$ and $z_0$ is a fixed point of $f_\lambda$, which
implies that $z_0\in U\cap \mathbb{R}^+$, and the orbit of $z_0$
does not escape from $U$. Therefore, $z_0\in K_{c_0}$. Because
$R_\lambda(0)$ is an $f_\lambda$-invariant ray that lands at $z_0$,
we conclude $z_0=\beta_{c_0}$ based on Lemma \ref{5c}.

Let $K=K_{c_0}\cup [\beta_{c_0},+\infty]\cup (-K_{c_0})\cup
[-\infty,-\beta_{c_0}]$. One can easily verify
$f^{-1}_\lambda(K)=\bigcup_{\omega^{2n}=1}\omega(K_{c_0}\cup
[0,+\infty])$. The set $Y=\mathbb{\bar{C}}\setminus K$ is a disk,
and its preimage $f^{-1}_\lambda(Y)$ consists of $2n$ components
that are symmetric about the origin under the rotation $z\mapsto
e^{i\pi/n}z$. For each component $X$ of $f^{-1}_\lambda(Y)$,
$f_\lambda:X\rightarrow Y$ is a conformal map. Let $X_0$ be the
component of $f^{-1}_\lambda(Y)$ that is contained in $S_0$ and $g$
be the inverse map of $f_\lambda:X_0\rightarrow Y$. Based on the
Denjoy-Wolff theorem, the sequence of maps $\{g^k;k\geq0\}$
converges locally and uniformly in $Y$ to a constant, say $x$.
Because $g(Y\cap V)\subset X_0\cap U$, we conclude $x\in K_{c_0}$.

Let $\Delta$ be the component of $\mathbb{\bar{C}}\setminus
(\overline{B}_\lambda\cup K_{c_0}\cup(-K_{c_0}) \cup\mathbb{R})$
that intersects with $T_\lambda$ and lies in the upper half plane.

\textbf{Claim:\ }{\it  There is a path $\mathcal{L}\subset
\Delta\cup\{0\}$ stemming from $T_\lambda$ and converging to
$\beta_{c_0}$. More precisely, $\mathcal{L}$ can be parameterized as
$\mathcal{L}: [0,+\infty)\rightarrow \Delta\cup \{0\}$ such that
$\mathcal{L}(0)=0,
 \mathcal{L}((0,+\infty))\subset\Delta$ and
  $\lim_{t\rightarrow +\infty}\mathcal{L}(t)=\beta_{c_0}$}

  Let $p_0=\sqrt[2n]{-\lambda}$ be the preimage of $0$ that lies in
  $S_0$ and $\gamma_0=[0, p_0]$ be the segment connecting $0$ with
  $p_0$. Then, $\gamma_0\cap (K_{c_0}\cup \partial B_\lambda)=\emptyset$.
  Indeed, $\gamma_0 \cap K_{c_0}=\emptyset$ follows from the fact that
  $f_\lambda(\gamma_0)\cap K_{c_0}\subset i\mathbb{R}\cap K_{c_0}=\emptyset$. In the following, we show that $\gamma_0 \cap \partial B_\lambda=\emptyset$. It suffices to show that $B_\lambda\cap D=\emptyset$, where $D=\{z\in \mathbb{\mathbb{C}};
  |z|<\sqrt[2n]{\lambda}\}$. Otherwise, $B_\lambda\cap
  D\neq\emptyset$ implies $B_\lambda\cap
  \partial D\neq\emptyset$. Because $\varphi:
  z\mapsto\sqrt[n]{\lambda}/\bar{z}$ maps $B_\lambda$ onto
  $T_\lambda$ and the restriction  $\varphi |_{\partial D}$ is the
  identity map, we have $B_\lambda\cap
  \partial D=\varphi(B_\lambda\cap
  \partial D)=T_\lambda \cap
  \partial D$. But this implies $B_\lambda\cap T_\lambda\neq\emptyset$, contradiction.

  Note that $g$ maps $\gamma_0$ outside $D$ and $g(\gamma_0)$
  connects $p_0$ with a preimage of $p_0$ that lies inside $S_0$.
  Let $\mathcal{L}=\bigcup_{k\geq0}g^k(\gamma_0)$. By construction, $\mathcal{L}\cap (K_{c_0}\cup
  \partial B_\lambda)=\emptyset$, and $\mathcal{L}$ converges to $x\in
   K_{c_0}$. Because
   $f_\lambda(\mathcal{L})=\mathcal{L}\cup f_\lambda(\gamma_0)\supset\mathcal{L}$, we conclude
   $x=\beta_{c_0}$ based on Lemma \ref{5c}.

   Let
   $\mathcal{C}=\mathcal{L}\cup\mathcal{L}^*\cup \{\beta_{c_0}\}$,
   where $\mathcal{L}^*=\{\bar{z}; z\in \mathcal{L}\}$.
   $\mathcal{C}$ is a Jordan curve separating $K_{c_0}\setminus \{\beta_{c_0}\}$ from $\partial B_\lambda\setminus\{\beta_{c_0}\}$. The conclusion follows.
\end{proof}

\begin{rem}\label{7f1}
Based on the proof of Proposition \ref{7f}, we conclude $$
\partial B_\lambda\cap \mathbb{R}=\{\pm\beta_{c_0}\},\
K_{c_0}\cap \mathbb{R}=[\beta'_{c_0},\beta_{c_0}],\  \partial
T_\lambda\cap \mathbb{R}=\{\pm\beta'_{c_0}\}.$$
\end{rem}

\begin{cor}\label{7g}  Suppose $\lambda\in\mathbb{R}^+$ and $C_\lambda\cap
A_\lambda=\emptyset$; then, $\partial B_\lambda$ is locally
connected.
\end{cor}
\begin{proof}
By Proposition \ref{7f}, if $n$ is odd, then $P(f_\lambda)\cap
\partial B_\lambda\subset (-K_{c_0}\cup K_{c_0})\cap\partial B_\lambda\subset\{\pm
\beta_{c_0}\}$; if $n$ is even, then $P(f_\lambda)\cap
\partial B_\lambda\subset K_{c_0}\cap\partial B_\lambda\subset\{
\beta_{c_0}\}$. If $\beta_{c_0}$ is a parabolic point, then
$f_\lambda$ is geometrically finite, and the local connectivity of
$\partial B_\lambda$ follows from \cite{TY}. Otherwise, based on
Propositions \ref{6a} and \ref{6c}, $\partial B_\lambda$ is also
locally connected.
\end{proof}

\subsection{Local connectivity implies higher regularity}

At this point, we have already proven that $\partial B_\lambda$ is
locally connected if the Julia set is not a Cantor set.  Based on
the arguments of Devaney \cite{D0}, we prove the following
proposition, which will lead to Theorem \ref{11a}.

\begin{pro}\label{7h} If $\partial B_\lambda$ is locally connected, then $\partial B_\lambda$ is a Jordan curve.
\end{pro}
\begin{proof} Let
$W_0$ be the component of
$\mathbb{\overline{C}}-\overline{B}_\lambda$ containing $0$. It is
obvious that $\partial W_0\subset \partial B_\lambda,\
T_\lambda\subset W_0,\ \partial T_\lambda\subset \overline{W}_0$.
Based on Lemma \ref{1a}, $e^{i\pi/n} W_0=W_0$.

Recall that $H_\lambda(z)=\sqrt[n]{\lambda}/z$, so
$H_\lambda(\partial W_0)\subset H_\lambda(\partial
B_\lambda)=\partial T_\lambda \subset \overline{W}_0$. Because
$\partial B_\lambda$ is locally connected, $\partial W_0$ is locally
connected. It follows that $\mathbb{\overline{C}}-\overline W_0$ is
connected and
$H_\lambda(\mathbb{\overline{C}}-\overline{W}_0)\subset {W}_0$.

Now, we show that $f_\lambda^{-1}(0)\subset W_0$. If not,
$f_\lambda^{-1}(0)\cap
(\mathbb{\overline{C}}-\overline{W}_0)\neq\emptyset$. Based on the
symmetry of $f_\lambda^{-1}(0)$ and
$\mathbb{\overline{C}}-\overline{W}_0$, we have
$f_\lambda^{-1}(0)\subset \mathbb{\overline{C}}-\overline{W}_0$.
This will contradict the fact that
$f_\lambda^{-1}(0)=H_\lambda(f_\lambda^{-1}(0))\subset
H_\lambda(\mathbb{\overline{C}}-\overline{W}_0)\subset W_0$.

Because no point on $\partial W_0$ can be mapped into $W_0$, we have
$f_\lambda^{-1}(W_0)\subset W_0$ and $f_\lambda^{-1}
(\overline{W}_0)\subset \overline{W}_0$. Take a point $z\in \partial
W_0$; we have $\partial B_\lambda\subset
J(f_\lambda)=\overline{\bigcup_{k\geq0}f_\lambda^{-k}(z)}\subset
\overline{W}_0$ and $\partial B_\lambda\subset \partial W_0$.
Therefore, $\partial W_0=\partial B_\lambda$.

Now, we show that $\partial B_\lambda$ is a Jordan curve. If two
different external rays, say $R_\lambda(t_1)$ and $R_\lambda(t_2)$,
land at the same point $p\in \partial B_\lambda$, then
$\overline{R_\lambda(t_1)\cup R_\lambda(t_2)}$ decomposes $\partial
B_\lambda$ into two parts. It turns out that $\partial
W_0\neq\partial B_\lambda$, which is a contradiction.
\end{proof}

The aim of this section is to prove Theorem \ref{11c}, as follows:

\noindent\textit{Proof of Theorem \ref{11b}.} By Theorem \ref{11a}
and Proposition \ref{6a}, it suffices to show that $f_\lambda$
satisfies the {\bf BD} condition on $\partial B_\lambda$. First, we
deal with three special cases:

{\bf Case 1.} The critical orbit escapes to infinity.

{\bf Case 2.} The parameter $\lambda\in \mathbb{R}^+$ and $\partial
B_\lambda$ contains no parabolic point.

{\bf Case 3.} The map $f_\lambda$ is critically finite.

In case 1, $P(f_\lambda)\cap \partial B_\lambda=\emptyset$. Based on
Proposition \ref{6c}, $f_\lambda$ satisfies the {\bf BD} condition
on $\partial B_\lambda$. For case 2, by Proposition \ref{7f}, either
$P(f_\lambda)\cap
\partial B_\lambda=\emptyset$ or $P(f_\lambda)\cap \partial
B_\lambda=\{\beta_c\}$ or $P(f_\lambda)\cap \partial
B_\lambda=\{\pm\beta_c\}$. In either case, $\beta_c$ is a repelling
fixed point of $f_\lambda$. By Proposition \ref{6a}, $f_\lambda$
satisfies the {\bf BD} condition on $\partial B_\lambda$. For case
3, $f_\lambda$ satisfies the {\bf BD} condition on $\partial
B_\lambda$ by Corollary \ref{6c}.

In the remaining cases, we can use the Yoccoz puzzle to study the
higher regularity of $\partial B_\lambda$. There are two remaining
cases:

{\bf Case 4.} $\partial B_\lambda$ contains no critical point.

{\bf Case 5.} $C_\lambda\subset\partial B_\lambda$ and  $C_\lambda$ is not recurrent.

In either case, by Proposition \ref{4ab}, we can find an admissible
graph $\mathbf{G}_\lambda(\theta_1,\cdots,\theta_N)$. With respect
to the Yoccoz puzzle induced by this graph, we consider the critical
tableaux. For case 4, there are two possibilities:

{\bf Case 4.1.} Some $T(c)$ with $c\in C_\lambda$ is periodic.

{\bf Case 4.2.} No $T(c)$ with $c\in C_\lambda$ is periodic.

For case 4.1, we conclude from Proposition \ref {7c} that $\#
(P(f_\lambda)\cap \partial B_\lambda)<\infty$. Because $\partial
B_\lambda$ contains no parabolic point, all periodic points in
$P(f_\lambda)\cap\partial B_\lambda$ are repelling. Thus, based on
Proposition \ref{6c}, $f_\lambda$ satisfies the {\bf BD} condition
on $\partial B_\lambda$.

For case 4.2, we have already shown that ${\rm
End}(c)=\bigcap_{d\geq0}\overline{P_d(c)}=\{c\}$ for $c\in
C_\lambda$ in the proof of Proposition \ref {7b}. Thus, we can
choose a $d_0$ large enough such that
$${\rm Eucl.diam}(P_{d_0}(c))<{\rm Eucl.dist}(c, \partial B_\lambda).$$
For $d\geq d_0$, let $U_d$ be the union of all puzzle pieces of
depth $d$ that intersect with $\partial B_\lambda$ and $V_d$ be the
interior of $\overline{U_d}$. For every $u\in \partial B_\lambda$,
there is a number $\varepsilon_u>0$ such that
$B(u,\varepsilon_u)\subset V_{d_0}$. For any $m\geq0$ and any
component $U_m(u)$ of $f_\lambda^{-m}(B(u,\varepsilon_u))$
intersecting with $\partial B_\lambda$, $U_m(u)\subset
V_{d_0+m}\subset V_{d_0}$. By the choice of $d_0$, the sequence
$U_m(u)\rightarrow\cdots\rightarrow
f_\lambda^{m-1}(U_m(u))\rightarrow B(u,\varepsilon_u)$ meets no
critical point of $f_\lambda$; thus, $f_\lambda^m: U_m(u)\rightarrow
B(u,\varepsilon_u)$ is a conformal map. Therefore, in this case,
$f_\lambda$ satisfies the {\bf BD} condition on $\partial
B_\lambda$.

In the following, we deal with case 5. Again,based on Proposition
\ref {7b}, ${\rm End}(c)=\{c\}$ for $c\in C_\lambda$. Thus, in this
case one can verify that  $C_\lambda$ is not recurrent if and only if all tableaux $T(c)$ with $c\in
C_\lambda$ are non-critical. Based on Lemma \ref{5a}, $f_\lambda$ is
critically finite. It follows from Corollary \ref{6d} that
$f_\lambda$ satisfies the {\bf BD} condition on $\partial
B_\lambda$. \hfill $\Box$

\subsection{Corollaries}

In this section, we present some corollaries of Theorem \ref{11a}.

\begin{pro}\label{7i}  If $\partial B_\lambda$ contains a parabolic cycle, then the multiplier of the cycle is $1$ and the Julia set $J(f_\lambda)$
contains a quasi-conformal copy of the quadratic Julia set of
$z\mapsto z^2+1/4.$
\end{pro}
\begin{proof}
Suppose $\mathcal{C}=\{z_0, f_\lambda(z_0), \cdots,
f^{q}_\lambda(z_0)=z_0\}$ is a parabolic cycle on $\partial
B_\lambda$. We will first consider the case $\lambda\in
\mathbb{R}^+$, then deal with the case $\lambda\in \mathcal{H}$.

First, suppose $\lambda\in \mathbb{R}^+$. By Lemma \ref {7e} and
Proposition \ref {7f}, $f_\lambda$ is $1-$renormalizable at $c_0$
and $P(f_\lambda)\cap
\partial B_\lambda\subset (-K_{c_0}\cup K_{c_0})\cap
\partial B_\lambda=\{\pm\beta_{c_0}\}$. Because a parabolic point must
attract a critical point, we conclude that $\beta_{c_0}$ is a
parabolic fixed point of $f_\lambda$. Therefore, $(f_\lambda,U,V)$
is quasi-conformally conjugate to a quadratic polynomial $z\mapsto
z^2+\mu$ with a $\beta-$fixed point that is also a parabolic point,
thus $\mu=1/4$. The conclusion follows in this case.

In the following, we deal with the case $\lambda\in \mathcal{H}$.
Based on Proposition \ref{4ab}, we can find an admissible graph
$\mathbf{G}_\lambda({\theta}_1,\cdots,{\theta}_N)$. Based on
Proposition \ref{3f}, the parabolic cycle $\mathcal{C}$ avoids the
graph $\mathbf{G}_\lambda({\theta}_1,\cdots,{\theta}_N)$. With
respect to the Yoccoz puzzle induced by this graph and with an
argument similar to that used to prove Corollary \ref{5b}, we
conclude that there is a critical point $c\in C_\lambda$ and a point
$z\in \mathcal{C}$ such that $P_d(z)=P_d(c)$ for all $d\geq0$. Thus,
the tableau $T(c)$ is periodic. Suppose the period of $T(c)$ is $k$.
It is obvious that $k$ is a divisor of $q$. By Lemma \ref{5a}, when
$d_0$ is large enough, the triple $(\epsilon f_\lambda^p,
{P}_{d_0+p}(c), {P}_{d_0}(c))$ is either a $k$-renormalization of
$f_\lambda$ at $c$ (in this case, $(\epsilon,p)=(1,k)$) or a
$k/2$-$*$-renormalization of $f_\lambda$ at $c$ (in this case,
$(\epsilon,p)=(-1,k/2)$). Moreover, the small filled Julia set
$K_c={\rm End}(c)= \bigcap_{d\geq 0}\overline{{P}_d (c)}$ and $z\in
K_c\cap\partial B_\lambda$.

On the other hand, based on  Lemma \ref{5d}, there is a unique
external ray $R_\lambda(t)$ landing at $\beta_c$, which is the
$\beta$-fixed point of the renormalization $(\epsilon f_\lambda^p,
{P}_{d_0+p}(c), {P}_{d_0}(c))$. Note that we have already proved
that $\partial B_\lambda$ is a Jordan curve; the intersection
$\partial B_\lambda \cap \overline{{P}_d (c)}$ shrinks to a single
point as $d\rightarrow\infty$. Thus, we have $K_c\cap \partial
B_\lambda=\{\beta_c\}$. By the previous argument, $\beta_c=z$.

Based on the straightening theorem of Douady and Hubbard, $(\epsilon
f_\lambda^p,  {P}_{d_0+p}(c),  {P}_{d_0}(c))$ is quasi-conformally
conjugate to a quadratic polynomial $p_\mu(z)=z^2+\mu$ in a
neighborhood of the small filled Julia set $K_c$. For this quadratic
polynomial, the $\beta$-fixed point is also a parabolic point, thus
$\mu=1/4$. Therefore, the Julia set $J(f_\lambda)$ contains a
quasi-conformal copy of the quadratic Julia set of $z\mapsto
z^2+1/4$. Because the multiplier of the parabolic point of $z\mapsto
z^2+1/4$ is $1$, it turns out that $(\epsilon f_\lambda^p)'(z)=1$,
$(f_\lambda^k)'(z)=1$ and $(f_\lambda^q)'(z)=1$.
\end{proof}

\begin{pro}\label{7j}  Suppose $f_\lambda$ has no Siegal disk and
the Julia set $J(f_\lambda)$ is connected, then every Fatou
component is a Jordan domain.
\end{pro}
\begin{proof} By Proposition  \ref{7h} and the fact that $H_\lambda(B_\lambda)=T_\lambda$,
we conclude that both $T_\lambda$ and $B_\lambda$ are Jordan
domains.

If the critical orbit tends to $\infty$, then the Julia set is a
Sierpinski curve that is locally connected, and all Fatou components
are quasi-disks (by Proposition \ref{6a}).

If the critical orbit remains bounded, then for any $U\in
\mathcal{P}\setminus \{T_\lambda, B_\lambda\}$, there is a smallest
integer $k\geq1$ such that $f_\lambda^k: U\rightarrow T_\lambda$ is
a conformal map. Thus, if two radial rays $R_U(\theta_1)$ and
$R_U(\theta_2)$ land at the same point, then
$R_{T_\lambda}(\theta_1)=f_\lambda^k(R_U(\theta_1))$ and
$R_{T_\lambda}(\theta_2)=f_\lambda^k(R_U(\theta_2))$ also land at
the same point. This implies that $U$ is also a Jordan domain. If
there are other Fatou components, then they are eventually mapped to
a parabolic basin or an attracting basin. By Proposition \ref{5a},
the map is either renormalizable or $*-$renormalizable. It is known
that every bounded Fatou component of a quadratic polynomialwithout
a Siegal disk is a Jordan disk; it turns out that all Fatou
components of $f_\lambda$ are Jordan disks in this case.
\end{proof}

\begin{pro}\label{7k}  If $f_\lambda$ has a Cremer point, then
the Cremer point cannot lie on the boundary of any Fatou component.
In other words, all Cremer points are buried on the Julia set.
\end{pro}
\begin{proof} Suppose $f_\lambda$ has a Cremer point $z$, then
the Fatou set
$F(f_\lambda)=\bigcup_{k\geq0}f_\lambda^{-k}(B_\lambda)$. If $z$
lies on the boundary of some Fatou component, then after iterations,
one sees that $z\in\partial B_\lambda$. By Theorem 1, there is a
periodic external ray $R_\lambda(t)$ landing at $z$. But this is a
contradiction because, by the Snail Lemma, every periodic external
ray can only land at a parabolic point or a repelling point (see
\cite{M1}).
\end{proof}

\section{Local connectivity of the Julia set $J(f_\lambda)$}

In this section, we study the local connectivity of the Julia set
$J(f_\lambda)$. We will prove Theorem \ref{11c}.

 The proof is based on the `Characterization of Local Connectivity'
(Proposition \ref{8a}(see \cite{W})and the `Shrinking Lemma'
(Proposition \ref{8b} (see \cite{TY} or \cite{LM}), as follows.

\begin{pro}\label{8a}  A connected and compact set $X\subset
\mathbb{\overline{C}}$ is locally connected if and only if it
satisfies the following conditions:

1. Every component of $\mathbb{\overline{C}}\setminus X$ is locally
connected.

2. For any $\epsilon>0$, there are only a finite number of
components of $\mathbb{\overline{C}}\setminus X$ with spherical
diameter greater than $\epsilon$.
\end{pro}

\begin{pro}\label{8b} Let $f:\mathbb{\overline{C}}\rightarrow
\mathbb{\overline{C}}$ be a rational map and $D$ be a topological
disk whose closure $\overline{D}$ has no intersection with the
post-critical set $P(f)$. Then, either $\overline{D}$ is contained
in a Siegel disk or a Herman ring or for any $\epsilon>0$ there are
at most finitely many iterated preimages of $D$ with spherical
diameter greater than $\epsilon$.
\end{pro}

\noindent\textit{Proof of Theorem \ref{11c}.}

 1. If $f_\lambda$ is geometrically finite, then $J(f_\lambda)$ is locally connected (See \cite{TY}). Otherwise, the Fatou set $F(f_\lambda)=\bigcup_{k\geq0}f_\lambda^{-k}(B_\lambda)$. Because $\overline{B}_\lambda\cap P(f_\lambda)=\emptyset$, we conclude base on Shrinking Lemma that for any $\epsilon>0$, there are at most finitely many iterated preimages of $B_\lambda$ with spherical diameter greater than $\epsilon$. Based on Proposition \ref {8a}, $J(f_\lambda)$ is locally connected.

2. If $f_\lambda$ is neither renormalizable nor $*-$renormalizable,
then the parameter $\lambda\in\mathcal{H}$ by Lemma \ref{7e}. We can
assume that $f_\lambda$ is not critically finite; otherwise, the
Julia set is locally connected. Thus, based on Proposition
\ref{4ab}, we can find an admissible graph. By Lemma \ref{5a}, none
of the tableaux $T(c)$ with $c\in C_\lambda$ are periodic. The local
connectivity of $J(f_\lambda)$ follows from Proposition \ref {7b}.

3. (The notations here are the same as in Section 7.3) We need only
consider the case when $f_\lambda$ is not geometrically finite. In
this case, the Fatou set
$F(f_\lambda)=\bigcup_{k\geq0}f_\lambda^{-k}(B_\lambda)$. Note that
for any $z>0$, $f_\lambda(z)\geq 2 \sqrt{z^n\cdot
\frac{\lambda}{z^n}}=2\sqrt{\lambda}=v^+_\lambda$. Thus,
$\{f_\lambda^k(v^+_\lambda); k\geq0\}\subset [v^+_\lambda,
\beta_{c_0}]$.

If $v^+_\lambda=\beta'_{c_0}$, one can easily verify that the triple
$(f_\lambda, U,V)$ is quasi-conformally conjugate to the quadratic
polynomial $z\mapsto z^2-2$, which is critically finite. Therefore,
$f_\lambda$ is also critically finite, and the Julia set is locally
connected.

If $v^+_\lambda>\beta'_{c_0}$, then $\overline{T}_\lambda\cap
[v_\lambda^+, \beta_{c_0}]=\emptyset$ by Remark \ref{7f1}. Because
$P(f_\lambda)\subset [-\beta_{c_0},v^-_\lambda]\cup
[v^+_\lambda,\beta_{c_0}]\cup \{\infty\}$, we have
$\overline{T}_\lambda\cap P(f_\lambda)=\emptyset$. Based on
Proposition \ref{8b}, for any $\epsilon>0$ there are at most
finitely many iterated preimages of $T_\lambda$ with spherical
diameter greater than $\epsilon$. Based on Proposition \ref{8a}, the
Julia set is locally connected. \hfill $\Box$











\end{document}